\documentclass[a4paper,11pt]{article}
\usepackage[utf8]{inputenc}
\usepackage[english]{babel}
\usepackage[a4paper,twoside,top=1in, bottom=1.1in, left=0.7in, right=0.7in]{geometry} 
\usepackage{graphicx}
\usepackage{hyperref}
\usepackage{amsmath}
\usepackage{amsthm}
\usepackage{amssymb}
\usepackage{esint} 
\usepackage{mathrsfs} 
\usepackage{xcolor}
\usepackage{array}
\usepackage{hhline}
\usepackage{enumitem} 
\usepackage{imakeidx} 
\usepackage{xparse} 
\usepackage{mathtools}
\usepackage{float}

\usepackage{comment}

\definecolor{newblue}{rgb}{0.2, 0.3, 0.85}
\hypersetup{colorlinks=true, linkcolor=newblue, citecolor=newblue, urlcolor  = newblue} 

\usepackage{fancyhdr} 
\usepackage{ifthen} 
\usepackage{forloop} 
\usepackage{xstring} 
\usepackage{blindtext} 
\usepackage{nomencl}
\usepackage{dsfont} 
\usepackage{faktor} 
\usepackage{tikz}
\usetikzlibrary{arrows}
\usetikzlibrary{patterns}

\numberwithin{equation}{section}

\usepackage[english,capitalize]{cleveref}

\usepackage[maxbibnames=99,backend=biber, doi=false, url=false]{biblatex}
\usepackage[colorinlistoftodos,textwidth=2.3cm]{todonotes}

\definecolor{dgreen}{rgb}{0.0, 0.56, 0.0}

\newcommand{\N}{\ensuremath{\mathbb N}}
\newcommand{\Q}{\ensuremath{\mathbb Q}}
\newcommand{\R}{\ensuremath{\mathbb R}}
\newcommand{\meas}{\mathfrak{m}}
\newcommand{\lip}{{\rm lip \,}}

\newcommand{\res}{\mathbin{\vrule height 1.6ex depth 0pt width
0.13ex\vrule height 0.13ex depth 0pt width 1.3ex}} 

\DeclarePairedDelimiter\abs{\lvert}{\rvert}

\DeclarePairedDelimiter\scal{\langle}{\rangle}

\newcommand{\st}{\ensuremath{\ :\ }} 
\newcommand{\eqdef}{\ensuremath{\vcentcolon=}}
\newcommand \eps{\ensuremath{\varepsilon}} 
\renewcommand{\epsilon}{\varepsilon}

\newcommand{\de}{\ensuremath{\,\mathrm d}} 
\renewcommand{\d}{\ensuremath{\mathrm d}} 
\DeclareMathOperator{\supp}{spt} 

\DeclareMathOperator{\sn}{sn}

\newcommand{\CD}{\mathsf{CD}}
\newcommand{\RCD}{\mathsf{RCD}}

\newcommand{\dist}{\mathsf{d}}

\let\div\undefined
\DeclareMathOperator{\div}{div}

\newcommand{\ric}{\ensuremath{\mathrm{Ric}}} 
\newcommand{\sect}{\ensuremath{\mathrm{Sect}}} 
\DeclareMathOperator{\vol}{vol}
\DeclareMathOperator{\inj}{inj}

\theoremstyle{plain}
\newtheorem{thm}{Theorem}[section] 
\theoremstyle{plain}

\theoremstyle{plain}
\newtheorem{prop}[thm]{Proposition}
\theoremstyle{plain}
\newtheorem{lemma}[thm]{Lemma}
\theoremstyle{plain}
\newtheorem{cor}[thm]{Corollary}
\theoremstyle{definition}
\newtheorem{defn}[thm]{Definition} 
\theoremstyle{definition}
\newtheorem{remark}[thm]{Remark}
\theoremstyle{definition}
\newtheorem{example}[thm]{Example}
\theoremstyle{definition}

\title{The isoperimetric problem on Riemannian manifolds via Gromov--Hausdorff asymptotic analysis}
\author{Gioacchino Antonelli\footnote{\href{mailto:gioacchino.antonelli@sns.it}{gioacchino.antonelli@sns.it}, Scuola Normale Superiore, Piazza dei Cavalieri, 7, 56126 Pisa, Italy.}
\and Mattia Fogagnolo\footnote{\href{mailto:mattia.fogagnolo@sns.it}{mattia.fogagnolo@sns.it}, Centro di Ricerca Matematica Ennio De Giorgi, Scuola Normale Superiore, Piazza dei Cavalieri, 3, 56126 Pisa, Italy.}
\and Marco Pozzetta\footnote{\href{mailto:marco.pozzetta@unina.it}{marco.pozzetta@unina.it}, Dipartimento di Matematica e Applicazioni, Universit\`a di Napoli Federico II, Via Cintia, Monte S. Angelo, 80126 Napoli, Italy.}}

\date{\today}

\begin{document}

\maketitle

\begin{abstract} 
In this paper we prove the existence of isoperimetric regions of any volume in Riemannian manifolds with Ricci bounded below assuming Gromov--Hausdorff asymptoticity to the suitable simply connected model of constant sectional curvature.

The previous result is a consequence of a general structure theorem for perimeter-minimizing sequences of sets of fixed volume on noncollapsed Riemannian manifolds with a lower bound on the Ricci curvature. We show that, without assuming any further hypotheses on the asymptotic geometry, all the mass and the perimeter lost at infinity, if any, are recovered by at most countably many isoperimetric regions sitting in some (possibly nonsmooth) Gromov--Hausdorff limits at infinity.

The Gromov--Hausdorff asymptotic analysis allows us to recover and extend different previous existence theorems. 

While studying the isoperimetric problem in the smooth setting, the nonsmooth geometry naturally emerges, and thus our treatment combines techniques from both the theories.
\end{abstract}

\noindent\textbf{MSC (2020).} Primary: 49J45, 26B30, 53A35. Secondary: 53C23, 49J52. \\
\textbf{Keywords.} Gromov--Hausdorff convergence, isoperimetric problem,  Ricci curvature, RCD spaces, finite perimeter sets.

\tableofcontents

    
    
    


\section{Introduction}

The classical isoperimetric problem can be formulated on every ambient space possessing notions of \emph{volume} and \emph{perimeter} on (some subclass of) its subsets. Among sets having assigned positive volume, the problem deals with finding those having least perimeter. Among the most basic questions in the context of the isoperimetric problem, one would naturally ask whether perimeter-minimizing sets exists, but also what goes wrong in the minimization process if such minimizers do not exist.
We are interested here in the isoperimetric problem set on  smooth Riemannian manifolds and in giving a good description of the minimization process.

We denote by $(M^n,g)$ a Riemannian manifold of dimension $n$ and metric tensor $g$. \emph{We will always assume, unless specified differently, that $n \geq 2$}. The symbols $\dist, \vol, P$ denote the geodesic distance, the volume measure, and the perimeter functional induced by $g$.
In such a framework, the isoperimetric problem consists in the minimization problem
\[
\min \left\{ P(E) \st \vol(E) = V \right\},
\]
for fixed $V \in (0,\vol(M^n))$, where the competitors $E\subset M^n$ are finite perimeter sets on $(M^n,g)$. The infimum of the perimeter $P(E)$ among such competitors of given volume $V$ is called \emph{isoperimetric profile of $M^n$ at $V$} and it is commonly denoted by $I_{(M^n,g)}(V)$. If $\vol(E)=V$ and $P(E)=I_{(M^n,g)}(V)$, hence if $E$ solves the isoperimetric problem for its own volume, we will say that $E$ is an \emph{isoperimetric region} (or an isoperimetric set).

Unless otherwise stated we will also always assume that
\begin{equation}\label{eqn:H}
\text{$(M^n,g)$ \emph{is complete, noncompact, and has infinite volume.}}
\end{equation}
In fact, in case $(M^n,g)$ is compact an easy application of direct methods in Calculus of Variations provides the existence of isoperimetric regions for any volume in $(0,\vol(M^n))$; also, when $(M^n,g)$ is complete noncompact but with finite volume, the existence of isoperimetric regions for any volume is ensured by the application of \cite[Theorem 2.1 and Remark 2.3]{RitRosales04}.


A classical way for studying the isoperimetric problem at a given volume $V>0$ is to argue by means of direct methods in Calculus of Variations. So, for a given Riemannian manifold $(M^n,g)$ and a volume $V>0$, one considers a sequence of finite perimeter sets $\Omega_i\subset M^n$ with $\vol(\Omega_i)=V$ and $P(\Omega_i)\to I_{(M^n,g)}(V)$. It is well-known by the theory of finite perimeter sets that, up to subsequence, $\Omega_i$ converges in $L^1_{\rm loc}$ to a set $\Omega$, the perimeter is lower semicontinuous, but the volume of $\Omega$ might be strictly less than $V$. It is then common to try to understand the consequences of having lost part of the mass of the minimizing sequence at infinity. Indeed, under suitable assumptions on the geometry of $(M^n,g)$, one can try to infer that the potential leak of mass would be inconvenient, thus getting existence results. By the same approach, one can also grasp new information about the isoperimetric profile $I_{(M^n, g)}$.

\smallskip


Since the possible leak of mass of a minimizing sequence is due to the fact that the ambient $M^n$ is not compact, it has been spontaneous in the literature to assume a priori asymptotic assumptions on the manifold $(M^n,g)$. In \cite{Nar14}, Nardulli assumed that  $(M^n,g)$ is noncollapsed, that its Ricci curvature is bounded from below, and that 
for any sequence of points $p_i \in M^n$ there exists a pointed Riemannian manifold $(M^n_\infty,\dist_\infty,p_\infty)$  such that that $(M^n,\dist,p_i) \to (M^n_\infty,\dist_\infty,p_\infty)$ in a suitable pointed $C^{1,\alpha}$-sense. We recall that a Riemannian manifold $(M^n,g)$ is said to be \emph{noncollapsed} if there is $v_0>0$ such that $\vol(B_1(p)) \geq v_0$ for all $p\in M^n$. Here with $B_r(p)$ we denote the open ball of radius $r$ of center $p\in M^n$ according to the distance $\dist$. What he proved in \cite[Theorem 2]{Nar14} is that, under the latter asymptotic condition, a description of the mass lost at infinity in the previous minimization process can be given, more precisely showing that it is recovered by finitely many isoperimetric regions, each of them contained in one of the limit manifolds $(M^n_\infty,\dist_\infty,p_\infty)$.

On the other hand it turns out, see \cref{rem:GromovPrecompactness}, that the class of pointed uniformly noncollapsed manifolds of a given dimension having a uniform lower bound on the Ricci tensor is precompact with respect to pointed measure Gromov--Hausdorff (pmGH for short) convergence, see \cref{def:GHconvergence} for such a notion. Actually, the precompactness property holds at the level of $\RCD$ spaces, which are metric measure spaces with a synthetic lower bound on the Ricci tensor (see \cref{sub:RCD}), with a uniform bound from below on the measure of unit balls. This means that given any sequence of points $p_i$ on a noncollapsed manifold $(M^n,g)$ with Ricci bounded below, the sequence of pointed metric measure spaces $(M^n,\dist,\vol,p_i)$ converges in the pmGH sense to a pointed $\RCD$ space, up to a subsequence. Let us point out that, as a consequence of the celebrated volume convergence theorem in \cite{Colding97, DePhilippisGigli18}, the measure on such a limit is the Hausdorff measure of dimension $n$ with respect to the corresponding distance.
Eventually one may hope for a description analogous to the one mentioned above, coming from \cite{Nar14}, without further assumptions on $(M^n,g)$ but the noncollapsedness and a lower bound on the Ricci tensor, exploiting the pmGH precompactness in order to give a description of the lost mass.

In fact, the first of our main results is the following theorem which precisely states that minimizing sequences $\Omega_i$ of a given volume $V$ split into a ``converging'' part $\Omega_i^c$ and into at most countably many ``diverging'' parts $\Omega_{i,j}^d$ that converge in a suitable sense to isoperimetric regions in pmGH limit $\RCD$ spaces. Moreover, the limits of $\Omega_i^c$ and of each $\Omega_{i,j}^d$ recover the assigned volume $V$ and the isoperimetric profile of $(M^n,g)$ at $V$ (in the sense of \eqref{eq:UguaglianzeIntro} below). All in all, the forthcoming result gives a description of the asymptotic behavior of the diverging mass of minimizing sequences. We stress that the identification of the ``converging'' part $\Omega_i^c$ of a minimizing sequence, which is the starting point of our arguments, is a classical result due to Ritor\'{e}--Rosales \cite[Theorem 2.1]{RitRosales04}.


\begin{thm}[Asymptotic mass decomposition]\label{thm:MassDecompositionINTRO}
Let $(M^n,g)$ be a noncollapsed manifold as in \eqref{eqn:H}, such that $\ric\ge k$ for some $k\in(-\infty, 0]$, and let $V>0$. For every minimizing (for the perimeter) sequence $\Omega_i\subset M^n$ of volume $V$, with $\Omega_i$ bounded for any $i$, up to passing to a subsequence, there exist an increasing sequence $\{N_i\}_{i\in\mathbb N}\subseteq \mathbb N$, disjoint finite perimeter sets $\Omega_i^c, \Omega_{i,j}^d \subset \Omega_i$, and points $p_{i,j}$, with $1\leq j\leq N_i$ for any $i$, such that
\begin{itemize}
    \item $\lim_{i} \dist(p_{i,j},p_{i,\ell}) = \lim_{i} \dist(p_{i,j},o)=+\infty$, for any $j\neq \ell<\overline N+1$ and any $o\in M^n$, where $\overline N:=\lim_i N_i \in \N \cup \{+\infty\}$;
    \item $\Omega_i^c$ converges to $\Omega\subset M^n$ in the sense of finite perimeter sets (\cref{def:ConvergenceFinitePerimeter}), and we have $\vol(\Omega_i^c)\to_i \vol(\Omega)$, and $ P( \Omega_i^c) \to_i P(\Omega)$. Moreover $\Omega$ is a bounded isoperimetric region;
    \item for every $j<\overline N+1$, $(M^n,\dist,\vol,p_{i,j})$ converges in the pmGH sense  to a pointed $\RCD(k,n)$ space $(X_j,\dist_j,\mathcal{H}^n,p_j)$, where $\mathcal{H}^n$ on $X_j$ is the $n$-dimensional Hausdorff measure defined by the distance $\dist_j$. Moreover there are isoperimetric regions $Z_j \subset X_j$ such that $\Omega^d_{i,j}\to_i Z_j$ in $L^1$-strong (\cref{def:L1strong}) and $P(\Omega^d_{i,j}) \to_i P_{X_j}(Z_j)$;
    \item it holds that
    \begin{equation}\label{eq:UguaglianzeIntro}
    I_{(M^n,g)}(V) = P(\Omega) + \sum_{j=1}^{\overline{N}} P_{X_j} (Z_j),
    \qquad\qquad
    V=\vol(\Omega) +  \sum_{j=1}^{\overline{N}} \mathfrak{m}_j(Z_j).
    \end{equation}
\end{itemize}
\end{thm}

Some comments about the above statement are in order. First of all, the fact that the sets of the minimizing sequence are assumed to be bounded does not undermine the generality because sets in the minimizing sequences for the isoperimetric problem can always be taken bounded by the approximation result recalled in \cref{rem:Approximation}. Also, in the above statement, 
the perimeter $P_{X_j}$ is the distributional perimeter on $(X,\dist_j,\mathcal{H}^n)$, see \cref{def:BVperimetro}. Moreover, the convergence in the $L^1$-strong sense in particular implies the convergence of the volumes of the sets, i.e., $\vol(\Omega_{i,j}^d) \to_i \mathcal{H}^n(Z_j)$. 

The above theorem is actually a simplification of a more detailed result, whose technical statement can be found in \cref{thm:MassDecomposition}. The main advantage of that complete formulation is the detailed construction of $\Omega_{i, j}^d$ from $\Omega_i^d\eqdef \Omega\setminus \Omega_i^c$, that is the diverging part of the minimizing sequence.

The \cref{thm:MassDecompositionINTRO} also implies that on noncollapsed manifolds with Ricci bounded below, the isoperimetric profile is strictly positive (see \cref{rem:ProfiloPositivo}).

Both the hypotheses of noncollapsedness and Ricci bounded below in \cref{thm:MassDecompositionINTRO} are necessary in order to guarantee some concentration of mass that eventually yields the nonempty limit sets $Z_j$. This is discussed in \cref{Sec:ExampleVanishing}, where we provide examples both of a collapsed manifold with sectional curvature bounded below and of a noncollapsed manifold with Ricci unbounded below in which minimizing sequences, in fact, avoid any concentration of mass, making impossible to formulate a result as the one in \cref{thm:MassDecompositionINTRO}.


Another class of spaces where a similar asymptotic decomposition of the minimizing sequences has been performed is that of the unbounded convex bodies in Euclidean spaces, treated in \cite{LeonardiRitore}.
It is possible to prove a decomposition result like \cref{thm:MassDecompositionINTRO} also without the assumption of having a minimizing sequence; more precisely, one can prove that an arbitrary sequence of sets with uniformly bounded volume and perimeter splits, up to subsequence, into subsets converging in $L^1_{\rm loc}$ to limit sets sitting either in $M$ or in some GH-limits at infinity. This yields a result of generalized compactness analogous to \cite{FloresNardulliCompactness}.


In view of \cref{thm:MassDecompositionINTRO}, one notices that the more is known about the GH-asymptotic structure of the manifold, the more information one gets about the minimizing sequence, and in turn about the isoperimetric problem. 
In the paper \cite{AntonelliBrueFogagnoloPozzetta2021}, where we focus on the nonnegative Ricci curved case, we apply indeed the above asymptotic decomposition in relation with the geometry of the \emph{asymptotic cones at infinity}. This is an analysis that clearly cannot overlook the generality reached in \cref{thm:MassDecompositionINTRO}. { Exploiting again the direct approach of the asymptotic mass decomposition, further results on the isoperimetry of manifolds with lower Ricci bounds have been recently obtained in \cite{AntonelliPasqualettoPozzettaSemolaFIRSThalf, AntonelliPasqualettoPozzettaSemolaSECONDhalf}.}

Here, we limit ourselves to deduce some existence theorems for isoperimetric regions when some precise structure at infinity is prescribed. 
To this end, we propose the following notion of GH-asymptoticity.

\begin{defn}[GH-asymptoticity]
\label{def:GHasymp}
Let $(M^n,g)$ be a noncompact Riemannian manifold with distance $\dist$ and volume measure $\vol$. We say that $(M^n,g)$ is \emph{Gromov--Hausdorff asymptotic, $\mathrm{GH}$-asymptotic for short, to a metric space}  $(X,\dist_X)$ if for any diverging sequence of points $q_i \in M^n$, i.e., such that $\dist(q,q_i)\to+\infty$ for any $q \in M^n$, there is $x_0\in X$ such that
\[
(M^n,\dist,q_i) \xrightarrow[i\to+\infty]{} (X,\dist_X, x_0),
\]
in the $\mathrm{pGH}$-sense (see \cref{def:GHconvergence}).

We say that $(M^n,g)$ is \emph{measure Gromov--Hausdorff asymptotic, $\mathrm{mGH}$-asymptotic for short, to a metric measure space}  $(X,\dist_X,\mathfrak{m}_X)$ if for any diverging sequence of points $q_i \in M^n$, there is $x_0\in X$ such that
\[
(M^n,\dist,\vol,q_i) \xrightarrow[i\to+\infty]{} (X,\dist_X,\mathfrak{m}_X, x_0),
\]
in the $\mathrm{pmGH}$-sense (see \cref{def:GHconvergence}).
\end{defn}

In the above definition, if $(X,\dist_X,\mathfrak{m}_X)$ is such that for every $x_1,x_2 \in X$ there is an isometry $\varphi:X\to X$ such that $\varphi(x_1)=\varphi(x_2)$ and $\varphi_\sharp\mathfrak{m}_X=\mathfrak{m}_X$, then $(M^n,g)$ is mGH-asymptotic to $(X,\dist_X,\mathfrak{m}_X)$ if for any diverging sequence of points $q_i \in M^n$, it occurs that
$(M^n,\dist,\vol,q_i) \to (X,\dist_X,\mathfrak{m}_X, x)$ for any $x \in X$. Loosely speaking, in such a case it does not matter the point at which the limit space is pointed.

We remark that the simply connected Riemannian manifolds of constant sectional curvature  satisfy the property 
above.

The following \cref{thm:MainRicNcGHIntro}
enables us to provide a full generalization of the existence result by Mondino--Nardulli \cite[Theorem 1.2]{MondinoNardulli16}, where the $C^0$-asymptoticity assumption therein is weakened with a GH-asymptoticity hypothesis here. For the next statement see \cref{thm:MainRicNcGH}.

\begin{thm}\label{thm:MainRicNcGHIntro}
Let $k\in(-\infty,0]$ and let $(M^n,g)$ be as in \eqref{eqn:H} such that $\ric\geq (n-1)k$ on $M\setminus\mathcal{C}$, where $\mathcal{C}$ is compact.

Suppose that $(M^n,g)$ is GH-asymptotic to the simply connected model of constant sectional curvature $k$ and dimension $n$. Then for any $V>0$ there exists an isoperimetric region of volume $V$ on $(M^n,g)$.

\end{thm}
As an example, if $(M^n, g)$ has nonnegative Ricci curvature and its asymptotic volume ratio, defined by
$$
\mathrm{AVR}(M^n,g):=\lim_{r\to+\infty}\frac{\vol(B_r(p))}{\omega_nr^n},
$$
is strictly positive, then it is GH-asymptotic to flat $\R^n$ each time any of its \emph{asymptotic cones at infinity}  has a smooth cross-section. This is proved in details in \cite[Theorem 4.3]{AntonelliBrueFogagnoloPozzetta2021}. 


\smallskip

Let us now quickly describe some other examples that satisfy the hypotheses of \cref{thm:MainRicNcGHIntro}.

We will notice in \cref{prop:C0Asymp} that, as a consequence of standard comparison theorems, a complete Riemannian manifold $(M^n,g)$ for which the injectivity radius diverges to $+\infty$ and the sectional curvature converges to $k\in\mathbb R$ at infinity, is GH-asymptotic (actually even $C^0$-asymptotic, see \cref{rem:GHandC0}) to the simply connected model of constant sectional curvature $k$ and dimension $n$. Hence, if on a manifold satisfying the latter assumption we have a lower bound $\ric\geq (n-1)k$ outside a compact set, \cref{thm:MainRicNcGHIntro} applies and we get the existence of isoperimetric regions for any volume. This is the case, for example, of ALE gravitational instantons (see \cite{Minerbe} and references therein for an account) and of the class of warped products described in \cref{rem:Warped}, which contains, for example, the Bryant type solitons (see \cite[Chapter 4, Section 6]{chow-ricci}  and references therein) and many other explicit solitons as those produced in \cite{CatinoMazzieri}. 
Moreover, the combination with a fundamental estimate on the injectivity radius \cite{CheegerGromovTaylor} enables us to show that nonnegatively Ricci curved manifolds  with asymptotically vanishing sectional curvature (\cref{def:AVSC}) and Euclidean volume growth, that is, $\mathrm{vol}(B_r(p)) \geq C r^n$ for some positive constant $C$ uniform in $p$, possess isoperimetric regions for any volume (see \cref{cor:AVR}). Such class of manifolds is quite rich, as it contains, for example, Perelman's examples of manifolds with non-unique asymptotic cones at infinity, see \cite{PerelmanExampleCones} and \cite[Section 8]{ChCo1}. Also, this class of manifolds naturally encompasses the case of manifolds with nonnegative Ricci curvature that are $C^{2, \alpha}$-asymptotically conical, for which the existence and the description of isoperimetric regions for large volumes were investigated in \cite{ChodoshEichmairVolkmann17}, see \cref{rem:AsymptoticallyEuclideanAndConical}. Since, by \cref{thm:MainRicNcGHIntro}, the Ricci curvature suffices to be nonnegatively defined just outside of a compact set,  compact perturbations of the above described metrics still enjoy existence of isoperimetric sets for any volume.

In analogy with \cite{MondinoNardulli16}, the main tool we are going to employ in addition to \cref{thm:MassDecompositionINTRO} to prove the above existence result in \cref{thm:MainRicNcGHIntro} is a comparison argument, introduced in \cite{MorganJohnson00}, following from the classical Bishop--Gromov monotonicity theorem recalled in \cref{thm:BishopGromov}. The coupling of a suitable asymptotic study of a minimizing sequence with a monotonicity formula,  aiming at excluding the drifting at infinity, seems to be a powerful and general strategy to infer the existence of isoperimetric sets on Riemannian manifolds. {Indeed, a similar idea is employed in the proof of the recent existence result for isoperimetric sets on asymptotically flat Riemannian manifolds with nonnegative scalar curvature, content of \cite[Proposition K.1]{Carlotto2016}}. In fact, such result mostly builds on a way easier asymptotic mass decomposition originated in \cite[Proposition 4.2]{EichmairMetzger} together with an isoperimetric inequality of Shi \cite{Shi} proved through the celebrated Hawking mass monotonicity along the Inverse Mean Curvature Flow \cite{HuiskenIlmanen}.

Apart from the already mentioned contributions, there are many other important results in literature about the existence and description of isoperimetric sets in Riemannian manifolds. Limiting ourselves to the contributions that inspired in some way our investigations, we recall \cite{MorganRitore02, RitRosales04} in which the authors studied the isoperimetric problem in abstract cones and in Euclidean cones respectively, \cite{Pedrosa2004}, where the isoperimetric problem is solved on cylinders, the isoperimetric existence theorem on Riemannian manifolds $(M^n, g)$ with compact quotient $M/\mathrm{Iso}(M^n)$, that has been pointed out by Morgan \cite[Chapter 3]{MorganBook}, building also on \cite{AlmgrenBook}, and the existence result for nonnegatively curved $2$-dimensional surfaces \cite{RitoreExistenceSurfaces01}.
For the existence and description of isoperimetric sets for large volumes, we mention the papers \cite{EichmairMetzger, EichmairMetzger2}, \cite{Chodosh2016}, \cite{ChodoshEichmairVolkmann17} and \cite{Arezzo} where an isoperimetric (for large volumes) foliation has been discovered on asymptotically Schwarzschildian, hyperbolic, conical, and cuspidal manifolds respectively.  
The isoperimetric problem has been and it is currently studied also in the sub-Riemannian { and sub-Finsler setting}: for example, the existence of isoperimetric sets of any volume has been established in Carnot groups \cite{LeonardiRigot}, in sub-Riemannian manifolds whose quotient by the group of contact transformations preserving the sub-Riemannian metric is compact \cite{GalliRitore}, { and in sub-Finsler nilpotent Lie groups \cite{Pozuelo}}. A different framework where this problem has been investigated is also that of $\mathbb{R}^n$ with  densities, see \cite{MorganPratelli, DePhilippis2017}.


\smallskip

We conclude this introduction by pointing out some other results and applications, part of which are technical and needed for proving \cref{thm:MassDecompositionINTRO} and \cref{thm:MainRicNcGHIntro}. Carrying out the asymtptotic analysis on Riemannian manifolds in the context of the Gromov--Hausdorff convergence allows us to derive useful comparison results between the isoperimetric profile of the manifold and the one of any pmGH limit along sequences of diverging points on the manifold. This leads to \cref{prop:ComparisonIsoperimetricProfile}, that essentially estimates from above the isoperimetric profile of a manifold $(M^n,g)$ with the one of any pmGH limit along sequences of points on $M^n$. 
\cref{prop:ComparisonIsoperimetricProfile} implies some interesting consequences on Cartan--Hadamard manifolds. We will prove that the isoperimetric profile of Cartan--Hadamard manifolds with Ricci bounded below and GH-asymptotic to $\R^n$ for $2\le n\le 4$ equals the one of the Euclidean space (\cref{thm:CartanHadamard}). Also, if in addition the sectional curvatures are strictly negative, the rigidity statement of \cref{thm:CartanHadamard} implies the nonexistence of isoperimetric regions, see \cref{ex:NonExistence}. In particular, this shows that a noncollaped manifold with a lower bound on the Ricci curvature may in general fail to enjoy existence of isoperimetric sets, even if the curvature is uniformly bounded.

\medskip

\textbf{Plan of the paper.} In \cref{sec:DefPreliminaryResults} we recall definitions, results and we prove a preliminary lemma (see \cref{lem:ProfileOnBoundedSets}) we will need. In \cref{sec:AsymptGeometry} we investigate the above mentioned relations between pmGH limits of manifolds and the isoperimetric profile of the manifold and the one of such pmGH limits. 
\cref{sec:MassDecomposition} is devoted to the analysis of the asymptotic behavior of the mass of minimizing sequences; here we prove \cref{thm:MassDecompositionINTRO} in its more detailed version, that is \cref{thm:MassDecomposition}, and apply it to deduce \cref{thm:MainRicNcGHIntro}. In \cref{sec:ApplicationsExamples} we discuss the applications and the examples anticipated above.

For the convenience of the reader, in \cref{sec:BishopGromov} we recall two useful well-known comparison results in Riemannian geometry, and in \cref{sec:BoundIsopRegions} we give a self-contained proof of the fact that suitable assumptions on a manifold $(M^n,g)$ imply that isoperimetric regions are bounded.

{
\medskip

\textbf{Addendum.} 
The previous \cref{thm:MassDecompositionINTRO} has been generalized to the setting of $\RCD(k,n)$ spaces endowed with the measure $\mathcal{H}^n$ in the subsequent paper \cite{AntonelliNardulliPozzetta}, where also a finite upper bound on $\overline{N}$ of \cref{thm:MassDecompositionINTRO} is provided. The latter nontrivial extension is based on the arguments developed in the present paper together with the analysis of the topological regularity of isoperimetric sets on $\RCD$ spaces carried out in \cite{AntonelliPasqualettoPozzetta21}.

}

\bigskip

\textbf{Acknowledgments.} The authors would like to thank Lorenzo Mazzieri for a useful conversation about \cref{ex:NonExistence} and Andrea Mondino for discussions related to \cref{sec:ApplicationsExamples}. They are also grateful to Elia Bru\`{e}, Tobias Colding, Nicola Gigli, Gian Paolo Leonardi, Stefano Nardulli, Vincenzo Scattaglia, and Daniele Semola for inspiring conversations about the subject.

G.A. was also partially supported by the European Research Council
(ERC Starting Grant 713998 GeoMeG `\emph{Geometry of Metric Groups}').

\section{Definitions and preliminary results}\label{sec:DefPreliminaryResults}


For the notions of BV and Sobolev spaces on Riemannian manifolds we refer the reader to \cite[Section 1]{MirandaPallaraParonettoPreunkert07}.
For every finite perimeter set $E$ in $\Omega$ we denote with $P(E,\Omega)$ the perimeter of $E$ inside $\Omega$. When $\Omega=M^n$ we simply write $P(E)$. We denote with $\mathcal{H}^{n-1}$ the $(n-1)$-dimensional Hausdorff measure on $M^n$ relative to the distance induced by $g$. We recall that for every finite perimeter set $E$ one has $P(E)=\mathcal{H}^{n-1}(\partial^*E)$ and the characteristic function $\chi_E$ belongs to $BV_{\rm loc}(M^n, \vol)$ with generalized gradient $D\chi_E = \nu \mathcal{H}^{n-1}\res \partial^* E$ for a function $\nu:M\to T M^n$ with $|\nu|=1$ at $|D\chi_E|$-a.e. point, where $\partial^* E$ is the essential boundary of $E$.

We recall the following terminology.

\begin{defn}[Convergence of finite perimeter sets]\label{def:ConvergenceFinitePerimeter}
Let $(M^n,g)$ be a Riemannian manifold. We say that a sequence of measurable (with respect to the volume measure) sets $E_i$ \emph{locally converges} to a measurable set $E$ if the characteristic functions $\chi_{E_i}$ converge to $\chi_E$ in $L^1_{\rm loc}(M^n,g)$. In such a case we simply write that $E_i\to E$ locally on $M^n$.

If the sets $E_i$ have also locally finite perimeter, that is, $P(E_i,\Omega)<+\infty$ for any $k$ and any bounded open set $\Omega$, we say that $E_i\to E$ \emph{in the sense of finite perimeter sets} if $E_i\to E$ locally on $M^n$ and the sequence of measures $D\chi_{E_i}$ locally weakly* converges as measures, that is, with respect to the duality with compactly supported continuous functions. In such a case, $E$ has locally finite perimeter and the weak* limit of $D\chi_{E_i}$ is $D\chi_E$.
\end{defn}

\begin{defn}[Isoperimetric profile]
Let $(M^n,g)$ be a Riemannian manifold. We define the isoperimetric profile function $I:[0,\vol(M^n))\to[0,+\infty)$ as follows
$$
I_{(M^n,g)}(V):=\inf\{P(\Omega):\text{$\Omega$ is a finite perimeter set in $M^n$ such that $\vol(\Omega)=V$}\}.
$$
We also occasionally write $I(V)$ when the ambient manifold $M^n$ is understood.
\end{defn}

\begin{remark}[Approximation of finite perimeter sets with smooth sets]\label{rem:Approximation}
It can be proved, see \cite[Lemma 2.3]{FloresNardulli20}, that when $M^n$ is a complete Riemannian manifold every finite perimeter set $\Omega$ with $0<\vol(\Omega)<+\infty$ and $\vol(\Omega^c)>0$ is approximated by relatively compact sets $\Omega_i$ in $M^n$ with smooth boundary such that $\vol(\Omega_i)=\vol(\Omega)$ for every $i\in\mathbb N$, $\vol(\Omega_i\Delta\Omega)\to 0$ when $i\to +\infty$, and $P(\Omega_i)\to P(\Omega)$ when $i\to +\infty$. Thus, by approximation, one can deduce that 
$$
I(V)=\inf\{\mathcal{H}^{n-1}(\partial\Omega):\text{$\Omega \Subset M^n$ has smooth boundary, $\vol(\Omega)=V$}\},
$$
see \cite[Theorem 1.1]{FloresNardulli20}.
\end{remark}

\begin{defn}[Isoperimetric region]\label{def:IsoperimetricRegion}
Given a Riemannian manifold $(M^n,g)$ the set $E$ is an {\em isoperimetric region} in $M^n$ if $0<\vol(E)<+\infty$ and for every finite perimeter set $\Omega\subset M^n$ such that $\vol(\Omega)=\vol(E)$ one has $P(E)\leq P(\Omega)$.
\end{defn}

The above definition of isoperimetry can of course be
rephrased in terms of the isoperimetric profile $I$ by saying that a subset $E \subset M^n$ of finite perimeter is isoperimetric for the volume $V$ if $\mathrm{vol}{(E)} = V$ and $I(V) = P(E)=\mathcal{H}^{n-1}{(\partial^* E)}$.

We also need to recall the definition of the simply connected radial models with constant sectional curvature.

\begin{defn}[Models of constant sectional curvature, cf. {\cite[Example 1.4.6]{Petersen2016}}]\label{def:Models}
Let us define
\[
\sn_k(r) := \begin{cases}
(-k)^{-\frac12} \sinh((-k)^{\frac12} r) & k<0,\\
r & k=0,\\
k^{-\frac12} \sin(k^{\frac12} r) & k>0.
\end{cases}
\]

If $k>0$, then $((0,\pi/\sqrt{k}]\times \mathbb S^{n-1},\d r^2+\mathrm{sn}_k^2(r)g_1)$, where $g_1$ is the canonical metric on $\mathbb S^{n-1}$, is the radial model of dimension $n$ and constant sectional curvature $k$. The metric can be smoothly extended at $r=0$, and thus we shall write that the the metric is defined on the ball $\mathbb B^n_{\pi/\sqrt{k}} \subset \R^n$. The Riemannian manifold $(\mathbb B^n_{\pi/\sqrt{k}}, g_k\eqdef \d r^2+\mathrm{sn}_k^2(r)g_1)$ is the unique (up to isometry) simply connected Riemannian manifold of dimension $n$ and constant sectional curvature $k>0$.

If instead $k\leq 0$, then $((0,+\infty)\times\mathbb S^{n-1},\d r^2+\mathrm{sn}_k^2(r)g_1)$ is the radial model of dimension $n$ and constant sectional curvature $k$. Extending the metric at $r=0$ analogously yields the unique (up to isometry) simply connected Riemannian manifold of dimension $n$ and constant sectional curvature $k\leq 0$, in this case denoted by $(\R^n,g_k)$.

We denote by $v(n,k,r)$ the volume of the ball of radius $r$ in the (unique) simply connected Riemannian manifold of sectional curvature $k$ of dimension $n$, and by $s(n, k, r)$ the volume of the boundary of such a ball. In particular $s(n,k,r)=n\omega_n\mathrm{sn}_k^{n-1}(r)$ and $v(n,k,r)=\int_0^rn\omega_n\mathrm{sn}_k^{n-1}(t)\de t$, where $\omega_n$ is the Euclidean volume of the Euclidean unit ball in $\mathbb R^n$.

Moreover, for given $n$, we denote by $\dist_k,\vol_k, P_k$ the geodesic distance, the volume measure, and the perimeter functional on the simply connected Riemannian manifold of sectional curvature $k$ (and dimension $n$), respectively.
\end{defn}

Let us also recall a classical definition for the convenience of the reader.

\begin{defn}[$\mathrm{AVR}$ and Euclidean volume growth]\label{def:AVR}
Let $(M^n,g)$ be a complete noncompact Riemannian manifold with $\ric\geq 0$. Thus, from Bishop--Gromov comparison in \cref{thm:BishopGromov} we know that the function $[0,+\infty)\ni r\to \frac{\vol(B_r(p))}{\omega_nr^n}$ is nonincreasing and goes to 1 as $r\to 0^+$. For any $p\in M^n$, we define
$$
\mathrm{AVR}(M^n,g):=\lim_{r\to+\infty}\frac{\vol(B_r(p))}{\omega_nr^n},
$$
the {\em asymptotic volume ratio} of $(M^n,g)$. The previous definition is independent of the choice of $p\in M^n$.
Notice that, by Bishop--Gromov comparison, we have $0\leq \mathrm{AVR}(M^n,g)\leq 1$, and $\vol(B_r(p))\geq \mathrm{AVR}(M^n,g) \omega_nr^n$ for every $r>0$, and every $p\in M^n$. If $\mathrm{AVR}(M^n,g)>0$ we say that $(M^n,g)$ has {\em Euclidean volume growth}.
\end{defn}

Let us now briefly recall the main concepts we will need from the theory of metric measure spaces. We recall that a {\em metric measure space, $\mathrm{m.m.s.}$ for short,} $(X,\dist_X,\mathfrak{m}_X)$ is a triple where $(X,\dist_X)$ is a locally compact separable metric space and $\mathfrak{m}_X$ is a Borel measure bounded on bounded sets. A {\em pointed metric measure space} is a quadruple $(X,\dist_X,\mathfrak{m}_X,x)$ where $(X,\dist_X,\mathfrak{m}_X)$ is a metric measure space and $x\in X$ is a point. 
\begin{center}
    For simplicity, and since it will always be our case, we will always assume that given $(X,\dist_X,\meas_X)$ a m.m.s.\! the support ${\rm spt}\,\meas_X$ of the measure $\meas_X$ is the whole $X$.
\end{center}
We assume the reader to be familiar with the notion of pointed measured Gromov--Hausdorff convergence, referring to \cite[Chapter 27]{VillaniBook} and to \cite[Chapter 7 and 8]{BuragoBuragoIvanovBook} for an overview on the subject. In the following treatment we introduce the pmGH-convergence already in a proper realization even if this is not the general definition. Nevertheless, the (simplified) definition of Gromov--Hausdorff convergence via a realization is equivalent to the standard definition of pmGH convergence in our setting, because in the applications we will always deal with locally uniformly doubling measures, see \cite[Theorem 3.15 and Section 3.5]{GigliMondinoSavare15}. The following definition is actually taken from the introductory exposition of \cite{AmborsioBrueSemola19}.

\begin{defn}[pGH and pmGH convergence]\label{def:GHconvergence}
A sequence $\{ (X_i, \dist_i, x_i) \}_{i\in \N}$ of pointed metric spaces is said to converge in the \emph{pointed Gromov--Hausdorff topology, in the $\mathrm{pGH}$ sense for short,} to a pointed metric space $ (Y, \dist_Y, y)$ if there exist a complete separable metric space $(Z, \dist_Z)$ and isometric embeddings
\[
\begin{split}
&\Psi_i:(X_i, \dist_i) \to (Z,\dist_Z), \qquad \forall\, i\in \N,\\
&\Psi:(Y, \dist_Y) \to (Z,\dist_Z),
\end{split}
\]
such that for any $\eps,R>0$ there is $i_0(\varepsilon,R)\in\mathbb N$ such that
\[
\Psi_i(B_R^{X_i}(x_i)) \subset \left[ \Psi(B_R^Y(y))\right]_\eps,
\qquad
\Psi(B_R^{Y}(y)) \subset \left[ \Psi_i(B_R^{X_i}(x_i))\right]_\eps,
\]
for any $i\ge i_0$, where $[A]_\eps\eqdef \{ z\in Z \st \dist_Z(z,A)\leq \eps\}$ for any $A \subset Z$.

Let $\meas_i$ and $\mu$ be given in such a way $(X_i,\dist_i,\meas_i,x_i)$ and $(Y,\dist_Y,\mu,y)$ are m.m.s.\! If in addition to the previous requirements we also have $(\Psi_i)_\sharp\mathfrak{m}_i \rightharpoonup \Psi_\sharp \mu$ with respect to duality with continuous bounded functions on $Z$ with bounded support, then the convergence is said to hold in the \emph{pointed measure Gromov--Hausdorff topology, or in the $\mathrm{pmGH}$ sense for short}.
\end{defn}

\subsection{$\RCD$ spaces}\label{sub:RCD}
Let us briefly introduce the so-called $\RCD$ condition for m.m.s. Since we will use part of the $\RCD$ theory just as an instrument for our purposes and since we will never use in the paper the specific definition of $\RCD$ space, we just outline the main references on the subject and we refer the interested reader to the survey of Ambrosio \cite{AmbrosioSurvey} and the references therein. 

After the introduction, in the independent works \cite{Sturm1,Sturm2} and \cite{LottVillani}, of the curvature dimension condition $\CD(k,n)$ encoding in a synthetic way the notion of Ricci curvature bounded from below by $k$ and dimension bounded above by $n$, the definition of $\RCD(k,n)$ m.m.s.\! was first proposed in \cite{GigliRCD} and then studied in \cite{Gigli13, ErbarKuwadaSturm15,AmbrosioMondinoSavare15}, see also \cite{CavallettiMilman16} for the equivalence between the $\RCD^*(k,n)$ and the $\RCD(k,n)$ condition. The infinite dimensional counterpart of this notion had been previously investigated in \cite{AmbrosioGigliSavare14}, see also \cite{AmbrosioGigliMondinoRajala15} for the case of $\sigma$-finite reference measures. 

\begin{remark}[pmGH limit of $\RCD$ spaces]\label{remark:stability}
	We recall that, whenever it exists, a pmGH limit of a sequence $\{(X_i,\dist_i,\meas_i,x_i)\}_{i\in\mathbb N}$ of (pointed) $\RCD(k,n)$ spaces is still an $\RCD(k,n)$ metric measure space.
\end{remark}

In particular, due to the compatibility of the $\RCD$ condition with the smooth case of Riemannian manifolds with Ricci curvature bounded from below and to its stability with respect to pointed measured Gromov--Hausdorff convergence, limits of smooth Riemannian manifolds with Ricci curvature uniformly bounded from below by $k$ and dimension uniformly bounded from above by $n$ are $\RCD(k,n)$ spaces. Then the class of $\RCD$ spaces includes the class of Ricci limit spaces, i.e., limits of sequences of Riemannian manifolds with the same dimension and with Ricci curvature uniformly bounded from below. The study of Ricci limits was initiated by Cheeger and Colding in the nineties in the series of papers \cite{ChCo0,ChCo1,ChCo2,ChCo3} and has seen remarkable developments in more recent years.
Since the above mentioned pioneering works, it was known that the regularity theory for Ricci limits improves adding to the lower curvature bound a uniform lower bound for the volume of unit balls along the converging sequence of Riemannian manifolds: this gives raise to the so-called notion of \textit{noncollapsed} Ricci limits. In particular, as a consequence of the volume convergence theorem proved in \cite{Colding97}, it is known that in the noncollapsed case the limit measure of the volume measures is the Hausdorff measure on the limit metric space, while this might not be the case for a general Ricci limit space.

Now we are ready to state the volume convergence theorems obtained by Gigli and De Philippis in \cite[Theorem 1.2 and Theorem 1.3]{DePhilippisGigli18}, which are the synthetic version of the celebrated volume convergence of Colding \cite{Colding97}. Whenever we write a metric measure space as a triple $(X,\dist,\mathcal{H}^n)$, it is understood that the measure $\mathcal{H}^n$ is the $n$-dimensional Hausdorff measure corresponding to the distance $\dist$ on $X$.

\begin{thm}\label{thm:volumeconvergence}
	Let $\{(X_i,\dist_i,\mathcal{H}^n,x_i)\}_{i\in\mathbb N}$ be a sequence of pointed $\RCD(k,n)$ m.m.s.\! with $k\in\mathbb{R}$ and $n\in [1,+\infty)$. Assume that $(X_i,\dist_i,x_i)$ converges in the pGH topology to $(X,\dist,x)$. Then precisely one of the following happens
	\begin{itemize}
		\item[(a)] $\limsup_{i\to\infty}\mathcal{H}^n\left(B_1(x_i)\right)>0$. Then the $\limsup$ is a limit and $(X_i,\dist_i,\mathcal{H}^n,x_i)$ converges in the pmGH topology to $(X,\dist,\mathcal{H}^n,x)$. Hence $(X,\dist,\mathcal{H}^n)$ is an $\RCD(k,n)$ m.m.s.;
		\item[(b)] $\lim_{i\to\infty}\mathcal{H}^n(B_1(x_i))=0$. In this case we have $\dim_{H}(X,\dist)\le n-1$, where $\dim_H(X,\dist)$ is the Hausdorff dimension of $(X,\dist)$. 
	\end{itemize}
	Moreover, for $k\in\mathbb R$ and $n\in[1,+\infty)$, let $\mathbb B_{k,n,R}$ be the collection of all equivalence classes up to isometry of closed balls of radius $R$ in $\RCD(k,n)$ spaces, equipped with the Gromov-Hausdorff distance. Then the map $\mathbb B_{k,n,R}\ni Z\to \mathcal{H}^n(Z)$ is real-valued and continuous.
\end{thm}


\begin{remark}[Gromov precompactness theorem for $\RCD$ spaces]\label{rem:GromovPrecompactness}
Here we recall the synthetic variant of Gromov's precompactness theorem for $\RCD$ spaces, see \cite[Equation (2.1)]{DePhilippisGigli18}. Let $\{(X_i,\dist_i,\meas_i,x_i)\}_{i\in\mathbb N}$ be a sequence of $\RCD(k_i,n)$ spaces with $n\in[1,+\infty)$, $\supp(\meas_i)=X_i$ for every $i\in\mathbb N$, $\meas_i(B_1(x_i))\in[v,v^{-1}]$ for some $v\in(0,1)$ and for every $i\in\mathbb N$, and $k_i\to k\in\mathbb R$. Then there exists a subsequence pmGH-converging to some $\RCD(k,n)$ space $(X,\dist,\meas,x)$ with $\supp(\meas)=X$.
\end{remark}



We conclude this part by recalling a few basic definitions and results concerning the perimeter functional in the setting of metric measure spaces (see \cite{Ambrosio02,Miranda03,AmbrosioDiMarino14}).

\begin{defn}[$BV$ functions and perimeter on m.m.s.]\label{def:BVperimetro}
Let $(X,\dist,\meas)$ be a metric measure space. A function $f \in L^1(X,\meas)$ is said to belong to the space of \emph{bounded variation functions} $BV(X,\dist,\meas)$ if there is a sequence $f_i \in {\rm Lip}_{\mathrm{loc}}(X)$ such that $f_i \to f$ in $L^1(X,\meas)$ and $\limsup_i \int_X \lip f_i \de \meas < +\infty$, where $\lip u (x) \eqdef \limsup_{y\to x} \frac{|u(y)-u(x)|}{\dist(x,y)}$ is the \emph{slope} of $u$ at $x$, for any accumulation point $x\in X$, and $\lip u(x):=0$ if $x\in X$ is isolated. In such a case we define
\[
|Df|(A) \eqdef \inf\left\{\liminf_i \int_A \lip f_i \de\meas \st \text{$f_i \in {\rm Lip}_{\rm loc}(A), f_i \to f $ in $L^1(A,\meas)$} \right\},
\]
for any open set $A\subset X$.

If $E\subset X$ is a Borel set and $A\subset X$ is open, we  define the \emph{perimeter $P(E,A)$  of $E$ in $A$} by
\[
P(E,A) \eqdef \inf\left\{\liminf_i \int_A \lip u_i \de\meas \st \text{$u_i \in {\rm Lip}_{\rm loc}(A), u_i \to \chi_E $ in $L^1_{\rm loc}(A,\meas)$} \right\},
\]
We say that $E$ has \emph{finite perimeter} if $P(E,X)<+\infty$, and we denote by $P(E)\eqdef P(E,X)$. Let us remark that the set functions $|Df|, P(E,\cdot)$ above are restrictions to open sets of Borel measures that we denote by $|Df|, |D\chi_E|$ respectively, see \cite{AmbrosioDiMarino14}, and \cite{Miranda03}. 

The \emph{isoperimetric profile of $(X,\dist,\meas)$} is then
\[
I_{X}(V)\eqdef \left\{ P(E) \st \text{$E\subset X$ Borel, $\meas(E)=V$} \right\},
\]
for any $V\in[0,\meas(X))$. If $E\subset X$ is Borel with $\meas(E)=V$ and $P(E)=I_X(V)$, then we say that $E$ is an \emph{isoperimetric region}.
\end{defn}

It follows from classical approximation results (cf. \cref{rem:Approximation}) that the above definition yields the usual notion of perimeter on any Riemannian manifold $(M^n,g)$ recalled at the beginning of this section.

\begin{remark}[Coarea formula on metric measure spaces]\label{rem:PerimeterMMS}
Let $(X,\dist,\meas)$ be a metric measure space. Let us observe that from the definitions given above, a Borel set $E$ with finite measure has finite perimeter if and only if the characteristic function $\chi_E$ belongs to $BV(X,\dist,\meas)$.

If $f \in BV(X,\dist,\meas)$, then $\{f>\alpha\}$ has finite perimeter for a.e. $\alpha \in \R$ and the \emph{coarea formula} holds
\begin{equation*}
    \int_X u \de |Df| = \int_{-\infty}^{+\infty} \left( \int_X u \de \abs{D\chi_{\{f>\alpha\}}} \right) \de \alpha,
\end{equation*}
for any Borel function $u:X\to [0,+\infty]$, see \cite[Proposition 4.2]{Miranda03}. If $f$ is also continuous and nonnegative, then $|Df|(\{f=\alpha\})=0$ for \emph{every} $\alpha \in [0,+\infty)$ and the \emph{localized coarea formula} holds
\begin{equation*}
    \int_{\{a<f<b\}} u \de |Df| = \int_a^b \left( \int_X u \de \abs{D\chi_{\{f>\alpha\}}} \right) \de \alpha,
\end{equation*}
for every Borel function $u:X\to [0,+\infty]$ and every $0\le a < b < +\infty$, see \cite[Corollary 1.9]{AmborsioBrueSemola19}.

Applying the above coarea formulas to the distance function $r(y)=\dist(y,x)$ from a fixed point $y\in X$, one deduces that balls $B_r(y)$ have finite perimeter for almost every radius $r>0$, the function $r\mapsto \meas(B_r(y))$ is continuous, $\meas(\partial B_r(y)) =0$ for every $r>0$, and $\tfrac{d}{dr}\meas(B_r(y)) = P(B_r(y))$ for a.e. $r>0$.
\end{remark}

\begin{remark}[Bishop--Gromov comparison theorem on m.m.s.]\label{rem:PerimeterMMS2}
Let us recall that for an arbitrary $\CD((n-1)k,n)$ space $(X,\dist,\meas)$ the classical Bishop--Gromov volume comparison (cf. \cref{thm:BishopGromov}) still holds. More precisely, for a fixed $x\in X$, the function $\meas(B_r(x))/v(n,k,r)$ is nonincreasing in $r$ and the function $P(B_r(x))/s(n,k,r)$ is essentially nonincreasing in $r$, i.e., $P(B_R(x))/s(n,k,R) \le P(B_r(x))/s(n,k,r)$ for almost every radii $R\ge r$, see \cite[Theorem 18.8, Equation (18.8), Proof of Theorem 30.11]{VillaniBook}. Moreover, it holds that $P(B_r(x))/s(n, k, r)\leq \vol(B_r(x))/v(n,k,r)$ for any $r>0$, indeed the last inequality follows from the monotonicity of the volume and perimeter ratios together with the coarea formula on balls.

Moreover, if $(X,\dist,\mathcal{H}^n)$ is an $\RCD((n-1)k,n)$ space, one can conclude that $\mathcal{H}^n$-almost every point has a unique measure Gromov--Hausdorff tangent isometric to $\mathbb R^n$ (\cite[Theorem 1.12]{DePhilippisGigli18}), and thus, from the volume convergence in \cref{thm:volumeconvergence}, we get
 \begin{equation}\label{eqn:VolumeConv}
 \lim_{r\to 0}\frac{\mathcal{H}^n(B_r(x))}{v(n,k,r)}=\lim_{r\to 0}\frac{\mathcal{H}^n(B_r(x))}{\omega_nr^n}=1, \qquad \text{for $\mathcal{H}^n$-almost every $x$},
 \end{equation}
 where $\omega_n$ is the volume of the unit ball in $\mathbb R^n$. Moreover, since the density function $x\mapsto \lim_{r\to 0} \tfrac{\mathcal{H}^n(B_r(x))}{\omega_nr^n}$ is lower semicontinuous (\cite[Lemma 2.2]{DePhilippisGigli18}), the latter \eqref{eqn:VolumeConv} implies that the density is bounded above by the constant $1$. Hence, from the monotonicity at the beginning of the remark we deduce that, if $(X,\dist,\mathcal{H}^n)$ is an $\RCD((n-1)k,n)$ space, then for every $x\in X$ we have $\mathcal{H}^n(B_r(x))\leq v(n,k,r)$ for every $r>0$. In particular, if $(X,\dist,\mathcal{H}^n)$ is an $\RCD((n-1)k,n)$ space, then for every $x\in X$ we have $P(B_r(x))\leq s(n,k,r)$ for every $r>0$.
\end{remark}

\begin{remark}[Representation of the perimeter on $\RCD$ spaces]\label{rem:PerimeterMMS3}
Let us fix $(X,\dist,\meas)$ an $\RCD((n-1)k,n)$ space. Hence, from Bishop--Gromov comparison in \cref{rem:PerimeterMMS2}, for any fixed $x\in X$, 
$$
\limsup_{r\to0}\meas(B_{2r}(x))/\meas(B_r(x)) \le \limsup_{r\to0} v(n,k,2r)/v(n,k,r)<+\infty,
$$
i.e., $\meas$ is \emph{asymptotically doubling}, and therefore the Lebesgue Differentiation Theorem holds true, see \cite[Theorem 3.4.3]{HeinonenKoskelaShanmugalingam} and \cite[Lebesgue Differentiation Theorem, p. 77]{HeinonenKoskelaShanmugalingam}.
So it makes sense to identify any Borel set $E$ with the set $E^1$ of points of density $1$, where, in general, 
\[
E^t\eqdef \left\{x\in X \st \lim_{r\searrow0} \frac{\meas(E\cap B_r(x))}{\meas(B_r(x))}=t \right\},
\]
for any $t\in[0,1]$. The \emph{essential boundary} of $E$ is then classically defined by $\partial^* E \eqdef X\setminus(E^0 \cup E^1)$.

As in the case of Riemannian manifolds, if $(X,\dist,\mathcal{H}^n)$ is  an $\RCD ((n-1)k,n)$ space endowed with the $n$-dimensional Hausdorff measure, the perimeter measure can be represented by
\begin{equation}\label{eq:PerimeterRepresentation}
|D\chi_E| = \mathcal{H}^{n-1}\res \partial^* E,
\end{equation}
for any finite perimeter set $E$. In fact, this follows by putting together the representation given in \cite[Theorem 5.3]{Ambrosio02} and the recent one contained in \cite[Corollary 4.2]{BruePasqualettoSemola}.

It easily follows from such a representation formula that if $E\subset X$ has finite perimeter and $x\in X$, then for a.e. radius $r>0$ the intersection $B_r(x)\cap E$ has finite perimeter and
\begin{equation}\label{eq:PerimeterIntersection}
|D \chi_{B_r(x)\cap E}| = \mathcal{H}^{n-1}\res (\partial^* E \cap B_r(x) ) + \mathcal{H}^{n-1}\res(E \cap \partial^* B_r(x)).
\end{equation}
Indeed for a.e. $r>0$ the ball $B_r(x)$ has finite perimeter and $|D\chi_E|(\partial B_r(x))=0$; so \eqref{eq:PerimeterIntersection} follows from \eqref{eq:PerimeterRepresentation} by noticing that for such an $r$ it holds that $\partial^*(B_r(x) \cap E) = (\partial^* E \cap B_r(x) ) \cup (E \cap \partial^* B_r(x))$ up to $\mathcal{H}^{n-1}$-negligible sets.

We mention that finer regularity properties of sets of finite perimeter have been recently proved in \cite{BruePasqualettoSemolaConstancy21}.
\end{remark}

\subsection{Sets of finite perimeter and GH-convergence}\label{sub:FiniteGH}

We need to recall a generalized $L^1$-notion of convergence for sets defined on a sequence of metric measure spaces converging in the pmGH sense. Such a definition is given in \cite[Definition 3.1]{AmborsioBrueSemola19}, and it is investigated in \cite{AmborsioBrueSemola19} capitalizing on the results in \cite{AmbrosioHonda17}.

\begin{defn}[$L^1$-strong and $L^1_{\mathrm{loc}}$ convergence]\label{def:L1strong}
Let $\{ (X_i, \dist_i, \mathfrak{m}_i, x_i) \}_{i\in \N}$  be a sequence of pointed metric measure spaces converging in the pmGH sense to a pointed metric measure space $ (Y, \dist_Y, \mu, y)$ and let $(Z,\dist_Z)$ be a realization as in \cref{def:GHconvergence}.

We say that a sequence of Borel sets $E_i\subset X_i$ such that $\mathfrak{m}_i(E_i) < +\infty$ for any $i \in \N$ converges \emph{in the $L^1$-strong sense} to a Borel set $F\subset Y$ with $\mu(F) < +\infty$ if $\mathfrak{m}_i(E_i) \to \mu(F)$ and $\chi_{E_i}\mathfrak{m}_i \rightharpoonup \chi_F\mu$ with respect to the duality with continuous bounded functions with bounded support on $Z$.

We say that a sequence of Borel sets $E_i\subset X_i$ converges \emph{in the $L^1_{\mathrm{loc}}$-sense} to a Borel set $F\subset Y$ if $E_i\cap B_R(x_i)$ converges to $F\cap B_R(y)$ in $L^1$-strong for every $R>0$.
\end{defn}

Observe that in the above definition it makes sense to speak about the convergence $\chi_{E_i}\mathfrak{m}_i \rightharpoonup \chi_F\mu$ with respect to the duality with continuous bounded functions with bounded support on $Z$ as $(X_i,\dist_i),(Y,\dist_Y)$ can be assumed to be topological subspaces of $(Z,\dist_Z)$ by means of the isometries $\Psi_i,\Psi$ of \cref{def:GHconvergence}, and the measures $\mathfrak{m}_i,\mu$ can be then identified with the push-forwards $(\Psi_i)_\sharp\mathfrak{m}_i,\Psi_\sharp\mu$ respectively.
 
The following result is taken from \cite{AmborsioBrueSemola19} and will be of crucial importance in the proof of \cref{thm:MassDecomposition}.
\begin{prop}[{\cite[Proposition 3.3, Corollary 3.4, Proposition 3.6, Proposition 3.8]{AmborsioBrueSemola19}}]\label{prop:SemicontinuitaAmbrosioBrueSemola}
Let $k\in\mathbb R$, $n\geq 1$, and $\{(X_i,\dist_i,\meas_i,x_i)\}_{i\in\mathbb N}$ be a sequence of $\RCD(k,n)$ m.m.s.\! converging in the pmGH sense to $(Y,\dist_Y,\mu,y)$. Then, 
\begin{itemize}
    \item[(a)] For any $r>0$ and for any sequence of finite perimeter sets $E_i\subset \overline B_{r}(x_i)$  satisfying
    $$
    \sup_{i\in\mathbb N}|D\chi_{E_i}|(X_i)<+\infty,
    $$
    there exists a subsequence $i_k$ and a finite perimeter set $F\subset \overline{B}_r(y)$ such that $E_{i_k}\to F$ in $L^1$-strong as $k\to+\infty$. Moreover 
    $$
    |D\chi_F|(Y)\leq \liminf_{k\to+\infty}|D\chi_{E_{i_k}}|(X_{i_k}).
    $$
    
    \item[(b)] For any sequence of Borel sets $E_i\subset X_i$ with 
    $$
    \sup_{i\in\mathbb N}|D\chi_{E_i}|(B_R(x_i))<+\infty, \qquad \forall\,R>0,
    $$
    there exists a subsequence $i_k$ and a Borel set $F\subset Y$ such that $E_{i_k}\to F$ in $L^1_{\mathrm{loc}}$.
    
    \item[(c)] Let $F\subset Y$ be a bounded set of finite perimeter. Then there exist a subsequence $i_k$, and uniformly bounded finite perimeter sets $E_{i_k}\subset X_{i_k}$ such that $E_{i_k}\to F$ in $L^1$-strong and $|D\chi_{E_{i_k}}|(X_{i_k})\to |D\chi_F|(Y)$ as $k\to+\infty$.
\end{itemize}
\end{prop}

With the help of the previous result we can now prove the following lemma which will be used in the forthcoming section.
\begin{lemma}\label{lem:ProfileOnBoundedSets}
Let $(X,\dist,\mathcal{H}^n)$ be an $\RCD ((n-1)k,n)$ space with $\mathcal{H}^n(X)=+\infty$. If, for some $v_0>0$,  $\mathcal{H}^n(B_1(x))\ge v_0$ for any $x\in X$, then the isoperimetric profile $I_X$ of $X$ can be rewritten as
\begin{equation}\label{eq:ProfileBoundedSets}
  I_X(V) = \inf \left\{P(E) \st \text{$E\subset X$ Borel, $\mathcal{H}^n(E)=V$, $E$ bounded} \right\}
\qquad
\forall\, V\in (0,+\infty).  
\end{equation}
\end{lemma}

\begin{proof}
Let us observe first that if $E\subset X$ is a finite perimeter set with finite measure $\mathcal{H}^n(E)<+\infty$, then for any point $o\in X$ there exists a sequence of radii $R_i\to+\infty$ such that
\begin{itemize}
    \item $\mathcal{H}^n(E \cap B_{R_i}(o)) \ge \mathcal{H}^n(E) - 1/i $ for any $i$;
    \item $P(E \cap B_{R_i}(o)) = P(E, B_{R_i}(o)) + \mathcal{H}^{n-1}(E \cap \partial^* B_{R_i}(o) ) $ for any $i$;
    \item $\mathcal{H}^{n-1}(E \cap \partial^* B_{R_i}(o) ) \le 1/i$ for any $i$.
\end{itemize}
Indeed, by the results in \cref{rem:PerimeterMMS}, and \cref{rem:PerimeterMMS3}, we know that
\[
\mathcal{H}^n (E \cap B_r(o)) = \int_0^{r} \mathcal{H}^{n-1}(E \cap \partial^* B_t(o) ) \de t \xrightarrow[r\to+\infty]{} \mathcal{H}^{n}(E) <+\infty.
\]
Recalling also \eqref{eq:PerimeterIntersection} in order to justify the second item above, the sought claim follows.

Now let $V\in (0,+\infty)$ and consider $E_j\subset X$ with $\mathcal{H}^n(E_j)=V$ such that $P_X(E_j) \le I_X(V) + 1/j$. Fix $o\in X$ and let $R_i^j$ be given by the first part of the proof applied to $E_j$. For any $i,j$ let $B_{\rho_{i,j}}(p_{i,j}) \Subset X \setminus B_{R_i^j}(o)$ be such that
\[
\mathcal{H}^n( B_{\rho_{i,j}}(p_{i,j}) ) = V - \mathcal{H}^n(E_j \cap B_{R_i^j}(o)) \le \frac1i \, ,
\]
and moreover $p_{i,j}$ is chosen such that the comparison inequalities discussed in \cref{rem:PerimeterMMS2} hold.
Such balls exist since $\mathcal{H}^n(X)=+\infty$. 
Since balls of radius $1$ have volume $\ge v_0$, we can also assume that $\rho_{i,j}<1$. Then the volume comparison (see \cref{rem:PerimeterMMS2}) implies $\mathcal{H}^n(B_{\rho_{i,j}}(p_{i,j}))\ge v(n,k,\rho_{i,j}) v_0/v(n,k,1)$. Hence $\lim_i \rho_{i,j} = 0$ for any $j$. We then get that, by using the perimeter comparison (see \cref{rem:PerimeterMMS2})
\[
\lim_i P ( B_{\rho_{i,j}}(p_{i,j}) ) \le \lim_i s(n,k,\rho_{i,j} ) = 0,
\]
for any $j$. Hence
\[
 P( [ E_j \cap B_{R_i^j}(o) ] \cup B_{\rho_{i,j}}(p_{i,j})) \le P(E_j) + \frac1i +  s(n,k,\rho_{i,j} )  \le I_X(V) + \frac1j + \frac1i +  s(n,k,\rho_{i,j} ) .
\]
Taking $i_j\ge j$ sufficiently large for any fixed $j$ so that $s(n,k,\rho_{i_j,j} ) \le 1/j$ yields that
\[
P_X( [ E_j \cap B_{R_{i_j}^j}(o) ] \cup B_{\rho_{i_j,j}}(p_{i_j,j})) \le  I_X(V) + \frac3j. 
\]
Hence $[ E_j \cap B_{R_{i_j}^j}(o) ] \cup B_{\rho_{i_j,j}}(p_{i_j,j})$ is a minimizing (for the perimeter) sequence of bounded sets of volume $V$. So this implies \eqref{eq:ProfileBoundedSets}.
\end{proof}

\section{Asymptotic geometry and isoperimetric profile}\label{sec:AsymptGeometry}

In this section we prove some inequalities regarding the isoperimetric profile in some special classes of Riemannian manifold. We first need the following useful result, which is proved without uniform assumptions on the volume of unit balls.

\begin{lemma}\label{lem:ProfileUSC}
Let $(X,\dist,\mathcal{H}^n)$ be an $\RCD ((n-1)k,n)$ space. Then its isoperimetric profile $I_X:(0,\mathcal{H}^n(X))\to [0,+\infty)$ is upper semicontinuous.
\end{lemma}

\begin{proof}
Fix $V\in (0,\mathcal{H}^n(X))$ and let $\eta>0$. Then take $E\subset X$ Borel such that $\mathcal{H}^n(E)=V$ and $P_X(E) \le I_X(V) + \eta$. Since the measure $\mathcal{H}^n$ is asymptotically doubling, see  \cref{rem:PerimeterMMS3}, we can identify $E$ with the set of density one points $E^1$, and we denote $E^0$ the set of density zero points of $E$. Let us fix $x \in E^1=E$ and $y \in E^0$ such that the comparison inequalities discussed in \cref{rem:PerimeterMMS2} hold. There is $\overline{\rho}$ such that
\[
\begin{split}
  \mathcal{H}^n(E \cap B_\rho(x) ) > \frac34 \mathcal{H}^n(B_\rho(x)) & \qquad \forall\,\rho \in (0,\overline{\rho}),\\
  \mathcal{H}^n(E \cap B_\rho(y) ) < \frac14 \mathcal{H}^n(B_\rho(y)) & \qquad \forall\,\rho \in (0,\overline{\rho}),
\end{split}
\]
and $\dist(x,y)>3\overline{\rho}$.
We claim that there is $\delta \in (0,V/2)$ and $\omega:(V-\delta,V+\delta)\to \R$ such that for any $v \in (V-\delta,V+\delta)$ there is $\rho_x=\rho_x(v),\rho_y=\rho_y(v) \in [0,\overline{\rho})$ such that
\begin{equation}\label{eq:ClaimUSC}
    \begin{split}
        \mathcal{H}^n((E \cup B_{\rho_y}(y) )\setminus B_{\rho_x}(x) ) &= v, \\
        P_X((E \cup B_{\rho_y}(y) )\setminus B_{\rho_x}(x)) &\le P_X(E) + \omega(v),\\
        \lim_{v\to V} \omega(v) &= 0.
    \end{split}
\end{equation}

We observe that such a claim implies the statement of the Lemma. Indeed  if $v_j\to V$ is any sequence, then the claim yields a sequence of sets $E_j \eqdef (E \cup B_{\rho_{y,j}}(y) )\setminus B_{\rho_{x,j}}(x)$ with $\mathcal{H}^n(E_j) = v_j$, and that satisfy
\[
I_X(v_j) \le P_X(E_j) \le P_X(E) + \omega(v_j) \le I_X(V) + \eta + \omega(v_j).
\]
Passing to the $\limsup$ as $j\to +\infty$ in the above inequality, since $v_j\to V$ and $V$ are arbitrary, and then letting $\eta\to 0$, readily implies that $I_X$ is upper semicontinuous.

So we are left to prove the claim. Take
\[
0<\delta < \min \left\{ \mathcal{H}^n(B_{\overline{\rho}}(y)\setminus E ) , \mathcal{H}^n(B_{\overline{\rho}}(x) \cap E ), V/2 \right\}.
\]
Observe that the function
\begin{equation}\label{eq:ContinuityFunctionVolume}
    \begin{split}
        [0,\overline{\rho})^2 \ni (\rho_1,\rho_2) \mapsto \, &\mathcal{H}^n ( (E \cup B_{\rho_2}(y) )\setminus B_{\rho_1}(x) ) \\
        &= \mathcal{H}^n ( E \cup B_{\rho_2}(y) ) - \mathcal{H}^n( E\cap B_{\rho_1}(x) ) \\
        &= \mathcal{H}^n(E) - \mathcal{H}^n(E \cap B_{\rho_2}(y)) + \mathcal{H}^n (B_{\rho_2}(y) ) - \mathcal{H}^n( E\cap B_{\rho_1}(x) ),
    \end{split}
\end{equation}
is continuous; indeed by the coarea formula (\cref{rem:PerimeterMMS}) we know that 
\[
\mathcal{H}^n(E \cap B_{\rho}(z)) = \int_0^\rho \int_X \chi_E \de |D\chi_{B_t(z)}| \de t,
\]
for any $\rho>0$ and $z \in X$.

We are ready to prove \eqref{eq:ClaimUSC}. Let $v \in (V-\delta,V+\delta)$; we need to define $\omega(v)$, $\rho_x(v)$, and $\rho_y(v)$. If $v=V$, then $\omega(V)=\rho_x(V)=\rho_y(V)=0$ works. So we assume $v> V$, the case $v<V$ being completely analogous. By the choice of $\delta$ there is $\rho_v\in(0,\overline{\rho})$ such that
\begin{equation*}
    \mathcal{H}^n(E \cup B_{\rho_v}(y)) = v, 
    \qquad
    \mathcal{H}^n(E \cup B_{\rho}(y)) > v \quad
    \forall\, \rho \in (\rho_v,\overline{\rho}).
\end{equation*}
By continuity of the map in \eqref{eq:ContinuityFunctionVolume} there is $\widetilde{\rho}_v \in (\rho_v,\overline{\rho})$ such that
\begin{equation*}
\forall\, \rho \in (\rho_v, \widetilde{\rho}_v) \,\, \exists\, \sigma \in (0,\overline{\rho}) \st \mathcal{H}^n ( (E\cup B_\rho(y))\setminus B_\sigma(x)) = v.
\end{equation*}
Hence there exist $\rho_x \in (\rho_v,\widetilde{\rho}_v)$ and $\rho_y \in(0,\overline{\rho})$ such that
\begin{equation}\label{eq:PerClaim}
\mathcal{H}^n ( (E\cup B_{\rho_y}(y))\setminus B_{\rho_x}(x)) = v,
\end{equation}
and in addition $P_X(B_{\rho_y}(y)) \le s(n,k,\rho_y)$, $P_X(B_{\rho_x}(x)) \le s(n,k,\rho_x)$ (see the comparison of the perimeter in \cref{rem:PerimeterMMS2}). Therefore
\begin{equation}\label{eq:PerClaim2}
\begin{split}
    P_X( (E\cup B_{\rho_y}(y))\setminus B_{\rho_x}(x)) &\le P_X(E) + P_X(B_{\rho_y}(y)) + P_X(B_{\rho_x}(x)) \\
    &\le P_X(E) +  s(n,k,\rho_y) + s(n,k,\rho_x).
\end{split}
\end{equation}
Moreover, we can clearly choose $\rho_x,\rho_y \to 0$ if $v\to V^+$. Hence defining $\omega(v) \eqdef  s(n,k,\rho_y) + s(n,k,\rho_x)$, \eqref{eq:PerClaim} and \eqref{eq:PerClaim2} imply the claimed \eqref{eq:ClaimUSC}.
\end{proof}

We now prove a proposition that roughly says that the isoperimetric profile of a manifold is less or equal than the isoperimetric profile of every pmGH limit at infinity.
The following proposition has to be read as a generalization of \cite[Lemma 2.7]{Nar14}.

\begin{prop}\label{prop:ComparisonIsoperimetricProfile} 
Let $(M^n,g)$ be a complete noncompact noncollapsed Riemannian manifold such that $\ric\geq (n-1)k$ for some $k\in(-\infty,0]$. Let $p_i\in M^n$ be a diverging sequence of points on $M^n$. Then, up to subsequence, there exists $(X_{\infty},\dist_{\infty},\mathcal{H}^n,p_{\infty})$ a pointed Ricci limit space, and thus an $\RCD(k,n)$ space,
such that
\begin{equation}\label{eq:ConvergenceIsop}
(M^n,\dist,\vol,p_i)\xrightarrow[i\to+\infty]{pmGH}(X_{\infty},\dist_{\infty},\mathcal{H}^n,p_{\infty}).
\end{equation}

Moreover, whenever a diverging sequence of points $p_i\in M^n$ and a pointed Ricci limit space $(X_{\infty},\dist_{\infty},\mathcal{H}^n,p_{\infty})$ satisfy \eqref{eq:ConvergenceIsop}, then
\begin{equation}\label{eq:GeneralInequalityIsopProfile}
    I_{(M^n,g)}(V) \le I_{(M^n,g)}(V_1) + I_{X_\infty}(V_2) 
    \qquad
    \forall\, V= V_1+ V_2,
\end{equation}
with $V,V_1,V_2\ge 0$.

In particular
\begin{equation}\label{eq:InequalityIsopProfile}
 I_{(M^n,g)}(V)\leq I_{X_{\infty}}(V) \qquad \forall\, V>0,
\end{equation}
and if, for any $j\ge 1$, $\{p_{i,j} \,|\, i \in \N\}$ is a diverging sequence of points on $M^n$ such that $(M^n,\dist,\vol,p_{i,j})\to (X_j,\dist_j,\meas_j,p_j)$ in the pmGH sense as $i\to+\infty$, then
\begin{equation}\label{eq:InfinityInequalityIsopProfile}
    I_{(M^n,g)}(V) \le I_{(M^n,g)}(V_0) + \sum_{j=1}^{+\infty} I_{X_j}(V_j),
\end{equation}
whenever $V= \sum_{j=0}^{+\infty} V_j$ with $V,V_j\ge0$ for any $j$.
\end{prop}

\begin{proof}
First, we observe that since $M^n$ is noncompact and noncollapsed, it has infinite volume; indeed, there exist countably many disjoint balls of radius $1$ contained in $M^n$. The convergence in \eqref{eq:ConvergenceIsop} is just a consequence of Gromov Precompactness Theorem. So we are left to prove \eqref{eq:GeneralInequalityIsopProfile}.

Without loss of generality let $V>0$ be fixed, and let $V_1,V_2\ge0$ with $V_1+V_2=V$. So we can assume $V_2>0$ without loss of generality. Consider $V^j:=V_2-1/j>0$ for $j$ large enough. Let $\Omega\subset M^n$ be a bounded set such that $\vol(\Omega)=V_1$ and $P(\Omega)\le I_{(M^n,g)}(V_1) + \eta$ for a fixed $\eta>0$.

By the fact that $M^n$ is noncollapsed and by \cref{thm:volumeconvergence} we know that for some $v_0>0$ we have $\mathcal{H}^n(B_1(x))\ge v_0$ for any $x \in X_\infty$: indeed, for every $x\in X_{\infty}$ there exists a sequence $\widetilde p_i$ such that $(M^n,\dist,\widetilde p_i)\to (X_{\infty},\dist_\infty,x)$ in the pGH sense, and then we can apply the second part of \cref{thm:volumeconvergence}, together with the fact that $M^n$ is noncollapsed to deduce the sought bound. As $X_\infty$ is noncompact, it also follows that $\mathcal{H}^n(X_\infty)=+\infty$.

Then by \eqref{eq:ProfileBoundedSets} there exist bounded sets $E_j\subset X_\infty$ with $\mathcal{H}^n(E_j) = V^j$ and $P_{X_\infty}(E_j) \le I_{X_\infty}(V^j) + 1/j$. By item (c) in \cref{prop:SemicontinuitaAmbrosioBrueSemola}, up to subsequences in $i$, for any $j$ there are $R_j>0$ and a sequence $F_i^j \subset B_{R_j}(p_i)\subset M^n$ such that $F_i^j\to E_j$ in $L^1$-strong as $i\to+\infty$ and $\lim_i P ( F_i^j) = P_{X_\infty}( E_j) $.

Therefore, if $o\in M^n$ is fixed, there is a ball $B_S(o)$ such that $\Omega\Subset B_S(o)$, and, since $p_i$ diverges at infinity, there are balls $B_{\rho_{i,j}}(o')\subset M^n$ for some $o'\in M^n$ such that
\begin{equation*}
    B_{\rho_{i,j}}(o') \Subset M^n \setminus (B_{R_j}(p_i) \cup B_S(o) ) ,
    \qquad
    \vol( B_{\rho_{i,j}}(o') ) = V_2 - \vol (F^j_i),
\end{equation*}
for any $i,j$, up to subsequences.
For any $j$ there is $i_j$ such that $F^j_{i_j}\Subset M\setminus B_S(o)$, $P(F_{i_j}^j) \le  P_{X_\infty}( E_j) + 1/j $, and $\vol(F_{i_j}^j)\ge V_2- 2/j$. Moreover, since $\lim_j\vol(B_{\rho_{i_j,j}}(o'))=0$, then $\lim_j P(B_{\rho_{i_j,j}}(o'))=0$. Hence, since $F_{i_j}^j$, $B_{\rho_{i_j,j}}(o')$ and $\Omega$ are mutually disjoint, we have, also by exploiting the previous inequalities,
\[
\begin{split}
I_{(M^n,g)}(V) &\le P(F_{i_j}^{j} \cup  B_{\rho_{i_j,j}}(o') \cup \Omega )  
= P(F_{i_j}^j ) + P (B_{\rho_{i_j,j}}(o') ) + P(\Omega) \\
&\le P_{X_\infty}( E_j) + \frac1j + P(B_{\rho_{i_j,j}}(o')) + I_{(M^n,g)}(V_1) + \eta\\
&\le I_{X_\infty}\left(V_2-\frac1j \right) + \frac2j+ P(B_{\rho_{i_j,j}}(o'))  + I_{(M^n,g)}(V_1) + \eta.
\end{split}
\]
Passing to the $\limsup$ in the previous estimate and using that $I_{X_\infty}$ is upper semicontinuous by \cref{lem:ProfileUSC} jointly with the fact that $\eta$ is arbitrary, finally implies \eqref{eq:GeneralInequalityIsopProfile}.

Now \eqref{eq:InequalityIsopProfile} clearly follows from \eqref{eq:GeneralInequalityIsopProfile} with $V_1=0$. Finally, in the notation and assumptions of \eqref{eq:InfinityInequalityIsopProfile}, we can iteratively apply \eqref{eq:GeneralInequalityIsopProfile} to get
\[
\begin{split}
    I_{(M^n,g)}(V)
    &\le I_{X_1}(V_1) + I_{(M^n,g)}\left( V_0 + \sum_{j=2}^{+\infty} V_j\right) 
    \le I_{(M^n,g)}\left( V_0 + \sum_{j=k}^{+\infty} V_j\right) + \sum_{j=1}^{k-1} I_{X_j}(V_j),
\end{split}
\]
for any $k \geq 2$. Letting $k\to+\infty$, since $\left( V_0 + \sum_{j=k}^{+\infty} V_j\right) \to V_0$, passing to the limsup in the above estimate and using \cref{lem:ProfileUSC} imply \eqref{eq:InfinityInequalityIsopProfile}.
\end{proof}

\begin{remark}[On the hypotheses in \cref{prop:ComparisonIsoperimetricProfile}]\label{rem:HypothesesProp4.2}
We remark that with the same proof of \cref{prop:ComparisonIsoperimetricProfile} we can prove a more general statement substituting $M^n$ with an arbitrary $\RCD(k,n)$ space $(X,\dist,\mathcal{H}^n)$ that satisfies $\mathcal{H}^n(B_1(x))>v_0$ for every $x\in X$ and for some $v_0>0$. 

\end{remark}

We recall that a manifold $(M^n,g)$ is Cartan--Hadamard if it is complete, $\sect\leq 0$ and $M^n$ is simply connected. Recall that if $(M^n,g)$ is Cartan--Hadamard, then $M^n$ is diffeomorphic to $\R^n$.

\begin{cor}\label{thm:CartanHadamard}
Let $(M^n,g)$ be a Cartan--Hadamard manifold of dimension $2\leq n\leq 4$ such that there exists $k\in(-\infty, 0)$ for which $\ric\geq (n-1)k$ on $M^n$, and such that there exists a diverging sequence $p_j\in M^n$ for which $(M^n,\dist,p_j)\to (\mathbb R^n,\dist_{\mathrm{eu}},0)$ in the pGH sense as $j\to+\infty$. Then
\[
I_{(M^n,g)} (V) = I_{(\mathbb R^n,g_{\mathrm{eu}})}(V),
\]
for any $V\ge0$.
Moreover, if there exists an isoperimetric region $\Omega$, then $(\Omega,g)$ is isometric to a Euclidean ball of the same volume.
\end{cor}

\begin{proof}
We recall the following result, which is due to Croke, see \cite[Proposition 14]{CrokeInequality80}.
Let $(M^n,g)$ be a complete Riemannian manifold. Then there exists $C = C(n)>0$ such that 
$$
\vol(B_r(p))\geq Cr^n, \qquad \text{for all $p\in M^n$, and for all $0<r<\inj(p)/2$.}
$$
Since $M^n$ is Cartan--Hadamard, for every $p\in M^n$ we have $\inj(p)=+\infty$. Hence we deduce that $M^n$ is noncollapsed. Thus, as a consequence of the volume convergence in \cref{thm:volumeconvergence}, we get that the pGH limit in the statement is actually a pmGH limit. Hence, from \cref{prop:ComparisonIsoperimetricProfile} we directly get that $I_{(M^n,g)}\leq I_{(\mathbb R^n,g_{\mathrm{eu}})}$. In case $2\leq n\leq 4$ a sharp isoperimetric inequality, i.e., with the Euclidean constant, is available for Cartan--Hadamard manifold, see \cite[p. 1]{WeilIsoperimetricaSuperficiCH} for the case $n=2$, \cite{Kleiner92} for the case $n=3$, and \cite{Croke84} for $n=4$. In particular in all the latter cases, denoting with $\omega_n$ the volume of the unit ball in $\mathbb R^n$, one has that $P(\Omega)\geq n\omega_n^{1/n}(\vol\Omega)^{(n-1)/n}$ for every finite perimeter set $\Omega\subset M^n$, and thus $I_{(M^n,g)}\geq I_{\mathbb (\mathbb R^n,g_{\mathrm{eu}})}$ when $2\leq n\leq 4$. As a result $I_{(M^n,g)}= I_{(\mathbb R^n,g_{\mathrm{eu}})}$ when $2\leq n\leq 4$.

If $\Omega$ is an isoperimetric region, since $2\leq n\leq 4$, we conclude that $\Omega$ is smooth (see \cite[Proposition 2.4]{RitRosales04}, or \cite{morgan2003regularity}). Thus, the rigidity  for the isoperimetric inequalities proved in \cite[p. 1]{WeilIsoperimetricaSuperficiCH}, \cite{Kleiner92}, and \cite{Croke84}, implies that every isoperimetric region $\Omega$ is isometric to a Euclidean ball of the same volume, thus completing the proof of the theorem.
\end{proof}


The argument that follows, providing examples of nonexistence of isoperimetric sets, is inspired by the parallel situation described in \cite[Example 5.6 and Example 5.7]{MondinoSpadaro} constituted by the isoperimetric-isodiametric problem.

\begin{example}[Nonexistence of isoperimetric sets]\label{ex:NonExistence}
An example of $2$-dimensional manifold satisfying the hypotheses of \cref{thm:CartanHadamard} is the helicoid. Indeed, the helicoid is simply connected, $\sect\leq 0$, being a minimal surface, and it can be readily checked that its sectional curvature tends to zero as the distance from the rotation axis increases. Taking into account the periodicity of the helicoid along its rotation axis, then $\ric \ge k$. An easy application of \cref{lem:ComparisonMetrics} shows that a sequence of points $p_j$ whose distance from the rotation axis diverges satisfies the hypotheses of \cref{thm:CartanHadamard}. Since also $\sect \neq 0$ at every point, no isoperimetric regions exist on the helicoid.

Moreover, if $(\Sigma,g)$ is a Cartan--Hadamard surface with induced distance $\dist$ and with asymptotically vanishing sectional curvature (see \cref{def:AVSC} below for the precise definition) then \cref{thm:CartanHadamard} allows us to conclude that $I_{(\Sigma,g)}=I_{(\mathbb R^2,g_{\mathrm{eu}})}$. Indeed since the sectional curvature is asymptotically vanishing, then $(\Sigma,g)$ clearly satisfies a uniform lower bound on the Ricci tensor; moreover, since the sectional curvature is asymptotically vanishing and $\inj(p)=+\infty$ for every $p\in M^n$, the result in \cref{prop:C0Asymp}  below allows to conclude that for every diverging sequence $p_j\in \Sigma$ we have
$$
(\Sigma,\dist,p_j)\xrightarrow{j\to+\infty} (\mathbb R^n,d_{\mathrm{eu}},0),
$$
in the pGH topology.
Thus all the hypotheses of \cref{thm:CartanHadamard} are satisfied and we get the sought equality.
Moreover, if in addition $\sect= 0$ at most at isolated points of $\Sigma$, we conclude from the rigidity part of \cref{thm:CartanHadamard} that no isoperimetric regions of any volume can exist on $\Sigma$. An example satisfying the previous conditions is the saddle, i.e., the surface of equation $z=x^2-y^2$ in $\mathbb R^3$.

In order to construct examples that satisfy the hypotheses of \cref{thm:CartanHadamard} in dimension $n>2$, one can take $(\Sigma,g)$ an arbitrary Cartan--Hadamard surface with Ricci uniformly bounded below satisfying the pGH-limit hypothesis in \cref{thm:CartanHadamard} (e.g., the previously discussed helicoid and saddle), and consider $\Sigma\times \mathbb R$, and $\Sigma\times \mathbb R^2$. Moreover, if one chooses $\Sigma$ such that $\sect= 0$ at most at isolated points of $\Sigma$, then $\Sigma\times\mathbb R$ and $\Sigma\times\mathbb R^2$ cannot have isoperimetric regions of any volume, since rigidity in \cref{thm:CartanHadamard} holds.
\end{example}

We mention another related class of Riemannian manifolds such that no isoperimetric regions exist, in addition to the Cartan--Hadamard manifolds above. These are studied in \cite{Rit01NonExistence} and consist of particular radial metrics of the form $\d t^2 + f(t)^2\d\theta^2$, where $\d\theta^2$ is the metric of the unit circle on $\R^2$. In such a case $\sect$ is a function of $t$, and if $t\mapsto \sect(t)$ is increasing and $\sup \sect$ is never achieved on the surface, then no isoperimetric regions exist \cite[Theorem 2.16]{Rit01NonExistence}.

\section{Asymptotic mass decomposition of minimizing sequences}\label{sec:MassDecomposition}

This section is devoted to the proof of the main result of the work, that yield an asymptotic description of the behavior of minimizing sequences (for the perimeter) that possibly lose part of the mass at infinity, culminating in \cref{thm:MassDecomposition}, that constitutes a more detailed version of \cref{thm:MassDecompositionINTRO}.

\smallskip

The starting point is a classical result due to Ritoré--Rosales that can be found in \cite[Theorem 2.1]{RitRosales04}, and which is meaningful for noncompact Riemannian manifolds of infinite volume.

\begin{thm}\label{thm:RitoreRosales}
Let $(M^n,g)$ be a complete noncompact Riemannian manifold, and fix $V>0$ and $o\in M^n$. Let $\{\Omega_i\}_{i\in\mathbb N}$ be a minimizing (for the perimeter) sequence of finite perimeter sets of volume $V$. Then there exists a diverging sequence $\{r_i\}_{i\in\mathbb N}$ such that
\begin{itemize}
    \item[(i)] $\Omega_i^c:=\Omega_i\cap B_{r_i}(o)$ and $\Omega_i^d:=\Omega_i\setminus B_{r_i}(o)$ are sets of finite perimeter with 
    $$
    \lim_{i\to+\infty}\left(P(\Omega_i^d)+P(\Omega_i^c)\right)=I(V).
    $$
    \item[(ii)] There exists a finite perimeter set $\Omega$ with $\vol(\Omega)\leq V$ such that 
    $$
    \lim_{i\to +\infty}\vol(\Omega_i^c)=\vol(\Omega), \qquad \lim_{i\to +\infty}P(\Omega_i^c)=P(\Omega).
    $$
    Moreover $\Omega^c_i \to \Omega$ in $L^1_{\rm loc}(M^n,g)$.
    \item[(iii)] $\Omega$ is an isoperimetric region for its own volume.
\end{itemize}
\end{thm}

We will need another classical and fundamental property of isoperimetric regions. In \cref{iso-bounded-general} we prove that the validity of an isoperimetric inequality for small volumes implies that isoperimetric regions are bounded.
Interestingly, noncollapsed manifolds with Ricci curvature bounded from below satisfy such an isoperimetric inequality.  This follows from \cite[Lemma 3.2]{Heb00}. Thus, such manifolds have bounded isoperimetric regions and we can state the following result.

\begin{cor}\label{cor:IsopBounded}
Let $(M^n,g)$ be a complete noncollapsed Riemannian manifold with $\ric \geq (n-1)k$ for some $k \in (- \infty, 0]$. Then the isoperimetric regions of $(M^n,g)$ are bounded.
\end{cor}

\subsection{Concentration lemmas}

The following lemma contains a so-called concentration-compactness result that will play a key role in the study of the decomposition of the diverging mass of minimizing sequences. The result is rather classical and could be stated at the level of measure theory, however we include here a brief proof specializing the concentration-compactness principle to a sequence of sets under the form we will apply it. The following result is inspired by \cite[Lemma I.1]{Lions84I}.

\begin{lemma}[Concentration-compactness]\label{lem:CoCo}
Let $(M^n,g)$ be a complete noncompact Riemannian manifold and let $E_i$ be a sequence of bounded measurable sets such that $\lim_i \vol(E_i) = W \in(0,+\infty)$. Then, up to passing to a subsequence, exactly one of the following alternatives occur.
\begin{enumerate}
    \item\label{it:COCOvanishing} For any $R>0$ it holds
    \[
    \lim_i \sup_{p\in M} \vol(E_i \cap B_R(p)) =0.
    \]
    
    \item\label{it:COCOCompactness} There exists a sequence of points $p_i \in M^n$ such that for any $\eps\in(0,W/2)$ there exist $R\ge 1$, $i_\eps \in \N$ such that $\vol(E_i \cap B_R(p_i)) \ge W-\eps$ for any $i \ge i_\eps$. Moreover, there is $I\in \N, r\ge 1$ such that $\vol(E_i \cap B_r(p_i)) \ge \vol(E_i \cap B_r(q))$ for any $q \in M^n$ and $\vol(E_i \cap B_r(p_i)) > W/2$ for any $i \ge I$.
    
    
    \item \label{it:COCODicotomia}
    There exists $w\in(0,W)$ such that for any $\eps\in(0,w/2)$ there exist $R\ge1$, $i_\eps \in \N$, a sequence of points $p_i \in M^n$, and a sequence of open sets $U_i$ such that
    \[
    U_i = M^n \setminus B_{R_i}(p_i) \quad\text{for some $R_i\to + \infty$},\,\,\text{and then}\,\, \dist(p_i,U_i)\xrightarrow{i\to+\infty}+\infty,
    \]
    and moreover
    \begin{equation*}
        \begin{split}
            |\vol(E_i \cap B_R(p_i)) - w | < \eps,& \\
            |\vol(E_i \cap U_i) - (W-w)| < \eps, &\\
             \vol(E_i \cap B_R(p_i)) \ge \vol(E_i \cap B_R(q))& \qquad \forall\,q \in M,
        \end{split}
    \end{equation*}
    for every $i \ge i_\eps$.
\end{enumerate}
\end{lemma}

\begin{proof}
Define $Q_i(\rho)\eqdef \sup_{p\in M} \vol(E_i\cap B_\rho(p))$. The functions $Q_i:(0,+\infty)\to \R$ are nondecreasing and uniformly bounded, since $\vol(E_i)\to W$. Hence the sequence $Q_i$ is uniformly bounded in $BV_{\rm loc}(0,+\infty)$ and then, up to subsequence, there exists a nondecreasing function $Q \in BV_{\rm loc}(0,+\infty)$ such that $Q_i\to Q$ in $BV_{\rm loc}$ and pointwise almost everywhere. Also, let us pointwise define $Q(\rho)\eqdef \lim_{\eta\to0^+} {\rm ess}\inf_{(\rho-\eta,\rho)} Q$, so that $Q$ is defined at every $\rho \in (0,+\infty)$. Moreover, observe that $Q(\rho) \le W$ for any $\rho>0$. Now three disjoint cases can occur, distinguishing the cases enumerated in the statement.
\begin{enumerate}
\item We have that $\lim_{\rho\to+\infty} Q(\rho)=0$, and hence $Q\equiv 0$ since it is nondecreasing. Then item 1 of the statement clearly holds.

\item We have that $\lim_{\rho\to+\infty} Q(\rho)=W$. Then there is $r\ge1$ such that $\exists \lim_{i} \sup_p \vol(E_i \cap B_r(p)) = Q(r) \ge \tfrac34 W$. Since $E_i$ is bounded for any $i$, let $p_i\in M^n$ such that $\sup_p \vol(E_i \cap B_r(p))= \vol(E_i \cap B_r(p_i))$ for any $i$. We claim that the sequence $p_i$ satisfies the property in item 2. Indeed, let $\eps \in (0,W/2)$ be given. Arguing as above, since $\lim_{\rho\to+\infty} Q(\rho)=W$, there is a radius $r'>0$ and a sequence $p_i'\in M^n$ such that $\vol(E_i \cap B_{r'}(p_i') ) \ge W-\eps$ for any $i\ge i_\eps$. Then $\dist(p_i,p_i') < r + r'$, for otherwise
\[
W \xleftarrow{} \vol(E_i) \ge \vol(E_i \cap B_r(p_i)) + \vol (E_i \cap B_{r'}(p_i')) ,
\]
and the right hand side is $> W$ for $i$ large enough. Hence taking $R=r+2r'$ we conclude that $\vol ( E_i \cap B_R(p_i)) \ge W- \eps$ for $i\ge i_\eps$ as claimed.


\item We have that $\lim_{\rho\to+\infty} Q(\rho)=w\in(0,W)$. Then for given $\eps\in (0,w/2)$ there is $R\ge 1$ such that
\[
w-\frac\eps8 \le Q(R) = \lim_i \sup_p \vol(E_i \cap B_R(p)) = \lim_i \vol(E_i \cap B_R(p_i)),
\]
for some $p_i \in M^n$, where in the last equality we used that
\[
\sup_p \vol(E_i \cap B_R(p)) = \vol(E_i \cap B_R(p_i)),
\]
for some $p_i$ since $E_i$ is bounded. This implies that $\vol(E_i \cap B_R(p_i)) \ge \vol(E_i \cap B_R(q))$ for any $i$ and any $q \in M^n$, and there is $i_\eps$ such that $|\vol(E_i \cap B_R(p_i)) - w | < \eps/4$ for $i\ge i_\eps$.

For $i\ge i_\eps$, there is an increasing sequence $\rho_j\to+\infty$ such that $Q(\rho_j)=\lim_i Q_i(\rho_j)$ and we have
\begin{equation*}
    \begin{split}
        w & = \lim_{j\to+\infty} Q(\rho_j) = \lim_j \lim_i \sup_p \vol(E_i \cap B_{\rho_j}(p)) \ge \limsup_j \limsup_i \vol(E_i \cap B_{\rho_j}(p_i)) \\
        & = \limsup_j \limsup_i \left(\vol(E_i \cap B_{R}(p_i)) + \vol(E_i \cap B_{\rho_j}(p_i) \setminus B_R(p_i))\right) \\
        & \ge w- \frac\eps4 + \limsup_j \limsup_i  \vol(E_i \cap B_{\rho_j}(p_i) \setminus B_R(p_i)) .
    \end{split}
\end{equation*}
Then there is $j_0$ such that for any $j\ge j_0$ we have that $\rho_j>R$ and there is $i_j$, with $i_j$ increasing to $+\infty$ as $j\to+\infty$, that satisfies
\begin{equation}\label{eq:CCstima1}
\vol(E_i \cap B_{\rho_j}(p_i) \setminus B_R(p_i)) <\frac\eps2  \qquad \forall\, i \ge\max\{i_\eps,i_j\}.
\end{equation}
Hence define
\[
R_i \eqdef \rho_{\max\{j \st i \ge i_j\}}.
\]
In this way $\vol(E_i \cap B_{R_i}(p_i) \setminus B_R(p_i))  <\eps/2 $ for any $i\ge \max\{i_\eps,i_{j_0}\}$ by \eqref{eq:CCstima1}. Defining $U_i \eqdef M^n \setminus B_{R_i}(p_i)$ we finally get that $\dist(p_i,U_i)=R_i\to+\infty$ and
\[
\begin{split}
W \xleftarrow{} \vol(E_i) &= \vol(E_i \cap B_R(p_i)) + \vol( E_i \cap  B_{R_i}(p_i) \setminus B_R(p_i)) + \vol(E_i \cap U_i ) \\
&\le w + \frac34\eps + \vol(E_i \cap U_i ),
\end{split}
\]
for $i\ge \max\{i_\eps,i_{j_0}\}$. By the first line in the above identity, recalling that $|\vol(E_i\cap B_R(p_i))-w|<\varepsilon/4$, we also see that $\limsup_i \vol(E_i \cap U_i ) \le W-w+\epsilon/4$. Hence the proof of item 3 is completed renaming $\max\{i_\varepsilon,i_{j_0}\}$ into $i_\varepsilon$ and by eventually taking a slightly bigger $i$ in order to ensure the validity of the second inequality of item 3.
\end{enumerate}
\end{proof}

We briefly recall here a useful covering lemma for complete Riemannian manifolds with Ricci curvature bounded from below, cf. \cite[Lemma 1.1]{Heb00}.

\begin{lemma}\label{lem:Covering}
Let $k\in\mathbb R$ and let $(M^n,g)$ be a complete Riemannian manifold such that $\ric\geq (n-1)k$. Let $0<\rho<T_k$, where $T_k:=\pi/\sqrt{k}$ if $k>0$, or $T_k:=+\infty$ otherwise.
Then, there exists a countable family $\{B_{\rho}(x_i)\}_{i\in\mathbb N}$ of open balls such that
\begin{itemize}
    \item[(i)] $\cup_{i\in\mathbb N} B_\rho(x_i) = M^n$,
    \item[(ii)] $B_{\rho/2}(x_i)\cap B_{\rho/2}(x_j)=\emptyset$ for every $i,j\in\mathbb N$,
    \item[(iii)] for every $y\in M^n$ it holds 
    $$
    \sharp\{i:y\in B_{\rho}(x_i)\}\leq
    \sharp\{i:y\in B_{2\rho}(x_i)\}\leq \frac{v(n,k,6\rho)}{v(n,k,\rho/2)}.
    $$
\end{itemize}
\end{lemma}

\begin{proof}
Let $\mathscr F$ be the collection of the countable families of pairwise disjoint balls $\{B_{\rho/2}(x_i):x_i\in M\}_{i\in\mathbb N}$ ordered with the relation $\subset$. By Zorn Lemma it is immediate to deduce the existence of a maximal, with respect to $\subset$, family $\mathscr G:=\{B_{\rho/2}(x_i):x_i\in M\}_{i\in\mathbb N}$ in $\mathscr F$. We want to show that $\mathscr G$ verifies the claims.

Item (ii) for the family $\mathscr G$ is verified by definition. 
Suppose by contradiction item (i) is false. Thus there exists $x\in M $ such that for every $i\in\mathbb N$ we have $d(x,x_i)\geq \rho$. Then, by the triangle inequality, we get that $B_{\rho/2}(x)\cap B_{\rho/2}(x_i)=\emptyset$ for all $i\in\mathbb N$. Thus $\mathscr G \cup \{B_{\rho/2}(x)\}$ is an element of $\mathscr F$ that strictly contains $\mathscr G$, giving a contradiction with the fact that $\mathscr G$ is maximal with respect to $\subset$.

In order to prove item (iii) for the family $\mathscr G$ let us first prove that the number $n$ of disjoint balls $B_{\rho/2}(\widetilde x_1),\dots,B_{\rho/2}(\widetilde x_n)$ that are contained in $B_{3\rho}(x)$, where $x,\widetilde x_1,\dots,\widetilde x_n\in M^n$, is bounded above by $v(n,k,6\rho)/v(n,k,\rho/2)$. Indeed, calling $B_{\rho/2}(\widetilde x_{i_0})$ one of the balls with the minimum volume among $B_{\rho/2}(\widetilde x_1),\dots,B_{\rho/2}(\widetilde x_n)$, we can estimate
$$
n\leq \frac{\vol(B_{3\rho}(x))}{\vol(B_{\rho/2}(\widetilde x_{i_0}))}\leq \frac{\vol(B_{6\rho}(\widetilde x_{i_0}))}{\vol(B_{\rho/2}(\widetilde x_{i_0}))}\leq \frac{v(n,k,6\rho)}{v(n,k,\rho/2)},
$$
where in the first inequality we are using that $B_{\rho/2}(\widetilde x_1),\dots,B_{\rho/2}(\widetilde x_n)$ are disjoint and contained in $B_{3\rho}(x)$, and $B_{\rho/2}(\widetilde x_{i_0})$ is one of the balls with the minimum volume among them; in the second inequality we are using $B_{3\rho}(x)\subset B_{6\rho}(\widetilde x_{i_0})$ by the triangle inequality; and in the third inequality we are using Bishop--Gromov volume comparison (see \cref{thm:BishopGromov}).

Thus the claim is proved. In order to conclude the proof of item (iii), let  $y\in M^n$ be an element of $n$ balls $B_{2\rho}(x_1)$, $\dots$, $B_{2\rho}(x_n)$ of the family $\mathscr G$ constructed above. Then, by the triangle inequality, $B_{\rho/2}(x_i)\subset B_{3\rho}(y)$ for every $1\leq i\leq n$. Since $B_{\rho/2}(x_1),\dots,B_{\rho/2}(x_n)$ are disjoint and contained in $B_{3\rho}(y)$ and since $y,x_1,\dots,x_n\in M^n$, the previous discussion implies that $n\leq v(n,k,6\rho)/v(n,k,\rho/2)$. As also $\{i:y\in B_{\rho}(x_i)\}\subset\{i:y\in B_{2\rho}(x_i)\}$, the proof of item (iii) is concluded.
\end{proof}

We can now deduce a lower bound on the concentration of the mass of a finite perimeter set. The following result is a simpler version of \cite[Lemma 2.5]{Nar14}.

\begin{lemma}[Local mass lower bound]\label{lem:MasLowBound}
Let $k\in\R$ and let $(M^n,g)$ be a complete Riemannian manifold such that $\ric\geq (n-1)k$. Assume that $(M^n,g)$ is noncollapsed with $\vol(B_1(q))\ge v_0>0$ for any $q \in M^n$. Then there exists a constant $C_{n,k,v_0}>0$ such that for any nonempty bounded finite perimeter set $E$ there exists $p_0\in M^n$ such that
\[
\vol(E\cap B_1(p_0) ) \ge \min \left\{C_{n,k,v_0} \frac{\vol(E)^n}{P(E)^n}, \frac{v_0}{2} \right\}.
\]
\end{lemma}

\begin{proof}
Without loss of generality we can assume that $k \le 0$.
We distinguish two possible cases. If there is $p_0 \in M^n$ such that $\vol(E \cap B_1(p_0)) \ge \tfrac12 \vol(B_1(p_0))$, then clearly $\vol(E \cap B_1(p_0)) \ge v_0/2$ and we already have a lower bound. So suppose instead that
\begin{equation}\label{eq:MLB1}
\vol(E \cap B_1(p)) < \frac12 \vol(B_1(p)) 
\qquad
\forall\, p \in M.
\end{equation}
We apply \cref{lem:Covering} with $\rho =1$, which yields a covering $\{ B_1(x_i)\}_{i \in \N}$. Since $E$ is bounded, there is $i_0$ such that
\[
L \eqdef \sup_{i\in\N} \vol(E \cap B_1(x_i))^{\frac1n} = \vol(E \cap B_1(x_{i_0}))^{\frac1n}.
\]
By \eqref{eq:MLB1} we can apply the relative isoperimetric inequality in balls contained in \cite[Corollaire 1.2]{MaheuxSaloffCoste95}. This immediately gives that
\begin{equation}\label{eq:MLB2}
    \vol(E \cap B_1(p))^{\frac{n-1}{n}} \le c\, \mathcal{H}^{n-1}(\partial^* E \cap B_1(p))
    \qquad
\forall\, p \in M,
\end{equation}
where $c=c(n,k,v_0)$.
Therefore, using \eqref{eq:MLB2} and item (iii) in \cref{lem:Covering} we can estimate
\begin{equation*}
    \begin{split}
        \vol (E)
        & \le \sum_i \vol(E \cap B_1(x_i)) = \sum_i \vol(E \cap B_1(x_i))^{\frac1n} \vol(E \cap B_1(x_i))^{\frac{n-1}{n}} \\
        &\le L \sum_i c\, \mathcal{H}^{n-1}(\partial^* E \cap B_1(x_i)) \le L c \,\frac{v(n,k,6)}{v(n,k,1/2)} P(E) ,
    \end{split}
\end{equation*}
that is
\[
\vol( E \cap B_1(x_{i_0})) = L^n \ge C_{n,k,v_0}\frac{\vol(E)^n}{P(E)^n},
\]
where $C_{n,k,v_0} \eqdef \left(v(n,k,1/2)/(c\,v(n,k,6)) \right)^n$.
\end{proof}

\subsection{Asymptotic mass decomposition}
We are now ready to prove the following key result that has to be read as a generalization of \cite[Theorem 2]{Nar14}. Indeed, roughly speaking, we are going to prove that whenever a complete noncompact noncollapsed Riemannian manifold with a lower bound on $\ric$ is given, then the diverging part $\Omega_i^d$ of any perimeter-minimizing sequence, see \cref{thm:RitoreRosales}, can be splitted in different sets that convergence, in volume and perimeter, to isoperimetric regions in some pmGH limits at infinity. 

We thus recover, in the weaker setting of Gromov--Hausdorff convergence, the statement of \cite[Theorem 2]{Nar14}, except from the precise bound on the number of regions that go to infinity contained in \cite[item (X) of Theorem 2]{Nar14}\footnote{Such an upper bound has been eventually provided in the subsequent paper \cite{AntonelliNardulliPozzetta}.}, without asking anything a priori on the geometry at infinity of the manifold. For the proof we are inspired by the strategies of \cite{Nar14}, even though our reasoning is somewhat different as it heavily exploits the results from the nonsmooth theory discussed in \cref{sub:FiniteGH}.

\begin{thm}[Asymptotic mass decomposition]\label{thm:MassDecomposition}
Let $k\in\R$ and let $(M^n,g)$ be a complete noncompact Riemannian manifold such that $\ric\geq (n-1)k$. Assume that $(M^n,g)$ is noncollapsed with $\vol(B_1(q))\ge v_0>0$ for any $q \in M^n$.

Let $\{\Omega_i\}_{i\in\mathbb N}$ be a minimizing (for the perimeter) sequence of finite perimeter sets of volume $V>0$, assume that $\Omega_i$ is bounded for any $i$, and let $\Omega_i^c,\Omega_i^d$ be as in \cref{thm:RitoreRosales}.

If $\lim_i \vol(\Omega_i^d) = W>0$, then, up to subsequence, there exist an increasing sequence of natural numbers $N_i\ge1$, a sequence of points $p_{i,j} \in M^n$ for $j=1,\ldots,N_i$, a sequence of radii $T_{i,j} \ge 1$ for $j=1,\ldots,N_i$ verifying the following properties.
\begin{itemize}
    \item[(i)] Letting $\overline{N}\eqdef \lim_i N_i \in \N \cup \{+\infty\}$, we have
    \begin{equation}\label{eq:AsympMassDecomp1}
    \begin{split}
        \lim_i \dist(p_{i,j},q)=+\infty & \qquad \forall\,q\in M,\forall\,j < \overline{N}+1,\\
        \lim_i \dist(p_{i,j},p_{i,k})=+\infty & \qquad \forall\,j\neq k < \overline{N}+1,\\
        B_{T_{i,j}}(p_{i,j}) \cap B_{T_{i,k}}(p_{i,k}) = \emptyset & \qquad \forall\,i\in \N, \forall\, j\neq  k\leq N_i, \\
        &\qquad j\neq \overline{N}, k \neq \overline{N} ,\\
        \lim_i T_{i,j} = T_j < + \infty & \qquad \forall\,j < \overline{N}, \\
        \text{if also $\overline{N}<+\infty$, then $\lim_i T_{i,\overline{N}} = + \infty$ and }&\\ 
        \partial B_{T_{i,\overline{N}}}(p_{i,\overline{N}}) \cap \partial B_{T_{i,j}}(p_{i,j}) =\emptyset& \qquad \forall\,i\st N_i=\overline{N}, \forall \,j<\overline{N}.
    \end{split}
    \end{equation}
    
    \item[(ii)] Denoting $ G_i\eqdef B_{T_{i,\overline{N}}}(p_{i,\overline{N}}) \cap \Omega_i^d  \setminus \bigcup_{j=1}^{\overline{N}-1} B_{T_{i,j}}(p_{i,j}) $ if $\overline{N}<+\infty$ and $i$ is such that $N_i=\overline{N}$, it holds that
    \begin{equation*}
        \lim_i P(\Omega_i^d) = \begin{cases}
        \lim_i P(G_i) +\sum_{j=1}^{\overline{N}-1} P( \Omega_i^d \cap B_{T_{i,j}}(p_{i,j}) )
        & \text{ if } \overline{N}<+\infty,\\
        \lim_i \sum_{j=1}^{N_i}  P( \Omega_i^d \cap B_{T_{i,j}}(p_{i,j}) )
        & \text{ if } \overline{N}=+\infty,
        \end{cases}
    \end{equation*}
    
    
    \item[(iii)] For any $j < \overline{N}+1$ there exists an $\RCD((n-1)k,n)$ space, points $p_j \in X_j$ and Borel sets $Z_j\subset X_j$ such that
    \begin{equation}\label{eq:Itemiii1}
        \begin{split}
            (M^n,\dist,\vol,p_{i,j}) \xrightarrow[i]{} (X_j,\dist_j,\mathfrak{m}_j, p_j) & \qquad \text{in the $pmGH$ sense for any $j$},\\
            \Omega_i^d \cap B_{T_{i,j}}(p_{i,j}) \xrightarrow[i]{} Z_j \subset X_j& \qquad \text{in the $L^1$-strong sense for any $j<\overline{N}$}, \\
            \vol(\Omega_i^d \cap B_{T_{i,j}}(p_{i,j})) \xrightarrow[i]{} \mathfrak{m}_j(Z_j)& \qquad \forall \,j<\overline{N} \\
            \lim_i P(\Omega_i^d \cap B_{T_{i,j}}(p_{i,j})) = P_{X_j} (Z_j)& \qquad \forall \,j<\overline{N},
        \end{split}
    \end{equation}
    and if $\overline{N}<+\infty$ then
    \begin{equation}\label{eq:Itemiii2}
        \begin{split}
        G_i
            \xrightarrow[i]{} Z_{\overline{N}} \subset X_{\overline{N}}  & \qquad \text{in the $L^1$-strong sense},\\
            \vol(G_i)
            \xrightarrow[i]{} \mathfrak{m}_{\overline{N}}(Z_{\overline{N}}), & \\
            \lim_i P(G_i)
            = P_{X_{\overline{N}}} (Z_{\overline{N}}),
        \end{split}
    \end{equation}
    where $P_{X_j}$ is the perimeter functional on $(X_j,\dist_j,\mathfrak{m}_j)$, and $Z_j$ is an isoperimetric region in $X_j$ for any $j< \overline{N}+1$. Moreover, the measures $\meas_j$'s are the Hausdorff measures with respect to the distance on the corresponding spaces for any $j<\overline{N}+1$.
    
    \item[(iv)] It holds that
 \begin{equation}\label{eq:Itemiv2}
    I(V) = P(\Omega) + \sum_{j=1}^{\overline{N}} P_{X_j} (Z_j),
    \qquad\qquad
    V=\vol(\Omega) +  \sum_{j=1}^{\overline{N}} \mathfrak{m}_j(Z_j),
    \end{equation}
    where $\Omega=\lim_i \Omega_i^c$ is as in \cref{thm:RitoreRosales}. In particular
    \begin{equation}\label{eq:Itemiv1}
        \begin{split}
            &\lim_i P(\Omega_i^d)  = \sum_{j=1}^{\overline{N}} P_{X_j} (Z_j),
            \qquad\qquad
             W= \sum_{j=1}^{\overline{N}} \mathfrak{m}_j(Z_j).
        \end{split}
    \end{equation}
\end{itemize}
\end{thm}

Let us mention the following useful consequence.

\begin{remark}\label{rem:ProfiloPositivo}
From \cref{thm:MassDecomposition} we deduce that if $(M^n,g)$ is noncollapsed with $\ric \ge (n-1) k$, then $I(V)>0$ for any $V>0$.

Indeed, fix $V>0$ and consider a perimeter-minimizing sequence of bounded sets $\Omega_i$ of volume $V$, so that we can apply \cref{thm:MassDecomposition}. Both on $M$ and on any pmGH-limit $X_j$ that may appear in (iii) in \cref{thm:MassDecomposition} there holds a weak local Poincar\'{e} inequality on balls, with constant depending only on the radius of the chosen ball and on the lower bound $k$ assumed on the Ricci curvature, see \cite[Theorem 1]{Rajala12}. As such inequality implies by approximation a relative isoperimetric inequality, we deduce that any set of finite perimeter with finite positive measure must have strictly positive perimeter. Therefore the first identity in \eqref{eq:Itemiv2} implies that $I(V)>0$.
\end{remark}

We now prove the asymptotic mass decomposition theorem.

\begin{proof}[Proof of \cref{thm:MassDecomposition}]
We divide the proof in several steps.

\begin{itemize}
    \item[Step 1.] Up to passing to a subsequence in $i$, we claim that for any $i$ there exist an increasing sequence of natural numbers $N_i\ge1$ with limit $\overline{N}\eqdef \lim_i N_i \in \N\cup\{+\infty\}$, points $p_{i,1},\ldots, p_{i,N_i} \in M^n$ for any $i$, radii $R_j\ge1$ and numbers $\eta_j\in(0,1]$ defined for $j<\overline{N}$, and, if $\overline{N}<+\infty$, also a sequence of radii $R_{i,\overline{N}}\ge 1$, such that
    \begin{equation}\label{eq:Step1}
        \begin{split}
        \lim_i \dist(p_{i,j},q)=+\infty & \qquad \forall\,q\in M,\forall\,j<\overline N+1  ,\\
             \lim_i \dist(p_{i,j},p_{i,k}) =+\infty& \qquad\forall\,j\neq k<\overline N+1  ,\\
             \dist(p_{i,j},p_{i,k})\ge R_j+R_k+2& \qquad 
              \forall\,i\in \N, \forall\, j\neq  k\leq N_i, \\
              &\qquad j\neq \overline{N}, k \neq \overline{N} ,\\
            \exists \, \lim_i \vol(\Omega_i^d \cap B_{R_j}(p_{i,j})) = w_j' >0  & \qquad \forall\, j<\overline N,\\
             \mathcal{H}^{n-1}(\partial^* \Omega_i^d \cap \partial B_{R_j}(p_{i,j}) ) = 0& \qquad \forall\,i, \forall\, j\le N_i , j \neq\overline N , \\
            \mathcal{H}^{n-1}(\Omega_i^d \cap \partial B_{R_j}(p_{i,j}) ) \le \frac{\eta_j}{2^j} & \qquad \forall\,i, \forall\, j\le N_i , j \neq\overline N, \\
            \text{if also $\overline{N}=+\infty$, then }  W\ge  \lim_i  \sum_{j=1}^{N_i} \vol( \Omega_i^d \cap  B_{R_j}(p_{i,j}) ) & \qquad \forall\, i,\\
            \text{if instead $\overline{N}<+\infty$, then $\lim_i R_{i,\overline{N}} = + \infty$,}&\\ 
             \mathcal{H}^{n-1}(\partial^* \Omega_i^d \cap \partial B_{R_{i,\overline{N}}}(p_{i,\overline N}) ) = 0 & \qquad \forall\,i \st N_i=\overline{N}, \\
            \mathcal{H}^{n-1}(\Omega_i^d \cap \partial B_{R_{i,\overline N}}(p_{i,\overline N}) ) \le \frac{1}{2^{\overline N}} & \qquad \forall\,i \st N_i=\overline{N},\\
            \dist\left(\partial B_{R_{i,\overline{N}}}(p_{i,\overline{N}}) , \partial B_{R_{j}}(p_{i,j}) \right) > 2& \qquad \forall\,i\st N_i=\overline{N}, \forall \,j\neq\overline{N},\\
            W = \lim_i \vol\left(  B_{R_{i,\overline{N}}}(p_{i,\overline{N}}) \cap \left( \Omega_i^d  \setminus \bigcup_{j=1}^{\overline{N}-1} B_{R_j}(p_{i,j}) \right)\right) & + \sum_{j=1}^{\overline{N}-1} \vol(  \Omega_i^d \cap  B_{R_j}(p_{i,j})) .\\
        \end{split}
    \end{equation}
    
   We first briefly explain how the proof of this step proceeds.
    We are going to produce the claimed points and radii by induction, with respect to $j$, applying \cref{lem:CoCo}. We will prove that each time we apply \cref{lem:CoCo} on some set during the proof of this step, we will never end up in \cref{it:COCOvanishing}. As a first step, we shall apply \cref{lem:CoCo} on $E_i= \Omega_i^d$. If \cref{it:COCOCompactness} occurs, then we will show that $\overline{N}=1=N_i$ for any $i$; indeed, \cref{it:COCOCompactness} yields a sequence of points $p_{i,1}$ and a diverging sequence of radii $R_{i,1}$ such that all the mass $W$ eventually concentrates in the sequence of balls $B_{R_{i,1}}(p_{i,1})$. Moreover $p_{i,1}$ diverges at infinity (as $\Omega_i^d$ does), we do not construct other sequences of points, and all the identities in \eqref{eq:Step1} can be realized by appropriately choosing $R_{i,1}$. If instead \cref{it:COCODicotomia} occurs, then \cref{it:COCODicotomia} yields the first sequence of points $p_{i,1}$ and a radius $R_1$ such that a certain amount $w_1'>0$ of mass eventually concentrates in the balls $B_{R_{1}}(p_{i,1})$. Moreover points $p_{i,1}$ diverge at infinity. Now, in this case, we will iterate the construction by applying \cref{lem:CoCo} to the sequence $\Omega_i^d \setminus B_{R_1}(p_{i,1})$. As anticipated, we will see that \cref{it:COCOvanishing} does not occur. If \cref{it:COCOCompactness} occurs, then \emph{for large} $i$ we find a second sequence of points $p_{i,2}$ and diverging radii $R_{i,2}$ such that all the remaining mass eventually concentrates in the sequence of balls $B_{R_{i,2}}(p_{i,2})$. In this case $\overline{N}=2=N_i$ for those $i$ such that $p_{i,2}$ are defined. Moreover, the accurate choice of the radii eventually realizes the relations in \eqref{eq:Step1}.  If instead \cref{it:COCODicotomia} occurs again, then \cref{it:COCODicotomia} yields a second sequence of points $p_{i,2}$ defined \emph{for large} $i$ and a radius $R_2$ such that a new amount $w_2'>0$ of mass eventually concentrates in the balls $B_{R_{2}}(p_{i,2})$. Also the radius $R_2$ can be chosen so that, in relation to the already constructed balls $B_{R_1}(p_{i,1})$, the identities in \eqref{eq:Step1} will be eventually satisfied. In this latter case we iterate the construction again by applying \cref{lem:CoCo} on $\Omega_i^d \setminus \cup_{j=1}^2 B_{R_j}(p_{i,j})$ and so on. Each time we apply \cref{lem:CoCo}, we make sure that the newly constructed sequences of points and radii satisfy the relations prescribed in \eqref{eq:Step1} in relation to the already defined sequences of balls.
    As described above, if \cref{it:COCOCompactness} occurs at some iteration, then the construction stops and $\overline{N}$ equals the number of times the iterations occurred. If instead each application of \cref{lem:CoCo} leads to \cref{it:COCODicotomia}, then $\overline{N}=+\infty$, the construction is iterated infinitely many times, and we end up with countably many sequences of points $p_{i,j}$ and radii $R_j$. These sequences will eventually satisfy \eqref{eq:Step1} because at any iteration the new sequences of points and the new radii are ``coherent'' with the previously constructed, i.e., they satisfy the relations in \eqref{eq:Step1} in relation to the already constructed balls.
    
    Observe that, for given $j\in\mathbb N$, the first index $i$ such that $p_{i,j}$ is defined (if it exists) depends on choosing a large index given by \cref{lem:CoCo} depending on a chosen threshold $\eps$; therefore the sequence $N_i$ is inductively constructed together with the appearance of the sequences $p_{i,j},R_j$.
    
    Now, we can move to the proof. As the first step ($j=1$), we apply \cref{lem:CoCo} on $E_i= \Omega_i^d$. Since $W>0$ and, from item (i) of \cref{thm:RitoreRosales} there exists a constant $C_1>0$ such that $P(\Omega_i^d) \le C_1$, \cref{lem:MasLowBound} implies that \cref{it:COCOvanishing} in \cref{lem:CoCo} does not occur. Indeed by \cref{lem:MasLowBound} we find a sequence $q_i\in M^n$ such that
    \[
    \vol( \Omega_i^d \cap B_1(q_i) ) \ge \min \left\{ \frac{C_{n,k,v_0}}{C_1^n}\left(\frac{W}{2}\right)^n , \frac{v_0}{2}\right\},
    \]
    for any large $i$ such that $\vol(\Omega_i^d) \ge W/2$. Such an estimate would contradict the occurrence of \cref{it:COCOvanishing}.
    
    So suppose that \cref{it:COCODicotomia} in \cref{lem:CoCo} occurs. Then denote by $w_1:=w\in(0,W)$ the number given by \cref{it:COCODicotomia}. Then take
    \begin{equation*}
        \alpha_1\eqdef \min\left\{ \frac{C_{n,k,v_0}}{\overline{C}^n}\left(\frac{W-w_1}{2} \right)^n, \frac{v_0}{2} \right\}, \qquad 
        \eps_1 < \frac13\frac{\eta_1}{2^2}\eqdef \frac13 \frac{1}{2^2}\min \left\{1, \alpha_1 , \frac{w_1}{2} \right\},
    \end{equation*}
    where $C_{n,k,v_0}$ is as in \cref{lem:MasLowBound} and $\overline{C} \eqdef C_1+ 2 \ge P(\Omega_i^d) + 2$ for any $i$. Hence let $p_{i,1},R^*_1\ge 1$ be given by \cref{it:COCODicotomia} applied with $\eps=\eps_1$.  Take $R_1\ge R^*_1$ such that $\partial B_{R_1}(p_{i,1})$ is Lipschitz and $\mathcal{H}^{n-1}(\partial^* \Omega_i^d \cap \partial B_{R_1}(p_{i,1}) ) = 0$ for any $i$. Moreover, up to subsequence, we have that there exists $ \lim_i \vol(\Omega_i^d \cap B_{R_1}(p_{i,1})) =: w_1' \in (0,W)$. Also, since $\Omega_i^d\cap \mathcal{C}=\emptyset$ definitely for any compact set $\mathcal{C}$, then $\dist(p_{i,1},q)\to+\infty$ for any fixed $q\in M^n$.\\
    Finally, by \cref{it:COCODicotomia} and since $\Omega_i^d$ is bounded, there is a sequence of open sets $V_i^1$ such that $\dist(B_{R_1}(p_{i,1}),V_i^1)\to+\infty$ and
    \begin{equation}\label{eq:MassaRestante}
    \vol(\Omega_i^d) - \vol(\Omega_i^d \cap B_{R_1}(p_{i,1})) - \vol (\Omega_i^d \cap V_i^1) < 3\eps_1 < \eta_1/2^2,
    \end{equation}
    for $i$ sufficiently large. So, for large $i$, by the coarea formula we can estimate
    \[
    \begin{split}
        \frac{\eta_1}{2^2} > \int_{R_1}^{\dist({p_{i,1}}, V_i^1)} \mathcal{H}^{n-1} (\Omega_i^d \cap \partial B_t (p_{i,1})) \de t > \int_{R_1}^{R_1+1} \mathcal{H}^{n-1} (\Omega_i^d \cap \partial B_t (p_{i,1})) \de t.
    \end{split}
    \]
    Therefore, up to taking a new radius in $(R_1,R_1+1)$, still denoted by $R_1$, we can further ensure that
    \[
    \mathcal{H}^{n-1}(\Omega_i^d \cap \partial B_{R_1}(p_{i,1}) ) \le \frac{\eta_1}{2}.
    \]
    
    If instead \cref{it:COCOCompactness} in \cref{lem:CoCo} occurs, then we take $\overline{N}=1=N_i$ for any $i$ as \cref{it:COCOCompactness} yields sequences $p_{i,1}, R_{i,1}$ such that $\vol( \Omega_i^d \cap B_{R_{i,1}}(p_{i,1}) ) \ge W - 1/i$, up to subsequence. Indeed, arguing as above, also in this case we can ensure all the remaining properties in \eqref{eq:Step1} and we can also take $R_{i,1}\to + \infty$ as $i\to+\infty$.
    
    So we have seen that in case for $j=1$ the alternative in \cref{it:COCODicotomia} occurs, the construction must be iterated. We now show the inductive construction only for the step $j=2$, the passage $j \Rightarrow j+1$ being completely analogous. For $j=2$ we now apply \cref{lem:CoCo} on $E_i= \Omega_i^d \setminus B_{R_1}(p_{i,1})$. Again, since $\vol(\Omega_i^d \setminus B_{R_1}(p_{i,1})) \to W-w_1' >0 $, \cref{it:COCOvanishing} in \cref{lem:CoCo} does not occur because of \cref{lem:MasLowBound}. Indeed, just like we did for $j=1$, a positive lower bound on the volume and the finite upper bound on the perimeter given by $P(\Omega_i^d \setminus B_{R_1}(p_{i,1}))\le P(\Omega_i^d) + \mathcal{H}^{n-1}(\Omega_i^d \cap \partial B_{R_1}(p_{i,1})) \le P(\Omega_i^d) + \eta_1/2 \le \overline{C} $ imply that \cref{it:COCOvanishing} would contradict \cref{lem:MasLowBound}.
    
    So if \cref{it:COCODicotomia} in \cref{lem:CoCo} occurs, then denote by $w_2:=w\in(0,W-w_1')$ the number given by \cref{it:COCODicotomia}. In this case we take
    \begin{equation*}
    \begin{split}
        \alpha_2\eqdef \min\left\{ \frac{C_{n,k,v_0}}{\overline{C}^n}\left(\frac{W-w_1'-w_2}{2} \right)^n, \frac{v_0}{2} \right\}, \qquad \eps_2 < \frac13\frac{\eta_2}{2^3} \eqdef \frac13 \frac{1}{2^3}\min \left\{1, \alpha_2 , \frac{w_2}{2} \right\},&
    \end{split}
    \end{equation*}
    Hence \cref{it:COCODicotomia} gives sequences $p_{i,2}, R^*_2 \ge 1$. As before, let $R_2\ge R^*_2$ such that, up to passing to a subsequence, we have that
    $\partial B_{R_2}(p_{i,2})$ is Lipschitz, $\mathcal{H}^{n-1}(\partial^* \Omega_i^d \cap \partial B_{R_2}(p_{i,2}) ) = 0$ for any $i$, there exists $\lim_i \vol(\left(\Omega_i^d \setminus B_{R_1}(p_{i,1})\right) \cap B_{R_2}(p_{i,2})) = w_2' >0$, and we have that $\dist(p_{i,2},q)\to+\infty$ for any $q \in M^n$. Moreover, the use of the coarea formula as done above now yields
    \[
     \mathcal{H}^{n-1}(\left(\Omega_i^d\setminus B_{R_1}(p_{i,1})\right) \cap \partial B_{R_2}(p_{i,2}) ) \le \frac{\eta_2}{2^2}.
    \]
    We now show that $\dist(p_{i,1},p_{i,2})\to + \infty$, and then, up to subsequence, we can also assume that $\dist(p_{i,1},p_{i,2})\ge R_1+R_2+2$ for any $i$ such that $p_{i,2}$ is defined. Indeed, if $\limsup_i\dist(p_{i,1},p_{i,2})$ is bounded, then
    \[
    \vol(\left(\Omega_i^d\setminus B_{R_1}(p_{i,1})\right) \cap  B_{R_2}(p_{i,2})) \le \vol ( \Omega_i^d\setminus \left( B_{R_1}(p_{i,1}) \cup V_i^1\right)) \le \frac{\eta_1}{2^2},
    \]
    for large $i$.
    
    On the other hand we know that, for large $i$, $\vol (\Omega_i^d \setminus B_{R_1}(p_{i,1}) )\ge W-w_1' + o(1) \ge (W-w_1)/2$, and $P(\Omega_i^d \setminus B_{R_1}(p_{i,1}))\le  \overline{C} $, for large $i$.
    Therefore, using the characterization of $p_{i,2}$ in \cref{it:COCODicotomia} and applying \cref{lem:MasLowBound} on $\Omega_i^d \setminus B_{R_1}(p_{i,1})$, we get for some $q_i\in M^n$ that
    \[
    \begin{split}
        \frac{\eta_1}{2^2} &\ge \vol (\left(\Omega_i^d\setminus B_{R_1}(p_{i,1})\right) \cap  B_{R_2}(p_{i,2}))
        \ge \vol (\left(\Omega_i^d\setminus B_{R_1}(p_{i,1})\right) \cap  B_{R^*_2}(p_{i,2}))\\
        &\ge \vol (\left(\Omega_i^d\setminus B_{R_1}(p_{i,1})\right) \cap  B_{R^*_2}(q_i)) 
        \ge \vol (\left(\Omega_i^d\setminus B_{R_1}(p_{i,1})\right) \cap  B_{1}(q_i)) \\&\ge \min\left\{C_{n,k,v_0}\frac{\vol(\Omega_i^d\setminus B_{R_1}(p_{i,1}))^n}{P(\Omega_i^d\setminus B_{R_1}(p_{i,1}))^n} , \frac{v_0}{2}  \right\} \ge \alpha_1,
    \end{split}
    \]
    for large $i$.
    But since $\eta_1\le \alpha_1$, the above inequality yields a contradiction.\\
    Now since $\dist(p_{i,1},p_{i,2})\to_{i} + \infty$, the above identities simplify into
    \[
    w_2' = \lim_i \vol(\Omega_i \cap B_{R_2}(p_{i,2}) ) , \qquad
    \mathcal{H}^{n-1}(\Omega_i^d \cap \partial B_{R_2}(p_{i,2})) \le \frac{\eta_2}{2^2},
    \]
    up to passing to a subsequence; also, by \cref{it:COCODicotomia}, analogously as in \eqref{eq:MassaRestante} we obtain
    \[
    \begin{split}
    \vol&(\Omega_i^d) - \vol(\Omega_i^d \cap B_{R_1}(p_{i,1})) - \vol(\Omega_i^d \cap B_{R_2}(p_{i,2})) - \vol( \left( \Omega_i^d \setminus B_{R_1}(p_{i,1}) \right) \cap V^2_i)\\
    &
    = \vol(\Omega_i^d \setminus B_{R_1}(p_{i,1}) ) - \vol((\Omega_i^d \setminus B_{R_1}(p_{i,1}))\cap B_{R_2}(p_{i,2})) - \vol( \left( \Omega_i^d \setminus B_{R_1}(p_{i,1}) \right) \cap V^2_i)
    \\
    &\le 3 \eps_2 = \frac{\eta_2}{2^3}, 
    \end{split}
    \]
    for any large $i$ such that $p_{i,2}$ is defined, for a sequence of bounded open sets $V^2_i$ such that $\dist(p_{i,2},V^2_i)\to+\infty$. At this point, the new sequence $p_{i,2}$ and the radii $R_2$ satisfy the conditions prescribed in \eqref{eq:Step1} in relation to the already constructed sequence of balls $B_{R_1}(p_{i,1})$. 
    
    Finally, if instead \cref{it:COCOCompactness} in \cref{lem:CoCo} occurs for $j=2$,
    then \cref{it:COCOCompactness} yields sequences $p_{i,2}, R_{i,2}\ge1$ such that $\vol( B_{R_{i,2}}(p_{i,2}) \cap \left(\Omega_i^d \setminus B_{R_{1}}(p_{i,1})\right)) \ge W - w_1' - 1/i$, up to subsequence for large $i$. Since \cref{it:COCOCompactness} also gives $r\ge 1$ such that $\vol( B_{r}(p_{i,2}) \cap \left(\Omega_i^d \setminus B_{R_{1}}(p_{i,1})\right) ) \ge \vol( B_{r}(q) \cap \left(\Omega_i^d \setminus B_{R_{1}}(p_{i,1}) \right)) $ for any $q \in M^n$, arguing as above one easily gets that $\dist (p_{i,1},p_{i,2})\to +\infty$. Hence $\overline{N}=2=N_i$ for large $i$. Moreover, arguing as in the above step $j=1$, also in this case we can ensure all the remaining properties in \eqref{eq:Step1}, we can also take $R_{i,2}\to + \infty$ as $i\to+\infty$, and assume that $\dist( \partial B_{R_{i,2}}(p_{i,2}) , \partial B_{R_{1}}(p_{i,1}) ) > 2 $ for large $i$.
    
    Now if for $j=2$ \cref{it:COCODicotomia} occurs, one needs to continue the construction for $j=3$. Now one applies \cref{lem:CoCo} on $E_i= \Omega_i^d \setminus (B_{R_1}(p_{i,1}) \cup B_{R_2}(p_{i,2}))$. Once again  \cref{it:COCOvanishing} cannot occur, because of \cref{lem:MasLowBound} and since $\vol (\Omega_i^d \setminus (B_{R_1}(p_{i,1}) \cup B_{R_2}(p_{i,2}))) \to W-w_1'-w_2' >0$. Then it can be checked that the construction inductively proceeds depending on whether \cref{it:COCOCompactness} or \cref{it:COCODicotomia} occurs for $j=3$ as discussed above for $j=2$. Eventually one gets the desired sequences $N_i, p_{i,j}, R_j, \eta_j$ as claimed in \eqref{eq:Step1}.

    \item[Step 2.] We claim that if $\overline{N}=+\infty$ then
    \begin{equation}\label{eq:Step2a}
        W= \lim_i \sum_{j=1}^{N_i} \vol( \Omega_i^d \cap  B_{R_j}(p_{i,j}) ).
    \end{equation}
    Moreover, we claim that, up to passing to a subsequence in $i$, there exist sequences of radii $\{T_{i,j}\}_{i \in \N}$ such that $T_{i,j} \in (R_j,R_j+1)$ for any $j<\overline{N}$, and $T_{i,\overline{N}} \in (R_{i,\overline{N}}, R_{i,\overline{N}}+1)$ if $\overline{N}<+\infty$, such that \eqref{eq:AsympMassDecomp1} holds and
    \begin{equation}\label{eq:Step2b}
        \lim_i \sum_{j=1}^{N_i} \mathcal{H}^{n-1}(\Omega_i^d \cap \partial B_{T_{i,j}}(p_{i,j}) ) = 0.
    \end{equation}
    
    Assume first that $\overline{N}=+\infty$. We observe that, up to passing to a subsequence in $i$, we have
    \[
    W \ge \lim_i \sum_{j=1}^{N_i} \vol( \Omega_i^d \cap  B_{R_j}(p_{i,j}) ) \ge \sum_{j=1}^{M} w'_j,
    \]
    for any $M \in \N$, and then $W\ge \sum_{j=1}^{+\infty} w'_j$.
    Suppose by contradiction that $W>  \lim_i \sum_{j=1}^{N_i} \vol( \Omega_i^d \cap  B_{R_j}(p_{i,j}) ) $, and define
    \[
    \widetilde\Omega_i^v \eqdef \Omega_i^d \setminus \bigcup_{j=1}^{N_i}   B_{R_j}(p_{i,j}).
    \]
    By the absurd hypothesis, up to passing to a subsequence, we have that $\lim_i \vol(\widetilde\Omega_i^v) = \omega>0$ and, by \eqref{eq:Step1}, we estimate
    \[
    \begin{split}
        P(\widetilde\Omega_i^v) & \le P(\Omega_i^d) + \sum_{j=1}^{N_i} \mathcal{H}^{n-1}(\Omega_i^d \cap  \partial B_{R_j}(p_{i,j}) ) \le \overline{C}.
    \end{split}
    \]
    On the other hand, applying \cref{lem:MasLowBound} on $\widetilde\Omega_i^v$ yields
    \[
    \vol(\widetilde\Omega_i^v \cap B_1(q_i))
    \ge \min\left\{C_{n,k,v_0} \frac{\vol(\widetilde\Omega_i^v)^n}{ P(\widetilde\Omega_i^v)^n} ,\frac{v_0}{2}\right\}
    \ge \min \left\{C_{n,k,v_0} \frac{(\omega/2)^n}{ \overline{C}^n} ,\frac{v_0}{2}\right\} \eqqcolon C_\omega,
    \]
    for some $q_i \in M^n$, for large $i$. Hence for large $i$ and for any fixed $j_0\le N_i$ we then have
    \[
    \begin{split}
        C_\omega 
        &\le \vol(\widetilde\Omega_i^v \cap B_1(q_i))
        \le \vol \left( B_1(q_i) \cap \, \Omega_i^d \setminus \bigcup_{j=1}^{j_0-1}  B_{R_j}(p_{i,j})  \right) \\
        &
        \le \vol \left( B_{R^*_{j_0}}(q_i) \cap \, \Omega_i^d \setminus \bigcup_{j=1}^{j_0-1}  B_{R_j}(p_{i,j})  \right)
        \\
        &
        \le \vol \left( B_{R^*_{j_0}}(p_{i,j_0}) \cap \, \Omega_i^d \setminus \bigcup_{j=1}^{j_0-1}  B_{R_j}(p_{i,j})  \right)
        \\
        & \le \vol(\Omega_i^d \cap B_{R_{j_0}}(p_{i,j_0})),
    \end{split}
    \]
    where $R^*_{j_0}\le R_{j_0}$ was determined by the application of \cref{it:COCODicotomia} in the Step 1.
    Since $N_i\to +\infty$, then from the estimate above we would get $+\infty =\sum_{j=1}^{+\infty} w'_j \le W$, that gives a contradiction. Hence \eqref{eq:Step2a} is proved.
    
    Now we prove \eqref{eq:Step2b}. Assume first that $\overline{N}=+\infty$. Then in the above notation, using \eqref{eq:Step1}, in particular the fact that $\dist(p_{i,j},p_{i,k})\geq R_j+R_k+2$, and the coarea formula, we estimate
    \begin{equation}\label{eq:Step2c}
    \begin{split}
        \vol(\widetilde\Omega_i^v) 
        & \ge \sum_{j=1}^{N_i} \int_{R_j}^{R_j+1} \mathcal{H}^{n-1}(\Omega_i^d \cap \partial B_t(p_{i,j})) \de t \ge \frac12  \sum_{j=1}^{N_i}  \mathcal{H}^{n-1}(\Omega_i^d \cap \partial B_{T_{i,j}}(p_{i,j})),
    \end{split}
    \end{equation}
    for some $T_{i,j}\in (R_j,R_j+1)$ for any $j$. Up to subsequence (in $i$) we have that $T_{i,j}\to T_j$ for any $j$, and since $\vol(\widetilde\Omega_i^v) \to 0$ by \eqref{eq:Step2a}, then \eqref{eq:Step2b} follows together with the properties stated in \eqref{eq:AsympMassDecomp1}. If instead $\overline{N}<+\infty$, since
    \[
    \dist\left(\partial B_{R_{i,\overline{N}}}(p_{i,\overline{N}}) , \partial B_{R_{j}}(p_{i,j}) \right) > 2 \qquad \forall\,i\st N_i=\overline{N}, \forall \,j<\overline{N},
    \]
    by \eqref{eq:Step1}, letting now $\widehat\Omega_i^v\eqdef \Omega_i^d \setminus \left(\bigcup_{j=1}^{\overline{N}-1} B_{R_j}(p_{i,j}) \cup B_{R_{i,\overline{N}}}(p_{i,\overline{N}})\right)$, as $\vol(\widehat \Omega_i^v)\to 0$ by the last line in \eqref{eq:Step1}, we can perform an analogous estimate as in \eqref{eq:Step2c}, therefore getting the desired $T_{i,j}$ for any $j\le\overline{N}$ satisfying \eqref{eq:Step2b}. Hence \eqref{eq:AsympMassDecomp1} holds also in this case.

    \item[Step 3.] We claim that letting $\Omega_i^v \eqdef \Omega_i^d \setminus \bigcup_{j=1}^{N_i} \left(\Omega_i^d \cap B_{T_{i,j}}(p_{i,j})\right)$, then
    \begin{equation}\label{eq:Step3a}
            \lim_i \vol(\Omega_i^v)=0,
    \end{equation}
    and that, if $\overline{N}=+\infty$, then
    \begin{equation}\label{eq:Step3b}
            W= \lim_i \sum_{j=1}^{N_i} \vol( \Omega_i^d \cap  B_{T_{i,j}}(p_{i,j}) ) = \sum_{j=1}^{+\infty} \lim_i \vol( \Omega_i^d \cap  B_{T_{i,j}}(p_{i,j}) ).
    \end{equation}
    
    Since $T_{i,j}\ge R_j$ for any $j<\overline{N}$, we have that, if $\overline{N}=+\infty$, then $\Omega_i^v \subset \widetilde\Omega_i^v$, while if $\overline{N}<+\infty$, then analogously $\Omega_i^v \subset \widehat\Omega_i^v$. Hence in any case \eqref{eq:Step3a} follows from \eqref{eq:Step2a}, if $\overline N=+\infty$, or from the last line in \eqref{eq:Step1}, if $\overline N<+\infty$.
    
    Now suppose that $\overline{N}=+\infty$. Since $T_{i,j}\ge R_j$ for any $j$, by \eqref{eq:Step2a} we see that
    \[
     W= \lim_i \sum_{j=1}^{N_i} \vol( \Omega_i^d \cap  B_{R_j}(p_{i,j}) ) \le \lim_i \sum_{j=1}^{N_i} \vol( \Omega_i^d \cap  B_{T_{i,j}}(p_{i,j}) ) \le W.
    \]
    Up to subsequence, denote $\omega_j\eqdef \lim_i \vol( \Omega_i^d \cap  B_{T_{i,j}}(p_{i,j}) )$ for any $j$. By the above identity, we see that $W \ge \sum_{j=1}^{+\infty} \omega_j$, and then $\lim_j \omega_j =0$. In order to prove the second part of \eqref{eq:Step3b}, suppose by contradiction that $\sum_{j=1}^{+\infty} \omega_j = Y <W$. We argue as before considering
    \[
    C^*\eqdef \min \left\{ \frac{C_{n,k,v_0}}{\overline{C}^n} \left( \frac{W-Y}{2}\right)^n , \frac{v_0}{2}\right\}.
    \]
    Let $j^*$ be such that $\omega_j< C^*$ for any $j\ge j^*$. From now on consider $j>j^*$. We clearly have
    \[
    \vol\left(\Omega_i^d \setminus \bigcup_{k=1}^{j-1} B_{R_k}(p_{i,k})  \right) \ge \frac{W-Y}{2},
    \]
    for any large $i$. Moreover
    \[
    P \left(\Omega_i^d \setminus \bigcup_{k=1}^{j-1} B_{R_k}(p_{i,k})  \right) \le P(\Omega_i^d) + \sum_{k=1}^{j-1} \mathcal{H}^{n-1}(\Omega_i^d \cap \partial B_{R_k}(p_{i,k}) ) \le \overline{C},
    \]
    by \eqref{eq:Step1}. On the other hand, applying  \cref{lem:MasLowBound} on $\Omega_i^d \setminus \bigcup_{k=1}^{j-1} B_{R_k}(p_{i,k})$ yields the existence of $q_i\in M^n$ such that
    \[
    \vol \left(B_1(q_i) \cap \,\Omega_i^d \setminus \bigcup_{k=1}^{j-1} B_{R_k}(p_{i,k})  \right) \ge C^*,
    \]
    for any large $i$. As $p_{i,j}$ is obtained by applying \cref{it:COCODicotomia} on $\Omega_i^d \setminus \bigcup_{k=1}^{j-1} B_{R_k}(p_{i,k}) $ and all the produced balls are disjoint, this implies that
    \[
    \vol(\Omega_i^d \cap B_{R_j}(p_{i,j}) ) \ge C^*,
    \]
    for any $j>j^*$ and any $i$ large. Hence $\omega_j\ge C^*$ for any $j>j^*$, and $\sum_{j=1}^{+\infty}\omega_j = +\infty$, yielding a contradiction.

    \item[Step 4.] We claim that
    \begin{equation}\label{eq:Step4a}
        \lim_i P(\Omega_i^v)=0,
    \end{equation}
    and, denoting $ G_i\eqdef B_{T_{i,\overline{N}}}(p_{i,\overline{N}}) \cap \Omega_i^d  \setminus \bigcup_{j=1}^{\overline{N}-1} B_{T_{i,j}}(p_{i,j}) $ if $\overline{N}<+\infty$, that
    \begin{equation}\label{eq:Step4b}
        \lim_i P(\Omega_i^d) = \begin{cases}
        \lim_i \left(P(G_i) +\sum_{j=1}^{\overline{N}-1} P( \Omega_i^d \cap B_{T_{i,j}}(p_{i,j}) )\right)
        & \overline{N}<+\infty,\\
        \lim_i \sum_{j=1}^{N_i}  P( \Omega_i^d \cap B_{T_{i,j}}(p_{i,j}) )
        & \overline{N}=+\infty,
        \end{cases}
    \end{equation}
    We also claim that item (iii) of the statement holds.
    
    In order to prove \eqref{eq:Step4a}, we assume without loss of generality that $\vol(\Omega_i^v)>0$. We assume first that $\overline{N}=+\infty$. By \eqref{eq:Step2b} we have that
    \begin{equation}\label{eq:Step4c}
    \lim_i P(\Omega_i^d) = \lim_i\left( P(\Omega_i^v) + \sum_{j=1}^{N_i} P(\Omega_i^d \cap B_{T_{i,j}}(p_{i,j}) )\right).
    \end{equation}
    If, by contradiction, $\lim_i P(\Omega_i^v)>0$, then we consider the new sequence
    \[
    F_i = \Omega_i^c \cup B_{\rho_i}(q_i) \cup \bigcup_{j=1}^{N_i} \Omega_i^d \cap B_{T_{i,j}}(p_{i,j}) ,
    \]
    where $B_{\rho_i}(q_i)$ is a ball such that $\vol(  B_{\rho_i}(q_i)) = \vol(\Omega_i^v)$ and $ B_{\rho_i}(q_i) \cap \Omega_i = \emptyset$. Observe that such a ball exists since $\Omega_i$ is bounded and $\vol(\Omega_i^v)\to 0$ by \eqref{eq:Step3a}, hence $\rho_i<1$ for large $i$. Actually $\rho_i\to 0$, indeed \cref{thm:BishopGromov} implies that
    \[
    \vol(B_r(q)) \ge v(n,k,r) \frac{\vol(B_1(q))}{v(n,k,1)} \ge \frac{v_0}{v(n,k,1)} v(n,k,r),
    \]
    for any $r\in(0,1)$. Hence $v(n,k,\rho_i)\to 0$ and hence $\rho_i\to 0$. Moreover by \cref{thm:BishopGromov} (together with \cref{rem:PerimeterMMS2}) we have
    \begin{equation}\label{eq:Step4d}
        P(B_{\rho_i}(q_i)) \le s(n,k,\rho_i) \xrightarrow[i]{} 0.
    \end{equation}
    Now observe that by \cref{thm:RitoreRosales} we have that
    \begin{equation*}
         \lim_i P(\Omega_i) = \lim_i \left(P(\Omega_i^c) + P(\Omega_i^d)\right) = \lim_i P(\Omega_i) + 2 \mathcal{H}^{n-1} (\partial B_{r_i}(o) \cap \Omega_i),
    \end{equation*}
    and thus $\lim_i \mathcal{H}^{n-1} (\partial B_{r_i}(o) \cap \Omega_i)=0$. Hence by definition of $F_i$ we can write
    \[
    P(F_i) = \mathcal{H}^{n-1}(\Sigma_i) + P(\Omega_i^c) + P( B_{\rho_i}(q_i) ) + \sum_{j=1}^{N_i} P (  \Omega_i^d \cap B_{T_{i,j}}(p_{i,j}) ) ,
    \]
    where $\Sigma_i \subset \partial B_{r_i}(o) \cap \Omega_i $, and thus $\lim_i \mathcal{H}^{n-1}(\Sigma_i) =0$. Therefore, by \eqref{eq:Step4c}, \eqref{eq:Step4d}, and since $\vol(F_i)= V$, the absurd hypothesis implies
    \[
    I(V) = \lim_i \left(P(\Omega_i^c) + P(\Omega_i^d)\right) > \lim_i P(F_i) \ge I(V),
    \]
    that is a contradiction. Employing the same argument, it is immediate to check that a similar reasoning implies that \eqref{eq:Step4a} holds even in case $\overline{N}<+\infty$. Indeed \eqref{eq:Step4c} still holds, $\Omega_i^d$ is bounded by assumption, and then the suitable new definition of $F_i$ leads to the same conclusion.
    
    So if $\overline{N}<+\infty$, we see that \eqref{eq:Step4a} and \eqref{eq:Step2b} imply the first line in \eqref{eq:Step4b}. If instead $\overline{N}=+\infty$, then \eqref{eq:Step4a} and \eqref{eq:Step4c} imply the second line in \eqref{eq:Step4b}.
    
    It remains to prove the claims in item (iii). By \cref{rem:GromovPrecompactness}, up to passing to a subsequence in $i$ and by a diagonal argument, we immediately have that
    for any $j < \overline{N}+1$ there exist a Ricci limit space $(X_j,\dist_j,\meas_j)$, where $\meas_j$ is the $n$-dimensional Hausdorff measure in $X_j$, which is thus an $\RCD((n-1)k,n)$ space, and points $p_j \in X_j$ such that
    \[
    (M^n,\dist,\vol,p_{i,j}) \xrightarrow[i]{} (X_j,\dist_j,\mathfrak{m}_j, p_j) \qquad \text{in the pmGH sense for any $j< \overline{N}+1$}.
    \]
    Let us deal with the case  $\overline{N}=+\infty$ first.
    
    Recalling for example from \eqref{eq:Step4b} that $P(\Omega_i^d \cap B_{T_{i,j}}(p_{i,j}))$ is uniformly bounded with respect to $i$ for any $j<\overline{N}$, we can directly apply item (a) of \cref{prop:SemicontinuitaAmbrosioBrueSemola} to get the convergence of $\Omega_i^d \cap B_{T_{i,j}}(p_{i,j})$ to some $Z_j\subset X_j$ in the $L^1$-strong sense for any $j<\overline{N}$. Moreover, again from item (a) of \cref{prop:SemicontinuitaAmbrosioBrueSemola}, we get that $\liminf_i P(\Omega_i^d \cap B_{T_{i,j}} (p_{i,j}) ) \geq P_{X_j}(Z_j)$ for every $j<\overline N$.
    
    We now check that $Z_j$ is isoperimetric for its own volume $\meas_j(Z_j)$ in $X_j$ for every $j<\overline{N}$, and that
    \begin{equation}\label{eq:ContinuityPerimeter}
        \begin{split}
    \lim_i P(\Omega_i^d \cap B_{T_{i,j}} (p_{i,j}) ) = P_{X_j}(Z_j),
    \end{split}
    \end{equation}
    for every $j< \overline{N}$.
    
    Since $M^n$ is noncollapsed, by \cref{thm:volumeconvergence} one has that, for some $v_0>0$, $\meas_j(B_1(x))\ge v_0>0$ for any $j$ and $x \in X_j$ (see the argument at the beginning of the proof of \cref{prop:ComparisonIsoperimetricProfile}). So, by \cref{lem:ProfileOnBoundedSets}, we have that if by contradiction for some $j<\overline{N}$ it occurs that either $Z_j$ is not isoperimetric or $\limsup_i P(\Omega_i^d \cap B_{T_{i,j}} (p_{i,j}) ) > P_{X_j}(Z_j)$, there exists a bounded finite perimeter set $W_j \subset X_j$ such that $\meas_j(W_j) = \meas_j (Z_j)$ and, possibly passing to subsequences in $i$, 
    \begin{equation}\label{eq:Step4e}
        \lim_i P(\Omega_i^d \cap B_{T_{i,j}} (p_{i,j}) ) \ge P_{X_j}(W_j) + \eta,
    \end{equation}
    for some $\eta>0$.
    
    By \cite[Theorem 2]{FloresNardulli20} it is known that $I$ is continuous, and thus there is $\eps_0>0$ such that
    \begin{equation}\label{eq:ContinuityProfile}
        |I(V)-I(V-\eps)|<\frac\eta2,
    \end{equation}
    whenever $|\eps|<\eps_0$.
    
    Now by item (c) in \cref{prop:SemicontinuitaAmbrosioBrueSemola}, up to subsequence, there exists a sequence of sets $E_{i,j}$ contained in $B_L(p_{i,j})$ for some $L>0$ such that $E_{i,j}$ converges in $L^1$-strong to $W_j$ and $\lim_i P(E_{i,j}) = P_{X_j}(W_j).$
    
    Moreover by \cref{thm:RitoreRosales} we know that $\Omega_i^c \to \Omega$ with $P(\Omega_i^c)\to P(\Omega)$, and $\Omega$ is an isoperimetric region on $(M^n,g)$. Hence $\Omega$ is bounded by \cref{cor:IsopBounded}. So for large $i$ there is $S>0$ such that $\Omega \Subset B_S(o) \Subset B_{r_i}(o)$, where $r_i$ is the sequence in \cref{thm:RitoreRosales},  and defining $\widetilde\Omega_i^c\eqdef \Omega_i^c \cap B_S(o)$ we have
    \[
    \vol(\widetilde\Omega_i^c) \to \vol(\Omega),
    \qquad
    P(\widetilde\Omega_i^c) \to P(\Omega).
    \]
    
    Therefore we can define a new sequence 
    \[
    H_i \eqdef \widetilde\Omega_i^c \cup E_{i,j} \cup \bigcup_{\stackrel{\ell=1}{\ell\neq j}}^K \Omega_i^d \cap B_{T_{i,\ell}}(p_{i,\ell}),
    \]
    where $K> j$ is such that, by taking into account \eqref{eq:Step3b} and the fact that $E_{i,j}$ converge in $L^1$-strong to $W_j$ that satisfies $\meas_j(W_j)=\meas_j(Z_j)=\lim_i\vol(\Omega_i^d\cap B_{T_{i,j}}(p_{i,j}))$, we have that
    \begin{equation*}
        \lim_i \left(\vol( \widetilde\Omega_i^c) + \vol( E_{i,j}) + \sum_{\stackrel{\ell=1}{\ell\neq j}}^K \vol(\Omega_i^d \cap B_{T_{i,\ell}}(p_{i,\ell}))\right) = V-\eps,
    \end{equation*}
    for some $\eps\in[0,\eps_0)$. Now since $K$ is finite, the sets whose union defines $H_i$ have diverging mutual distance, and thus $\lim_i \vol(H_i) = V-\eps$ and
    \[
\begin{split}
    \lim_i P(H_i) &= P(\Omega) + P_{X_j}(W_j) + \lim_i \sum_{\stackrel{\ell=1}{\ell\neq j}}^K P(\Omega_i^d \cap B_{T_{i,\ell}}(p_{i,\ell})) \\
    &\le P(\Omega) + P_{X_j}(W_j) + 
    %
     \lim_i \sum_{\stackrel{\ell=1}{\ell\neq j}}^{N_i} P(\Omega_i^d \cap B_{T_{i,\ell}}(p_{i,\ell})) 
   \\
    & =  
    P(\Omega) + P_{X_j}(W_j) + \lim_i \left(P(\Omega_i^d) - P(\Omega_i^d \cap B_{T_{i,j}}(p_{i,j}))\right) \\
    &\le I(V) -\eta,
\end{split}
    \]
    where in the last two lines we used \eqref{eq:Step4b}, \cref{thm:RitoreRosales}, and \eqref{eq:Step4e}. On the other hand $\lim_i P(H_i) \ge \liminf_i I(V-\eps_i)$ for some sequence $\eps_i\to \eps \in[0,\eps_0)$. Hence
    \[
    I(V) -\eta \ge \liminf_i I(V-\eps_i) = I(V-\eps) \ge I(V) -\frac\eta2,
    \]
    by continuity of $I$ and the choice of $\eps_0$ in \eqref{eq:ContinuityProfile}, that yields a contradiction. Hence if $\overline{N}=+\infty$, we completed the proof of item (iii).
    
    Finally in case $\overline{N}<+\infty$, and for indices $j<\overline N$ the proof of \eqref{eq:Itemiii1} can be performed in the very analogous way, exploiting the continuity of $I$. More precisely, also in this case the absurd hypothesis consists in \eqref{eq:Step4e} and we can define $W_j$, $E_{i,j}$, and $\widetilde\Omega_i^c$ as before. Moreover, as the $\overline{N}$-th generation of points $p_{i,\overline{N}}$ are determined in the Step 1 by the application of \cref{it:COCOCompactness} in \cref{lem:CoCo}, for any $\bar\eps>0$ we find $L'>0$ such that the newly defined sequence
    \[
    \widehat H_i\eqdef \widetilde\Omega_i^c \cup E_{i,j} \cup \, [ \Omega_i^d \cap B_{L'}(p_{i,\overline{N}}) ] \,\cup \bigcup_{\stackrel{\ell=1}{\ell\neq j}}^{\overline{N}-1} \Omega_i^d \cap B_{T_{i,\ell}}(p_{i,\ell}) 
    \]
    satisfies $\vol(\widehat H_i) \to V- \bar\eps$. Up to choosing a larger finite $L'$, the previous calculations can be still carried out, leading to the desired contradiction.
    
    It only remains to prove \eqref{eq:Itemiii2} and that the resulting $Z_{\overline{N}}$ is an isoperimetric region. 
    Similarly as above, since for example from \eqref{eq:Step4b} we know that $P(G_i)$ is uniformly bounded, by item (b) of \cref{prop:SemicontinuitaAmbrosioBrueSemola}, we have that, up to subsequence, $G_i$ converges to a finite perimeter set $Z_{\overline{N}}\subset X_{\overline{N}}$ in $L^1_{\rm loc}$, that means that for every $r>0$ it occurs that $G_i \cap B_r(p_{i,\overline{N}}) \to Z_{\overline{N}} \cap B_r(p_{\overline{N}})$ in $L^1$-strong as $i\to+\infty$. Now since $T_{i,\overline{N}}\ge R_{i,\overline{N}}$ and $p_{i,\overline{N}}, R_{i,\overline{N}}$ are produced by \cref{it:COCOCompactness} in \cref{lem:CoCo}, then for any $\delta>0$ there is $r>0$ such that $\vol(G_i \setminus B_r(p_{i,\overline{N}}) ) < \delta$ for any $i$. Hence it is immediate to deduce that $\vol(G_i) \to \mathfrak{m}_{\overline{N}}(Z_{\overline{N}})$, and thus $G_i \to Z_{\overline{N}}$ in $L^1$-strong. So now one can argue exactly as we did above for $j<\overline{N}$, and one shows that $Z_{\overline{N}}$ is an isoperimetric region in $X_{\overline{N}}$ and $\lim_i P(G_i) = P_{X_{\overline{N}}}(Z_{\overline{N}})$.

    \item[Step 5.]
    We claim that item (iv) holds. 
    
    Indeed we already know from \eqref{eq:Step3b} (and the last condition in \eqref{eq:Step1} when $\overline N<+\infty$), and \cref{thm:RitoreRosales} that
    \[
    W= \sum_{j=1}^{\overline{N}} \meas_j (Z_j),
    \qquad
    V= \vol(\Omega) + W.
    \]
    Moreover from \cref{thm:RitoreRosales}, \eqref{eq:Step4b}, and item (iii) we also deduce
    \[
    I(V) = \lim_i \left(P(\Omega_i^c) + P(\Omega_i^d)\right) \ge P(\Omega) + \sum_{j=1}^{\overline{N}} P_{X_j}(Z_j) = I(\vol(\Omega)) + \sum_{j=1}^{\overline{N}} I_{X_j}(\meas(Z_j)).
    \]
    On the other hand, we are exactly in the hypotheses for applying \eqref{eq:InfinityInequalityIsopProfile}, that yields
    \[
    I(V) \le I(\vol(\Omega)) + \sum_{j=1}^{\overline{N}} I_{X_j}(\meas(Z_j)).
    \]
    Hence equality holds, and this completes the proof of \eqref{eq:Itemiv2}.
    
\end{itemize}
\end{proof}

\subsection{Counterexamples and optimality of the assumptions}\label{Sec:ExampleVanishing}

In this part we construct examples of a submanifolds $(M^2,g)$ of $\R^3$ that are either collapsed with sectional curvature bounded below, or noncollapsed with Ricci unbounded below. Moreover, in such manifolds for any $v>0$ there exist perimeter-minimizing sequences $E_i$ of volume $\vol(E_i)=v$ for any $i$ such that
\begin{equation}\label{eq:ExampleVanishing}
    \lim_i \sup_{p \in M} \vol( E_i \cap B_R(p) ) = 0
\end{equation}
for any $R>0$. The occurrence of \eqref{eq:ExampleVanishing} exactly means that \cref{it:COCOvanishing} in \cref{lem:CoCo} happens. It follows that the strategy of the proof of \cref{thm:MassDecomposition}, which is based on the iteration of \cref{it:COCODicotomia} or \cref{it:COCOCompactness} in \cref{lem:CoCo}, together with the explicit estimate in \cref{lem:MasLowBound} and \cref{cor:IsopBounded}, is no longer applicable as long as one of the two hypotheses of \cref{thm:MassDecomposition} does not hold, namely noncollapsedness or Ricci bounded below. The occurrence of \eqref{eq:ExampleVanishing}, in fact, implies that no subset of $E_i$ can converge in $L^1_{\rm loc}$ to a nonempty limit set contained in some asymptotic GH-limit of the manifold.

First, we construct a submanifold $(M^2,g)$ of $\R^3$ that is collapsed, has sectional curvature bounded below, and such that \eqref{eq:ExampleVanishing} occurs for a perimeter-minimizing sequence of volume $1$. The remaining desired examples are then constructed following the same lines and are described at the end of the section.

Denoting by $(x,y,z)$ the standard coordinates in $\R^3$, we consider the plane $\Pi\eqdef\{x=0\}$. Define by induction numbers $z^i_j$ for any $i \ge 1$ and $j=1,\ldots, i$ by setting
\begin{equation}\label{eq:Znj}
    \begin{split}
    z^1_1 = 0,&\\
    z^i_1 = z^{i-1}_{i-1} + i& \qquad \forall\, i \ge 2, \\
    z^i_{j+1} = z^i_j + i & \qquad \forall\, i \ge 2, \,  j = 1, \ldots, i-1.
    \end{split}
\end{equation}
Let $h:[2,+\infty)\to \R$ be the function $h(x) := 1/x^2$, and let $\Sigma$ be the surface of revolution defined by $h$ by the rotation about the $x$-axis. For any $i \ge 1$ and $j=1,\ldots, i$, it is possible to glue to $\Pi$ a translated copy of $\Sigma$ so that the rotation axis of the surface coincide with $\{y=0, z= z^i_j\}$ and the resulting surface has sectional curvature $\ge k$ for some $k \in (-\infty, 0]$. For any such $i,j$, we denote by $\Sigma^i_j$ the translated copy of $\Sigma$. We are going to modify the profile function of each $\Sigma^i_j$ on some set $\{x \ge x^i_j\}$, yielding new surfaces of revolution denoted by $\mathcal{E}^i_j$, without lowering too much the sectional curvature. This will complete the construction of the desired surface $M$. In the end, we want the surfaces $\mathcal{E}^i_j$ to satisfy
\begin{enumerate}
    \item[i)]
    $\vol(\mathcal{E}^i_j \cap \{x \ge x^i_j\} ) = 1/i$ for any $i\ge 1$ and any $j=1,\ldots, i$;
    
    \item[ii)]
    $P(\mathcal{E}^i_j  \cap \{x \ge x^i_j\} ) \le 1/2^i$ for any $i\ge 1$ and any $j=1,\ldots, i$. 
\end{enumerate}

So let $i,j\le i$ be fixed. Take $x^i_j$ very large so that $P(\Sigma^i_j \cap \{x \ge x^i_j\} ) \le 1/2^i$ and $\vol(\Sigma^i_j \cap \{x \ge x^i_j\} ) < 1 /4i$. First modify the profile function $h$ of $\Sigma^i_j$ into $h^i_j$ as depicted in \cref{Fig}. More precisely, the new function $h^i_j$ is smooth and such that: $(h^i_j)^{(\ell)} (x^i_j+1) = 0 $ for any $\ell$, $|(h^i_j)''| \le 4 |h'(x^i_j)| = 8/(x^i_j)^3$ on $[x^i_j,x^i_j+1]$, $|(h^i_j)''| \le 4 |h'(x^i_j+2)|$ on $[x^i_j+1,x^i_j+2]$, and $h^i_j = C^i_j/x^2$ for $x \ge x^i_j+2$. Since the sectional curvature of a revolution surface with profile function $H(x)$ is given by $- H''/[H(1+(H')^2)^2]$, using that $h^i_j \ge h (x^i_j + 1)$ on $[x^i_j,x^i_j+1]$ and that $h^i_j \ge C^i_j/(x^i_j+2)^2$ on $[x^i_j+1,x^i_j+2]$, it is immediate to conclude that the sectional curvature of the revolution surface given by $h^i_j$ is bounded below by some $k \in (-\infty,0]$ independent of $i,j$. Moreover, up to choosing a bigger $x^i_j$, we can further ensure that the volume of the revolution surface defined by $h^i_j$ is $\le 1/2i$. At this point it suffices to introduce a piece of round cylinder to recover the missing volume, that is, we define the final profile function $f^i_j$ for $\mathcal{E}^i_j$ by
\[
f^i_j(x) = \begin{cases}
h^i_j(x) & x \in [x^i_j,x^i_j+1], \\
h^i_j(x^i_j+1) & x \in (x^i_j + 1, x^i_j+1+L], \\
h^i_j(x-L) & x \in (x^i_j+1+L,+\infty),
\end{cases}
\]
taking $L>0$ so that $\vol(\mathcal{E}^i_j \cap \{x \ge x^i_j\} ) = 1/i$ (see \cref{Fig}).

\begin{figure}[H]
	\begin{center}
		\begin{tikzpicture}[scale=1]
		\path[font=\normalsize]
		(1.3,2.3)node[]{$h^i_j$};
		\draw
		(-0.5,0)--(15.5,0);
		\draw
		(0,-0.5)--(0,3);
		\path[font=\Huge]
		(1,-0.01)node[]{$\cdot$};
		\path[font=\normalsize]
		(1,-0.4)node[]{$x^i_j$};
		\path[font=\Huge]
		(3,-0.01)node[]{$\cdot$};
		\path[font=\normalsize]
		(2.9,-0.4)node[]{$x^i_j+1$};
		\path[font=\Huge]
		(5,-0.01)node[]{$\cdot$};
		\path[font=\normalsize]
		(4.9,-0.4)node[]{$x^i_j+2$};
		\draw
		(1,2)to[out=345, in=180](3,1.5)to[out=0, in=160](5,1)to[out=340, in= 178](10,0.3);
		\draw[dashed]
		(3,1.5)--(8,1.5);
		\draw[dashed]
		(8,1.5)to[out=0, in=160](10,1)to[out=340, in= 178](15,0.3);
		\path[font=\normalsize]
		(10,1.6)node[]{$f^i_j$};
		\path[font=\Huge]
		(8,-0.01)node[]{$\cdot$};
		\path[font=\normalsize]
		(7.95,-0.4)node[]{$x^i_j+1+L$};
		\end{tikzpicture}
	\end{center}
	\caption{Qualitative picture of the functions $h^i_j$ and $f^i_j$.}\label{Fig}
\end{figure}
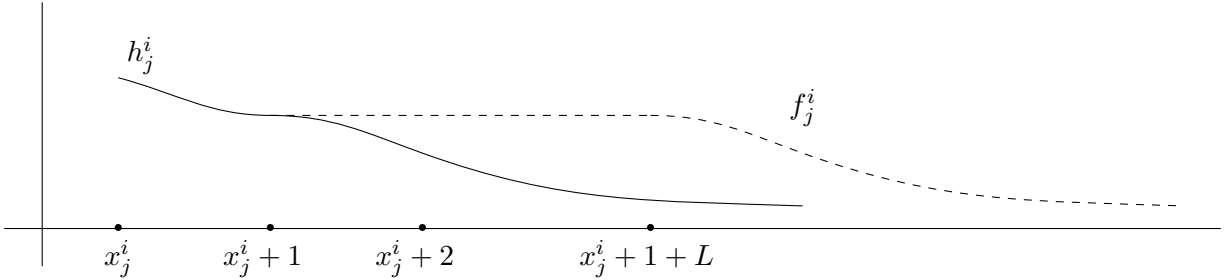

We can now show the existence of the claimed minimizing sequence. Let $E_i\eqdef \cup_{j=1}^i \mathcal{E}^i_j \cap \{x> x^i_j\}$. By i) we have that $\vol(E_i) =1 $ for any $i$, while ii) implies
\[
P(E_i) \le \frac{i}{2^i} \xrightarrow[i\to+\infty]{} 0 .
\]
Hence $E_i$ is perimeter-minimizing for the volume $1$, and $I_{(M,g)} (1)= 0$. Finally the choices of $z^i_j$ imply that, for any fixed $R>0$, any ball $B_R(p)$ intersects at most one connected component of $E_i$ for $i$ large. Hence we estimate
\[
\limsup_i \sup_{p \in M} \vol(E_i \cap B_R(p)) \le \limsup_i \max_{j} \vol (\mathcal{E}^i_j \cap \{x> x^i_j\} ) = \lim_i \frac1i = 0,
\]
and \eqref{eq:ExampleVanishing} follows.

In order to generalize the example to a manifold such that \eqref{eq:ExampleVanishing} occurs for some minimizing sequence of any assigned volume one can perform the following additional construction. Denote by $\{v_{\ell}\}_{\ell\in\N}$ an enumeration of $\Q \cap(0,+\infty)$. For any $\ell$ we can glue to the above constructed manifold a new sequence of cuspidal ends adapted to the volume $v_\ell$, just like done for the case of volume $1$, but along some lines $\{ y = y_\ell , x=0 \} \subset \Pi$ with $y_{\ell+1}>y_\ell$. This yields a final surface with sectional curvature bounded below such that for any $\ell$ there is a perimeter-minimizing sequence $\{E_i^\ell\}_{i\in \N}$ of volume $v_\ell$ such that \eqref{eq:ExampleVanishing} occurs. Then for a given $v>0$, a perimeter-minimizing sequence of volume $v$ satisfying \eqref{eq:ExampleVanishing} is given by $E_i^{\ell_i} \cup B_{r_i}(o)$ for some $v_{\ell_i}\to v^-$ and $r_i\to 0$ such that $\vol(E_i^{\ell_i} \cup B_{r_i}(o)) = v $. Observe that on such a manifold the isoperimetric profile identically vanishes (compare with \cref{rem:ProfiloPositivo}).

In order to get a noncollapsed surface such that for any $v>0$ there exist perimeter-minimizing sequences $E_i$ of $v$ such that \eqref{eq:ExampleVanishing} occurs, one can just replace the countably many sequences of cuspidal ends of the previous example with bubbles connected to $\Pi$ by means of shrinking catenoidal necks (clearly, the sectional curvature tends to $-\infty$ on such necks). In this way there occurs the same splitting of the volume of a minimizing sequence in a sequence of bubbles instead of a sequence of cuspidal ends. Once again, on such a manifold the isoperimetric profile identically vanishes (compare with \cref{rem:ProfiloPositivo}).



\subsection{The existence theorem}\label{sec:ExistenceRigidity}
The aim of this section is to exploit the previous result about the asymptotic mass decomposition to obtain existence of isoperimetric regions for Riemannian manifolds with some GH-prescriptions at infinity, completing the proof of \cref{thm:MainRicNcGHIntro}. 

When combined with a suitable asymptotic mass decomposition, the following result, due to Morgan--Johnson \cite[Theorem 3.5]{MorganJohnson00}, constitutes the key for the existence, since it asserts that a geodesic ball lying on a manifold with $\ric\ge (n-1)k$ is isoperimetrically more convenient then the ball in the model of curvature $k$ (having the same volume). The centrality of this comparison in such context was already pointed out in \cite{MondinoNardulli16}.

\begin{thm}\label{lem:ComparisonMJ}
Let $(M^n,g)$ be a complete Riemannian manifold such that $\ric\geq (n-1)k$ on some open set $\Omega\subset M^n$, for $k\in \R$. Then
$$
P(B)\leq P_{k}(\mathbb B_k(\vol(B))), \quad \text{for every geodesic ball $B\subset \Omega$},
$$
where
$\mathbb B_k(\vol(B))$ is a geodesic ball on the simply connected model of sectional curvature $k$ and dimension $n$ having volume equal to $\vol(B)$.

Moreover, equality holds if and only if $(B,g)$ is isometric to $(\mathbb B_k(\vol(B)),g_{k})$, where $g_k$ is the metric on the simply connected model of sectional curvature $k$ and dimension $n$.
\end{thm}

We can now state and prove our main existence result.

\begin{thm}\label{thm:MainRicNcGH}
Let $k\in(-\infty,0]$ and let $(M^n,g)$ be a complete noncompact Riemannian manifold such that $\ric\geq (n-1)k$ on $M\setminus\mathcal{C}$, where $\mathcal{C}$ is compact.

Suppose that $(M^n,g)$ is GH-asymptotic to the simply connected model of constant sectional curvature $k$ and dimension $n$.

Then for any $V>0$ there exists an isoperimetric region of volume $V$ on $(M^n,g)$.
\end{thm}

\begin{proof}
Since $(M^n,g)$ is GH-asymptotic to the simply connected model of constant sectional curvature $k$ and dimension $n$, then $(M^n,g)$ is noncollapsed. Indeed, if there is a sequence of balls $B_1(y_i)$ with $\lim_i \vol(B_1(y_i))=0$, then $y_i$ must diverge to infinity, hence by assumption $(M^n,\dist,y_i) $ converges in the pGH-sense to the simply connected model of constant sectional curvature $k$ and dimension $n$, which we denote here by $\mathbb M^n_k$, with its own geodesic distance and volume measure, and pointed at some fixed $o \in \mathbb M^n_k$. Hence item (b) in \cref{thm:volumeconvergence} occurs, and thus $n={\rm dim}_H \, \mathbb M^n_k  \le n-1$, which is impossible.

By \cref{rem:Approximation}, let $\Omega_i \subset M^n$ be a minimizing sequence (for the perimeter) of volume $V>0$ such that $\Omega_i$ is bounded and smooth for any $i$. Let $\Omega_i^c, \Omega_i^d$ be as in \cref{thm:RitoreRosales}. If $\vol(\Omega_i^d)\to 0$, then the set $\Omega$ given by \cref{thm:RitoreRosales} is an isoperimetric region of the volume $V$ and the proof ends. So suppose instead that $\lim_i \vol(\Omega_i^d) = W >0$. Then we can apply \cref{thm:MassDecomposition}. We employ the notation of \cref{thm:MassDecomposition}. By assumption and from \cref{thm:volumeconvergence}, for any $j<\overline{N}+1$ the pmGH limit space $(X_j,\dist_j,\mathfrak{m}_j,p_j)$ is $ \mathbb M^n_k$ with its own geodesic distance and volume measure, and pointed at some fixed $o \in \mathbb M^n_k$. Moreover, since in $\mathbb M_k^n$ balls are isoperimetric regions for their own volume, we have that, for any $j<\overline N+1$,
\begin{equation}\label{eq:ApplicoMJ}
P_k(Z_j) \ge P_k(\mathbb B_k(\vol_k(Z_j))),
\end{equation}
where $\mathbb B_k(\vol_k(Z_j))$ is a geodesic ball in $\mathbb M_k^n$ having volume equal to $\vol_k(Z_j)$, while $P_k$ is the perimeter functional on $\mathbb M^n_k$.

Now observe that for any compact set $\mathcal{K}\subset M^n$, we have that
\begin{equation}\label{eq:PalleGrandiEnd}
    \sup \left\{ \vol(B_r(p)) \st r>0\,\text{ and }\,B_r(p) \Subset M\setminus\mathcal{K} \right\} = +\infty.
\end{equation}
Indeed, suppose by contradiction the above supremum is bounded by a constant $S<+\infty$. Take $R>0$ such that $\vol_k( B^{\mathbb M_k^n}_R(o) ) > 10 S$, where $B^{\mathbb M_k^n}_R(o)$ is a ball of radius $R$ and center $o$ in $\mathbb M_k^n$. Consider a sequence of balls $B_R(x_i) \subset M^n$ of radius $R$ with $\dist(x_i,\mathcal{K})\to+\infty$. Then, up to passing to a subsequence, \cref{thm:volumeconvergence} and the absurd hypothesis imply that
\[
S \ge \lim_i \vol( B_R(x_i) ) > 10 S,
\]
that is impossible.

Hence \eqref{eq:PalleGrandiEnd}, together with the continuity of the volume with respect to the radius of balls, imply that, for any compact set $\mathcal{K}\subset M^n$ and any assigned finite volume $v$, we can find a ball of volume $v$ compactly contained in the end $M\setminus \mathcal{K}$. So let $v_j \eqdef \vol_k(Z_j)$ for $j<\overline{N}+1$. Since $\Omega = \lim_i \Omega_i^c$ is bounded by \cref{cor:IsopBounded}, there is a first compact set $\mathcal{K}_1$ such that $\Omega \cup \mathcal{C} \subset \mathcal{K}_1$. Then by \eqref{eq:PalleGrandiEnd} there is a ball $B_{r_1}(q_1) \Subset M^n \setminus \mathcal{K}_1$ such that $\vol(B_{r_1}(q_1)) = v_1$. Inductively, for any $2\leq j<\overline{N}+1$ we find a compact set $\mathcal{K}_j$ such that
\[
\mathcal{K}_j \Supset B_{r_{j-1}}(q_{j-1}) \cup \mathcal{K}_{j-1},
\]
and balls $B_{r_j}(q_j) \Subset M^n \setminus \mathcal{K}_j$ having volume $\vol(B_{r_j}(q_j)) = v_j$. Hence the balls $\{B_{r_j}(q_j)\st j<\overline{N}+1\}$ are pairwise located at positive distance and we can define the set
\[
\widetilde\Omega \eqdef \Omega \cup \bigcup_{j=1}^{\overline{N}} B_{r_j}(q_j).
\]
By \eqref{eq:Itemiv2} we have that
\[
\vol(\widetilde\Omega) = \vol(\Omega) + \sum_{j=1}^{\overline{N}} \vol( B_{r_j}(q_j) ) = \vol(\Omega) + \sum_{j=1}^{\overline{N}} v_j = \vol(\Omega) + W = V.
\]
Moreover, combining \eqref{eq:Itemiv2}, \eqref{eq:ApplicoMJ}, and \cref{lem:ComparisonMJ}, since all the constructed balls $B_{r_j}(q_j)$ are contained in an open set of $M^n$ on which $\ric \ge (n-1)k$, we obtain
\begin{equation}\label{eq:StimeMain1}
\begin{split}
   I(V) &= P(\Omega) + \sum_{j=1}^{\overline{N}} P_k(Z_j) \ge 
P(\Omega) + \sum_{j=1}^{\overline{N}} P_k(\mathbb B_k(\vol_k(Z_j)) )\\
&\geq P(\Omega) + \sum_{j=1}^{\overline{N}} P(B_{r_j}(q_j)) = P(\widetilde\Omega). 
\end{split}
\end{equation}
Therefore $\widetilde\Omega$ is an isoperimetric region for the volume $V$.
\end{proof}

\begin{remark}\label{rem:Nfinite}
We observe that a posteriori $\overline{N}$ is a finite natural number in $\N$ in the proof of \cref{thm:MainRicNcGH}. Indeed, if $\overline{N}=+\infty$, then the countably many constructed balls $B_{r_j}(q_j)$ can be easily taken so that the resulting $\widetilde\Omega$ is unbounded. But as $\widetilde\Omega$ turns out to be an isoperimetric region, it must be bounded by \cref{cor:IsopBounded}.
\end{remark}

\section{Applications and examples}\label{sec:ApplicationsExamples}

In this section we give effective conditions that imply the hypotheses of \cref{thm:MainRicNcGH}. We start by recalling some definitions about convergence of manifolds. The following definition is taken from \cite[Section 11.3.2]{Petersen2016}.

\begin{defn}[$C^0$-convergence of manifolds]\label{defn:C0ConvergenceManifolds}
Given $(M^n,g,p)$ a pointed Riemannian manifold, and $\{(M_i^n,g_i,p_i)\}_{i\in\mathbb N}$ a sequence of pointed Riemannian manifolds, we say that $(M_i^n,g_i,p_i)$ converge to $(M^n,g,p)$ {\em in the $C^0$-sense} if for every $R>0$ there exists a domain $\Omega\subset M^n$ containing $B_R(p)$ and, for large $i$, embeddings $F_i:\Omega\to M_i^n$ such that $F_i(p)=p_i$, $F_i(\Omega)$ contains $B_R(p_i)$, and the pull-back metrics $F_i^*g_i$ converge to $F$ in the $C^0$-sense on $\Omega$, i.e., all the components of the metric tensors converge in the $C^0$ norm in a finite covering of coordinate patches on $\Omega$. 
\end{defn}

By using the previous notion of convergence, we can define what means for a Riemannian manifold to be $C^0$-asymptotic to the simply connected model $\mathbb M^n_k$ of dimension $n\in\mathbb N$ and constant sectional curvature $k\in\mathbb R$. The forthcoming notion has been investigated in \cite{MondinoNardulli16}, see in particular \cite[Theorem 1.2]{MondinoNardulli16}.
\begin{defn}[$C^0$-local asymptoticity]
We say that a Riemannian manifold $(M^n,g)$ is {\em $C^0$-locally asymptotic} to the simply connected model $\mathbb M^n_k$ of dimension $n\in\mathbb N$ and constant sectional curvature $k\in\mathbb R$ if for every diverging sequence of points $p_i$ in $M^n$ we have that $(M^n,g,p_i)$ converge to $(\mathbb M^n_k,g_k,o)$ in the $C^0$-sense, where $g_k$ is the Riemannian metric on $\mathbb M^n_k$ and $o$ is a fixed origin. 
\end{defn}

\begin{remark}[GH-asymptoticity and $C^0$-local asymptoticity]\label{rem:GHandC0}
We remark that our \cref{thm:MainRicNcGH} implies one of the main theorems in \cite{MondinoNardulli16}, namely \cite[Theorem 1.2]{MondinoNardulli16}. Indeed, the notion of being $C^0$-locally asymptotic to the simply connected model $\mathbb M^n_k$ of constant sectional curvature $k\in\mathbb R$ and dimension $n\in\mathbb N$, see \cite[Definition 2.2, Definition 2.4]{MondinoNardulli16}, is readily stronger than being GH-asymptotic to $\mathbb M^n_k$, cf. \cite[Section 11.3.2]{Petersen2016}.

As a consequence, all the examples in \cite[Remark 1.1]{MondinoNardulli16}, namely the ALE gravitational instantons, the asymptotically hyperbolic Einstein manifolds, and the Bryant type solitons satisfy the hypotheses of \cref{thm:MainRicNcGH}.

\end{remark}


As an easy consequence of \cref{lem:ComparisonMetrics} we get a criterion to check that a Riemannian manifold is $C^0$-locally asymptotic, and hence GH-asymptotic (see \cref{rem:GHandC0}), to the simply connected model of constant sectional curvature $k\in\mathbb R$ and dimension $n\in\mathbb N$. We introduce our notions of \emph{sectional curvature asymptotically equal to $k$} and of \emph{asymptotically diverging injectivity radius}.


\begin{defn}[Sectional curvature asymptotically equal to $k$]\label{def:AVSC}
   Given $k\in (- \infty, 0]$, we say that a noncompact Riemannian manifold $(M^n,g)$ has {\em sectional curvature asymptotically equal to $k$} if there exists $o\in M^n$ such that for every $0<\varepsilon<1$ there exists $R_\epsilon>0$ for which 
   \begin{equation}\label{eqn:VanishingCurv}
      \begin{split}
          \abs{\sect_x(\pi)-k}\leq \varepsilon& \qquad \text{for all}\, x\in M\setminus \overline{B}_{R_\epsilon}(o),\,\,\text{for all 2-planes $\pi$ in $T_x M^n$}.
      \end{split}
   \end{equation}
   If $k=0$ we say that $(M^n,g)$ has {\em asymptotically vanishing sectional curvature}.
\end{defn}

We stress that, from now on, when we write $\abs{\sect}\leq c$ everywhere on some set $\Omega$, we mean that $\abs{\sect_x(\pi)}\leq c$ for every $x\in\Omega$ and every 2-plane $\pi\in T_x M^n$.

\begin{defn}[Asymptotically diverging injectivity radius]\label{def:ADIR}
   We say that a noncompact Riemannian manifold $(M^n,g)$ has {\em asymptotically diverging injectivity radius} if there exists $o\in M^n$ such that for every $S>1$ there exists $R_S>0$ for which 
   \begin{equation}\label{eqn:DivergingRadius}
      \begin{split}
          \mathrm{inj}(x)\geq  S &\qquad \text{for all}\,\, x\in M\setminus \overline{B}_{R_S}(o).
      \end{split}
   \end{equation}
\end{defn}
In the following statement we record how the coupling of the two conditions above suffices to infer the GH-asymptoticity to space forms. The following statement is a consequence of \cref{lem:ComparisonMetrics}.
\begin{prop}\label{prop:C0Asymp}
Let $(M^n,g)$ be a complete noncompact Riemannian manifold with \emph{sectional curvature asymptotically equal to $k$}, for some $k \in (-\infty, 0]$, and with \emph{asymptotically diverging injectivity radius}. 
Then, $(M^n, g)$ is $C^0$-locally asymptotic, and hence GH-asymptotic, to the simply connected model $\mathbb M^n_k$ of dimension $n\in\mathbb N$ and constant sectional curvature $k\in(-\infty, 0]$.
\end{prop}
By means of a classical compactness theorem, we can actually prove more than the $C^0$-local convergence above, as explained in the following Remark. 


We are now committed to link the two above defined notions of asymptotically constant sectional curvature and asymptotically diverging injectivity radius, in order to  discuss some effective and basic conditions that on a complete noncompact Riemannian manifold ultimately imply  the existence of isoperimetric regions of any volume. The following is essentially a consequence of a fundamental injectivity radius estimate in \cite{CheegerGromovTaylor}.

\begin{lemma}\label{lem:CGT}
Let $(M^n, g)$ be a complete Riemannian manifold with asymptotically vanishing sectional curvature. Let us assume there exists a compact set $\mathcal{C} \subset M^n$, a real number $\alpha<1$, and a constant $C>0$ such that
\begin{equation}
\label{volume-cond}
\mathrm{vol}{(B_r(p))} \geq C r^{n-\alpha},
\end{equation}
for any ball $B_r (p) \subset M^n \setminus \mathcal{C}$ with $r>1$.
Then $(M^n, g)$ has asymptotically diverging injectivity radius.
\end{lemma}

\begin{proof}
Let $0<\epsilon<1/100$, and fix $o \in M^n$. By the asymptotic vanishing of the sectional curvature, there exists a radius $R_\epsilon$ such that $\abs{\sect} < \epsilon$ on $M^n \setminus \overline{B}_{R_\epsilon}(o)$ and $\mathcal{C}\subset B_{R_\eps}(o)$. Let then $p \in M^n \setminus \overline{B}_{R_\epsilon + \pi/\sqrt{\epsilon}}(o)$, and observe that $B_{\pi/\sqrt{\epsilon}}(p) \Subset (M^n\setminus \overline{B}_{R_\epsilon}(o))$.

Assume that
$\mathrm{inj}(p) < \pi /\sqrt{\epsilon}$. Then, there exists $q \in \mathrm{Cut}(p)$ such that $d(p, q) = \mathrm{inj}(p)$, and that in particular still belongs to $M^n \setminus \overline{B}_{R_\epsilon} (o)$. Then, by \cite[Proposition 2.12, Chapter 13]{docarmo}, either there exists a geodesic $\gamma$ joining $p$ and $q$ such that $q$ is conjugate to $p$ along $\gamma$, or there exists a geodesic loop $\sigma$ based at $p$ passing through $q$ with length equal to $2\inj(p)$, so that in particular $\sigma$ is still contained in $M^n \setminus \overline{B}_{R_\epsilon}(o)$. In the first case, a straightforward application of Rauch's Comparison Theorem \cite[Proposition 2.4, Chapter 10]{docarmo} implies that $d(p, q) \geq \pi/\sqrt{\epsilon}$, a contradiction with the assumption above.

In the second case, we have $2 \, \mathrm{inj} (p) = \ell := \mathrm{length} \, (\sigma)$, and we can thus estimate $\mathrm{inj} (p)$ in terms of volumes of balls by means of \cite[Theorem 4.3]{CheegerGromovTaylor}, that yields
\begin{equation}
\label{cgt-applied}
\mathrm{inj}(p) 
\ge \frac{\pi}{8\sqrt{\eps}} \left(1 + \frac{v(n, -\epsilon, \frac{7\pi}{16\sqrt{\eps}})}{\mathrm{vol}\left(B_{\frac{3\pi}{16\sqrt{\epsilon}}}\left(p\right)\right)}\right)^{-1}
\geq \frac{\pi}{8\sqrt{\eps}} \left(1 + \frac{v(n, -\epsilon, \pi/\sqrt{\epsilon})}{\mathrm{vol}\left(B_{\frac{3\pi}{16\sqrt{\epsilon}}}\left(p\right)\right)}\right)^{-1}, 
\end{equation}
where we applied \cite[Theorem 4.3]{CheegerGromovTaylor} with $r=\pi/\sqrt{\eps}$, $r_0=r/4$, and $s=\tfrac{3}{16}r$ in the notation therein.

We have, for a dimensional constant $C(n)$, the following equality
\begin{equation}
\label{volume-incgt}
v(n, -\epsilon, \pi/\sqrt{\epsilon}) = \int\limits_0^{\pi/\sqrt{\epsilon}}\left(\frac{\sinh(s \sqrt{\epsilon})^{n-1}}{\sqrt{\epsilon}}\right)^{n-1} \de s = \frac{1}{(\sqrt\epsilon)^n} \int\limits_0^\pi \sinh(t)^{n-1} \de t = C(n) \frac{1}{(\sqrt\epsilon)^n}.
\end{equation}
Plugging \eqref{volume-cond} and \eqref{volume-incgt} into \eqref{cgt-applied} then yields
\[
\mathrm{inj}(p) \geq \frac{\pi}{8 \sqrt{\epsilon}} \left(1 + \overline{C} \epsilon^{-\frac\alpha2}\right)^{-1},
\]
where $\overline{C}=\overline{C}(C(n),C)$. All in all, we proved
\begin{equation*}
\label{inj-final}    
\mathrm{inj} (p) \geq \mathrm{min}\left\{\frac{\pi}{\sqrt{\epsilon}},  \frac{\pi}{8 \sqrt{\epsilon}} \left(1 + \overline{C} \epsilon^{-\frac\alpha2}\right)^{-1}  \right\}.    
\end{equation*}
Since $\alpha < 1$, the right-hand-side of the previous inequality diverges at infinity when $\varepsilon\to 0$, implying that $(M^n,g)$ has asymptotically diverging injectivity radius.
\end{proof}

As a consequence of \cref{prop:C0Asymp}, \cref{lem:CGT}, and \cref{thm:MainRicNcGH} we get the following isoperimetric existence result under curvature and volume conditions.

\begin{cor}\label{cor:Main1}
Let $(M^n ,g)$ be a complete Riemannian manifold with asymptotically vanishing sectional curvature. Moreover, assume that there exists a compact set $\mathcal{C}$ such that that $\ric\geq 0$ on $M\setminus\mathcal{C}$, and moreover there exist $\alpha<1$ and $C>0$ such that $\vol(B_r(p))\geq Cr^{n-\alpha}$ for any ball $B_r(p)\Subset M\setminus\mathcal{C}$ with $r>1$. Then, for every $V>0$ there exists an isoperimetric region of volume $V$. 
\end{cor}


The volume condition in the statement of \cref{cor:Main1} is automatically satisfied on manifolds $(M^n, g)$ with nonnegative Ricci curvature, asymptotically vanishing sectional curvature, and Euclidean volume growth, that is, $\mathrm{AVR}(M^n, g)>0$.

Indeed, in such a case, by Bishop-Gromov we have $\mathrm{AVR}(M^n, g) \omega_n r^n \leq \mathrm{vol} (B_r(p)) \leq \omega_n r^n$ for any $p \in M^n$ and any $r>0$. We observe also that, since such a condition on the volume of balls is needed to hold just outside some compact set, we actually get the existence of isoperimetric regions of any volume on any compact perturbation of a complete Riemannian manifold with asymptotically vanishing sectional curvature, $\ric \geq 0$, and Euclidean volume growth.
 
\begin{cor}\label{cor:AVR}
Let $(M^n, \tilde{g})$ be a complete Riemannian manifold with nonnegative Ricci curvature, asymptotically vanishing sectional curvature, and Euclidean volume growth. Let $(M^n, g)$ be a compact perturbation of $(M^n, \tilde{g})$, that is, there exists a compact set $\mathcal{C}$ such that $\tilde{g} = g$ on $M^n \setminus \mathcal{C}$. Then, for every $V>0$ there exists an isoperimetric region of volume $V$.
\end{cor}

It is interesting to observe that \cref{cor:AVR} applies to Perelman's example constructed in \cite{PerelmanExampleCones} (see also \cite[Section 8]{ChCo1}), that is a complete Riemannian manifold $(M^n,g)$ with nonnegative Ricci curvature, Euclidean volume growth, asymptotically vanishing sectional curvature (it satisfies a quadratic decay), and admitting non-isometric asymptotic cones. We recall that an asymptotic cone to a manifold $(M^n,g)$ at some $x\in M^n$ is the pGH limit of the sequence of metric spaces $(M^n,r_i^{-1}\dist, x)$ for some diverging sequence $r_i\to+\infty$. In particular, we deduce that the non-uniqueness of asymptotic cones is not an obstruction to the existence of isoperimetric regions, even in the case of $\ric\ge 0$ and Euclidean volume growth.




\begin{remark}[Asymptotically Euclidean and conical manifolds]\label{rem:AsymptoticallyEuclideanAndConical}
A direct application of \cref{cor:AVR} implies that every compact perturbation of an ALE manifold with $\ric\geq 0$, see e.g. \cite[Definition 4.13]{AgostinianiFogagnoloMazzieri}, has isoperimetric regions for any volume. Indeed, it is immediately checked that an ALE manifold with $\ric\geq 0$ has Euclidean volume growth and asymptotically vanishing sectional curvature.

An application of \cref{cor:AVR} implies also that every compact perturbation of a $C^2$-asymptotically conical manifold (in the sense of \cite{ChodoshEichmairVolkmann17}) with $\ric\geq 0$ has isoperimetric regions for every volume. Indeed, every $C^2$-asymptotically conical manifold has asymptotically vanishing sectional curvature and Euclidean volume growth. We remark that in \cite[Theorem 3]{ChodoshEichmairVolkmann17} the authors prove that a $C^{1,\alpha}$-asymptotically conical manifold (without further bounds on $\ric$) has isoperimetric regions for large volumes, and they describe the structure of isoperimetric regions with large volumes for $C^{2,\alpha}$-asymptotically conical manifolds. 
\end{remark}

Our results allow to generalize the applications on asymptotically conical manifolds considered in \cref{rem:AsymptoticallyEuclideanAndConical} to manifolds which are suitably asymptotic to warped products with $\ric\ge 0$ and asymptotically vanishing sectional curvature. We discuss this observation in the next remark.

\begin{remark}[Warped products with $\ric\ge0$ and asymptotically vanishing sectional curvature]\label{rem:Warped}
Let $(W,\widetilde g)$ be an arbitrary warped product defined by
\[
W:=(0,+\infty) \times L, \qquad \widetilde g:= \de r^2 + f(r)^2 g_L,
\]
where $(L,g_L)$ is a compact Riemannian manifold and $f:(0,+\infty)\to (0,+\infty)$ is a smooth function. Let $(M^n, g)$ be a complete Riemannian manifold such that there exists a compact set $\mathcal{K}\subset M^n$ such that $(M^n\setminus\mathcal K,  g)$ is isometric to $((a,+\infty)\times L , \widetilde g)$ for some $a\ge 0$.

We want to show here that if
\begin{equation}\label{eqn:Condition}
\lim_{r\to+\infty} f(r) = +\infty,
\qquad
\text{and $W$ has asymptotically vanishing sectional curvature},
\end{equation}
then $(M^n, g)$ has asymptotically diverging injectivity radius.
Observe that, by a direct computation of the Riemann tensor, the asymptotic vanishing of the sectional curvature is ensured every time $f'=o(f)$ and $f''=o(f)$ as $r\to+\infty$.

We now prove the latter claim after \eqref{eqn:Condition}. Indeed, since the sectional curvature is asymptotically vanishing, arguing as in the proof of \cref{lem:CGT}, for a given $\eps \in (0,1)$ it suffices to estimate from below the length $\ell$ of a geodesic loop based at $p\in M\setminus \mathcal{K}_\eps$  by $\ell \ge C /\eps$ for a constant $C$ independent of $\eps$, and for some compact set $\mathcal{K}_\eps \supset \mathcal{K}$.
We identify $M\setminus \mathcal K$ with $((a,+\infty)\times L , \widetilde g)$. Take $\widetilde{\mathcal{K}}_\eps $ such that $|\sect|<\eps$ on $M\setminus \widetilde{\mathcal{K}}_\eps$. As in \cref{lem:CGT}, we can consider $\mathcal{K}_\epsilon \supset \widetilde{\mathcal{K}}_\eps$ such that for every $p\in M\setminus \mathcal{K}_\eps$ we have $\dist(p,\widetilde{\mathcal{K}}_\eps) > \pi/\sqrt{\eps}$.
Also, without loss of generality, up to eventually enlarging $\mathcal{K}_\eps$, we can just estimate a geodesic loop $\gamma$ based at $p$ such that $\gamma:[0,1]\to ((a_\eps,+\infty)\times L, \widetilde g)$ and $a_\eps$ is such that $f(r)\ge 1/\eps$ for $r\ge a_\eps$. We have $\gamma=(\gamma_1,\gamma_2)$, and $\gamma_2'(0)\neq 0$, for otherwise $\gamma$ would be tangent to $(a_\eps,+\infty)$ and $\gamma$ would not be closed. Then $\gamma_2$ is a nonconstant continuous curve in $L$. For $S\subset L$, it can be shown by the direct computation of the second fundamental form of the isometric embedding $(a',+\infty)\times S \hookrightarrow ((a',+\infty)\times L, \widetilde g)$ that $S$ is a totally geodesic submanifold of $(L,g_L)$ if and only if so is $(a',+\infty)\times S$ in $((a',+\infty)\times L, \widetilde g)$, for any $a'$. This implies that $\gamma_2$ is a geodesic loop in $(L,g_L)$, up to reparametrization.
%
%
%
Hence the length of $\gamma$ is estimated from below by
\[
\begin{split}
\ell &= \int_0^1 \left(|\gamma_1'(t)|^2 + f^2(\gamma_1(t)) g_L(\gamma_2'(t),\gamma_2'(t)) \right)^{1/2} \de t \\
& \ge \frac{\sqrt 2}{2} \left( \int_0^1 |\gamma_1'(t)| \de t +  L(\gamma_2)\min_{[0,1]}f(\gamma_1(t)) \right) \\
&\ge \frac{\sqrt 2}{2} \left( \int_0^1 |\gamma_1'(t)| \de t +  {\rm syst}(L)\, \min_{[0,1]}f(\gamma_1(t)) \right),
\end{split}
\]
where ${\rm syst}(L)>0$ denotes the systole of $(L,g_L)$, that is the length of the shortest geodesic loop in $(L,g_L)$.
By construction, the above estimate implies that
\[
\ell \ge  \frac{\sqrt 2}{2}  {\rm syst}(L)\frac1\eps.
\]
Hence we conclude that the injectivity radius is asymptotically diverging.

\smallskip

Therefore, if in addition to \eqref{eqn:Condition}, we have that $\ric\geq 0$ outside a compact set of $M^n$, then \cref{prop:C0Asymp} applies, $(M^n ,g)$ is GH-asymptotic to the Euclidean space $\R^n$, and by \cref{thm:MainRicNcGH} there exist isoperimetric regions of any volume.

It is clear that the very same conclusion holds true if we assumed that $(M^n ,g)$ satisfies $\ric\geq 0$ outside a compact set and that $M^n$ is just $C^2$-asymptotic to $(W,\widetilde g)$.

\smallskip

We observe that, for examples, the Bryant type solitons mentioned in \cref{rem:GHandC0} fit in this setting of warped products.
\end{remark}



\appendix

\section{Comparison results in Riemannian Geometry}\label{sec:BishopGromov}

We write down a complete statement of a rather classical comparison result in Geometric Analysis, i.e., the Bishop--Gromov--G\"unther volume and area comparison under Ricci curvature lower bounds and sectional curvature upper bounds. The conclusions \eqref{1}, \eqref{2}, \eqref{4}, \eqref{5}, and the rigidity part of \cref{thm:BishopGromov} are consequences, e.g., of \cite[Theorem 3.101]{lafontaine}, \cite[Theorem 1.2 and Theorem 1.3]{SchoenYauLectures}, and the arguments within their proofs. We stress that in the case of a Ricci lower bound, the balls are not required to stay within the cut-locus as first realized by Gromov, see, e.g., \cite[Theorem 1.132]{chow-ricci}. Finally, the conclusion \eqref{3} follows from \cite[Corollary 2.22, item (i)]{PigolaRigoliSetti} and the coarea formula, while \eqref{6} follows verbatim from the proof of \cite[Corollary 2.22, item (i)]{PigolaRigoliSetti} by using \eqref{4} and concluding again with the coarea formula.

We also stress that in the forthcoming \cref{thm:BishopGromov} we do not actually assume the curvature bounds on all $M^n$, but just on an open subset $\Omega\subset M^n$. Consequently the conclusions hold for balls contained inside $\Omega$: indeed, the proofs of the classical geometric comparison theorems leading to \cref{thm:BishopGromov} can be localized, see, e.g., \cite[Remark 2.6]{PigolaRigoliSetti}. For a comparison result assuming more general Ricci lower bounds, we also refer the reader to \cite[Theorem 2.14]{PigolaRigoliSetti}.
\begin{thm}[Volume and perimeter comparison]
\label{thm:BishopGromov}
Let $(M^n,g)$ be a  complete Riemannian manifold, and let $\Omega \subset M^n$ be a on open subset such that $\ric\geq (n-1)k$ on $\Omega$ in the sense of quadratic forms for some $k \in \R$. Let us set $T_k:=+\infty$ if $k\leq 0$, and $T_k:=\pi/\sqrt{k}$ if $k>0$. Then, for every $p\in \Omega$ and for $r \leq T_k$ such that $B_r(p) \Subset \Omega$ the following hold
    \begin{align}
    \label{1}
    &\frac{\vol(B_r(p))}{v(n,k,r)}\to 1\,\text{as $r\to 0$ and it is nonincreasing}, \\
    \label{2}
    &\frac{P(B_r(p))}{s(n, k, r)}\to 1 \,\text{as $r\to 0$ and it is almost everywhere nonincreasing}, \\
    \label{3}
    &\frac{P(B_r(p))}{s(n, k, r)} \leq \frac{\vol(B_r(p))}{v(n,k,r)}\,\text{almost everywhere}.
    \end{align}
    Moreover, if one has $\vol(B_{\overline r}(p))=v(n,k,\overline r)$ for some $\overline r\leq T_k$ such that $B_{\overline{r}}(p)\Subset\Omega$, then $B_{\overline r}(p)$ is isometric to the ball of radius $\overline r$ in the simply connected model of constant sectional curvature $k$ and dimension $n$.

Conversely, let $(M^n,g)$ be a complete Riemannian manifold, and let $\Omega\subset M^n$ be an open subset such that $\sect\leq k$ on $\Omega$, for some $k\in\mathbb R$. Then, for every $p\in\Omega$ and for $r \leq \min\{T_k,\inj(p)\}$ such that $B_r(p) \Subset \Omega$ the following hold
\begin{align}
\label{4}
    &\frac{\vol(B_r(p))}{v(n,k,r)}\to 1\,\text{as $r\to 0$ and it is nondecreasing}, \\
    \label{5}
    &\frac{P(B_r(p))}{s(n, k, r)}\to 1 \,\text{as $r\to 0$ and it is nondecreasing}, \\
    \label{6}
    &\frac{P( B_r(p))}{s(n, k, r)} \geq \frac{\vol(B_r(p))}{v(n,k,r)}. 
    \end{align}
    Moreover, if one has $\vol(B_{\overline r}(p))=v(n,k,\overline r)$ for some $\overline r\leq \min\{T_k,\inj(p)\}$ such that $B_{\overline r}(p)\Subset\Omega$, then $B_{\overline r}(p)$ is isometric to the ball of radius $\overline r$ in the simply connected model of constant sectional curvature $k$ and dimension $n$. 
\end{thm}

In the case of sectional curvature bounds, it can be proved that the above result strengthens and yields a comparison between metric tensors. 

 \begin{lemma}[Comparison of metrics]\label{lem:ComparisonMetrics}
 Let $(M^n ,g)$ be a complete Riemannian manifold and fix $p\in M^n$. For every $k\in\mathbb R$, let $T_k$ be as in \cref{thm:BishopGromov}. Denote by $r$ the distance from $p$, let $R=\inj(p)$, and let $\{x^i\}_{i=1}^n$ be geodesic normal coordinates at the point $p$. Through the latter coordinates, let us identify $B_R(p)$ with the Euclidean ball $\mathbb B^n_R$, and let us denote by $g_1$ the canonical metric on $\mathbb S^{n-1}$. Then the following statements hold true
 \begin{itemize}
    \item[(i)] if $\sect(\nabla r \wedge X) \le k$ for any $X\perp \nabla r$ with $g(X,X)=1$, then the inequality $g \ge g_k := \d r^2 + \sn_k(r)^2g_1$ holds in the sense of quadratic forms on $\mathbb B^n_\rho$, where $\rho:=\min\{R,T_k\}$,
    
     \item[(ii)] if $\sect(\nabla r \wedge X) \ge k$ for any $X\perp \nabla r$ with $g(X,X)=1$, then the inequality $g \le g_k := \d r^2 + \sn_k(r)^2g_1$ holds in the sense of quadratic forms on $\mathbb B^n_\rho$, where $\rho:=\min\{R,T_k\}$.
 \end{itemize}
 \end{lemma}

\section{Boundedness of isoperimetric regions}\label{sec:BoundIsopRegions}

In this part we prove that having at disposal a Euclidean-like isoperimetric inequality for merely \emph{small} volumes suffices to imply that isoperimetric regions on a complete Riemannian manifold are bounded. This is a technical fact that we will employ several times. The proof is based on a rather classical argument already appearing in \cite[Proposition 3.7]{RitRosales04} and in \cite[Lemma 13.6]{MorganBook} in the Euclidean setting, and in \cite[Theorem 3]{Nar14} on Riemannian manifolds. However, we present here a rather self-contained proof for the convenience of the reader, pointing out that the weak assumption of a Euclidean-like isoperimetric inequality for small volumes is sufficient for the assertion.

\begin{thm}
\label{iso-bounded-general}
Let $(M^n,g)$ be a complete Riemannian manifold. Assume that there is $v_0>0$ such that the isoperimetric inequality
$$
c_0\vol(\Omega)^{(n-1)/n}\leq P(\Omega),
$$
holds true with some $c_0 > 0$ for any finite perimeter set $\Omega\subset M^n$ with $\vol(\Omega)<v_0$. Then the isoperimetric regions of $(M^n, g)$ are bounded.
\end{thm}

\begin{proof}
Let $E$ be an isoperimetric region and fix a point $p_0\in M^n$. Let, for every $r>0$,
\[
V(r)\eqdef \vol(E \setminus B_r(p_0)), 
\qquad
A(r)\eqdef P(E, M\setminus B_r(p_0)).
\]
By hypothesis there exists $r_0>0$ such that for any $r\ge r_0$ the volume $V(r)$ is sufficiently small to apply the isoperimetric inequality. In particular, for almost every $r\ge r_0$ we can write
\[
| V'(r) | + A(r) = \mathcal{H}^{n-1}(\partial B_r(p_0) \cap E) + P(E, M\setminus B_r(p_0)) = P(E \setminus B_r(p_0)) \ge c_0 V(r)^{\frac{n-1}{n}}.
\]
We want to prove that
\begin{equation}\label{eq:GoalBddIsopRegion}
    A(r) \le | V'(r) |  + C V(r) ,
\end{equation}
for some constant $C$, and for almost every $r$ sufficiently big. Combining  with the previous inequality, in this way we would get
\[
c_0 V(r)^{\frac{n-1}{n}} \le  C V(r)  + 2| V'(r) | \le \frac{c_0}{2}V(r)^{\frac{n-1}{n}} - 2 V'(r),
\]
because $| V'(r) | = - V'(r) $ and $ C V(r)\le \tfrac{c_0}{2}V(r)^{\frac{n-1}{n}} $ for almost every sufficiently big radius. Hence ODE comparison implies that $V(r)$ vanishes at some $r=\overline{r}<+\infty$, i.e., $E$ is bounded as a set of finite perimeter.

So we are left to prove \eqref{eq:GoalBddIsopRegion}. Let $R>0$ be fixed such that $P(E, B_R(p_0))>0$. There exists $\eps_0=\eps_0(R,E)>0$ and $C=C(R,E)>0$ such that for any $\eps\in (-\eps_0,\eps_0)$ there is a finite perimeter set $F$ with
\begin{equation}\label{eq:LocalModificationsSets}
    F\Delta E \Subset B_R(p_0),
    \qquad
    \vol(F)=\vol(E)+\eps,
    \qquad
    P(F,B_R(p_0)) \le P(E,B_R(p_0))+ C|\eps|.
\end{equation}
Indeed, this follows by the fact that, since the gradient of the characteristic function $\chi_E$ is represented by a measure $\nu|D\chi_E|$ where $\nu:M\to T M^n$ with $|\nu|=1$ at $|D\chi_E|$-a.e. point, and $P(E,\Omega) = \sup \left\{\int \scal{X,\nu} \de |D\chi_E| \st X \in \mathfrak{X}(\Omega), \supp X \Subset \Omega, |X|\le 1 \right\}$ for any open set $\Omega$, we can take a field $X$ with $|X|\le 1$ and compact support in $B_R(p_0)$ such that $ \int \scal{X,\nu} \de |D\chi_E| \ge \tfrac12P(E,B_R(p_0))>0$. Then, for small $t$, there is a smooth family of diffeomorphisms $\phi_t$ such that $\phi_0={\rm id}$ and $\partial_t \phi_t |_0 = X$. So the sets $F_t\eqdef \phi_t(E)$ verify the expansions
\[
\begin{split}
    \vol(F_t) &= \vol(E) + t \int \scal{X,\nu} \de |D\chi_E| + O(t^2),\\
    P(F_t,B_R(p_0)) &= P(E,B_R(p_0)) + t \int ( \div X - \scal{ \nabla_\nu X , \nu } )  \de |D\chi_E| + O(t^2).
\end{split}
\]
The above formulas for the variations of volume and perimeter are easily checked to hold on $M^n$ by the same computations carried out in the Euclidean space in \cite[Theorem 17.5 \& Proposition 17.8]{MaggiBook}. Since $\int \scal{X,\nu} \de |D\chi_E|>0$ and $\left|\int ( \div X - \scal{ \nabla_\nu X , \nu } )  \de |D\chi_E|\right| \le C(R,E)$, \eqref{eq:LocalModificationsSets} follows by taking $\eps_0(R,E)$ sufficiently small and then $F=F_{t_{\eps}}$ for the suitable $t_\eps$.

Now consider $r>R$ such that $V(r)<\eps_0$, and set $\eps=V(r)$. Then there is $F$ satisfying \eqref{eq:LocalModificationsSets}. Define also $\widetilde{F}= F \cap B_r(p_0)$, so that
\[
\vol(\widetilde{F}) = \vol(F) - \vol(F\setminus B_r(p_0)) = \vol(F) - \vol(E\setminus B_r(p_0)) = \vol(E) + \eps - \eps = \vol(E).
\]
Moreover, for almost every such $r$ we can additionally require that $P(F,\partial B_r(p_0))\eqdef \int_{\partial B_r(p_0)} \de |D\chi_F|=0$, as $|D\chi_F|$ is a finite Radon measure, see \cite[Proposition 2.16]{MaggiBook}. In this way (see \cite[Theorem 16.3]{MaggiBook}) we have
\[
\begin{split}
P(\widetilde{F}) &= P(F,B_r(p_0)) + \mathcal{H}^{n-1}(\partial B_r(p_0) \cap F) \\
&=P(F) - P(F,M\setminus B_r(p_0)) + \mathcal{H}^{n-1}(\partial B_r(p_0) \cap F).
\end{split}
\]
Since $E$ is an isoperimetric set we estimate
\[
\begin{split}
P(E)
& \le P(\widetilde{F})    = P(F) - P(E,M\setminus B_r(p_0)) + \mathcal{H}^{n-1}(\partial B_r(p_0) \cap E) \\
&\le P(E) + C\eps - A(r) + |V'(r)|,
\end{split}
\]
that is $A(r) \le |V'(r)| +  C V(r)$. Hence we see that \eqref{eq:GoalBddIsopRegion} holds for almost every $r>R$ such that $V(r)<\eps_0$, and the proof is completed.
\end{proof}

\printbibliography[heading=bibintoc,title={References}]

@article {AmborsioBrueSemola19,
    AUTHOR = {Ambrosio, L. and Bru\`{e}, E. and Semola, D.},
     TITLE = {Rigidity of the 1-{B}akry-\'{E}mery inequality and sets of finite
              perimeter in {RCD} spaces},
   JOURNAL = {Geom. Funct. Anal.},
  FJOURNAL = {Geometric and Functional Analysis},
    VOLUME = {29},
      YEAR = {2019},
    NUMBER = {4},
     PAGES = {949--1001},
       DOI = {10.1007/s00039-019-00504-5},
}

@article {AlmgrenBook,
    AUTHOR = {Almgren, Jr., F. J.},
     TITLE = {Existence and regularity almost everywhere of solutions to
              elliptic variational problems with constraints},
   JOURNAL = {Mem. Amer. Math. Soc.},
  FJOURNAL = {Memoirs of the American Mathematical Society},
    VOLUME = {4},
      YEAR = {1976},
    NUMBER = {165},
     PAGES = {viii+199},
   MRCLASS = {49F22 (58E15)},
  MRNUMBER = {420406},
MRREVIEWER = {Jean E. Taylor},
       DOI = {10.1090/memo/0165},
       URL = {https://doi-org.ezp.biblio.unitn.it/10.1090/memo/0165},
}

@Article{Carlotto2016,
author={Carlotto, A.
and Chodosh, O.
and Eichmair, M.},
title={Effective versions of the positive mass theorem},
journal={Inventiones mathematicae},
year={2016},
month={Dec},
day={01},
volume={206},
number={3},
pages={975-1016},
doi={10.1007/s00222-016-0667-3},
}

@book {VillaniBook,
    AUTHOR = {Villani, C.},
     TITLE = {Optimal transport},
    SERIES = {Grundlehren der Mathematischen Wissenschaften [Fundamental
              Principles of Mathematical Sciences]},
    VOLUME = {338},
      NOTE = {Old and new},
 PUBLISHER = {Springer-Verlag, Berlin},
      YEAR = {2009},
     PAGES = {xxii+973},
       DOI = {10.1007/978-3-540-71050-9},
}

@article {Sturm2,
    AUTHOR = {Sturm, K.-T.},
     TITLE = {On the geometry of metric measure spaces. {II}},
   JOURNAL = {Acta Math.},
  FJOURNAL = {Acta Mathematica},
    VOLUME = {196},
      YEAR = {2006},
    NUMBER = {1},
     PAGES = {133--177},
       DOI = {10.1007/s11511-006-0003-7},
}

@article {LottVillani,
    AUTHOR = {Lott, J. and Villani, C.},
     TITLE = {Ricci curvature for metric-measure spaces via optimal
              transport},
   JOURNAL = {Ann. of Math. (2)},
  FJOURNAL = {Annals of Mathematics. Second Series},
    VOLUME = {169},
      YEAR = {2009},
    NUMBER = {3},
     PAGES = {903--991},
       DOI = {10.4007/annals.2009.169.903},
}

@article {AmbrosioGigliSavare14,
    AUTHOR = {Ambrosio, L. and Gigli, N. and Savar\'{e}, G.},
     TITLE = {Metric measure spaces with {R}iemannian {R}icci curvature
              bounded from below},
   JOURNAL = {Duke Math. J.},
  FJOURNAL = {Duke Mathematical Journal},
    VOLUME = {163},
      YEAR = {2014},
    NUMBER = {7},
     PAGES = {1405--1490},
       DOI = {10.1215/00127094-2681605},
}

@article {ErbarKuwadaSturm15,
    AUTHOR = {Erbar, M. and Kuwada, K. and Sturm, K.-T.},
     TITLE = {On the equivalence of the entropic curvature-dimension
              condition and {B}ochner's inequality on metric measure spaces},
   JOURNAL = {Invent. Math.},
  FJOURNAL = {Inventiones Mathematicae},
    VOLUME = {201},
      YEAR = {2015},
    NUMBER = {3},
     PAGES = {993--1071},
       URL = {https://doi.org/10.1007/s00222-014-0563-7},
}

@article{EichmairMetzger,
author = "Eichmair, M. and Metzger, J.",
doi = "10.4310/jdg/1361889064",
fjournal = "Journal of Differential Geometry",
journal = "J. Differential Geom.",
month = "05",
number = "1",
pages = "159--186",
publisher = "Lehigh University",
title = "Large isoperimetric surfaces in initial data sets",
volume = "94",
year = "2013"
}

@article{EichmairMetzger2,
author={Eichmair, M.
and Metzger, J.},
title={Unique isoperimetric foliations of asymptotically flat manifolds in all dimensions},
journal={Inventiones mathematicae},
year={2013},
month={Dec},
day={01},
volume={194},
number={3},
pages={591-630},
doi={10.1007/s00222-013-0452-5},
url={https://doi.org/10.1007/s00222-013-0452-5}
}

@article{Shi,
    author = {Shi, Y.},
    title = "{The Isoperimetric Inequality on Asymptotically Flat Manifolds with Nonnegative Scalar Curvature}",
    journal = {International Mathematics Research Notices},
    volume = {2016},
    number = {22},
    pages = {7038-7050},
    year = {2016},
    month = {03},
    doi = {10.1093/imrn/rnv395},
    }

@article{HuiskenIlmanen,
author = "Huisken, G. and Ilmanen, T.",
fjournal = "Journal of Differential Geometry",
journal = "J. Differential Geom.",
number = "3",
pages = "353--437",
publisher = "Lehigh University",
title = "The {I}nverse {M}ean {C}urvature {F}low and the 
{R}iemannian {P}enrose {I}nequality",
volume = "59",
year = "2001"
}

@article {AmbrosioMondinoSavare15,
    AUTHOR = {Ambrosio, L. and Mondino, A. and Savar\'{e}, G.},
     TITLE = {Nonlinear diffusion equations and curvature conditions in
              metric measure spaces},
   JOURNAL = {Mem. Amer. Math. Soc.},
  FJOURNAL = {Memoirs of the American Mathematical Society},
    VOLUME = {262},
      YEAR = {2019},
       DOI = {10.1090/memo/1270},
}

@misc{Gigli13,
      title={The splitting theorem in non-smooth context}, 
      author={Gigli, N.},
      year={2013},
      eprint={1302.5555},
      archivePrefix={arXiv},
      }

@book {HeinonenKoskelaShanmugalingam,
    AUTHOR = {Heinonen, J. and Koskela, P. and Shanmugalingam,
              N. and Tyson, J. T.},
     TITLE = {Sobolev spaces on metric measure spaces},
    SERIES = {New Mathematical Monographs},
    VOLUME = {27},
      NOTE = {An approach based on upper gradients},
 PUBLISHER = {Cambridge University Press, Cambridge},
      YEAR = {2015},
     PAGES = {xii+434},
       DOI = {10.1017/CBO9781316135914},
}

@article{MondinoSpadaro,
author = "Mondino, A. and Spadaro, E.",
doi = "10.2140/apde.2017.10.95",
fjournal = "Analysis & PDE",
journal = "Anal. PDE",
number = "1",
pages = "95--126",
publisher = "MSP",
title = "On an isoperimetric-isodiametric inequality",
%url = "https://doi.org/10.2140/apde.2017.10.95",
volume = "10",
year = "2017"
}

@misc{LeonardiRitore,
%       title={Isoperimetric inequalities in unbounded convex bodies}, author={Leonardi, G. P. and Ritor\'e, M. and Vernadakis, E.},
%       year={2016},
%       eprint={1606.03906},
%       prefix = {arXiv},
%       note = {To appear in \emph{Memoirs of the AMS}},
% }

@article {Minerbe,
    AUTHOR = {Minerbe, V.},
     TITLE = {On the asymptotic geometry of gravitational instantons},
   JOURNAL = {Ann. Sci. \'Ec. Norm. Sup\'er. (4)},
  FJOURNAL = {Annales Scientifiques de l'\'Ecole Normale Sup\'erieure. Quatri\`eme
              S\'erie},
    VOLUME = {43},
      YEAR = {2010},
    NUMBER = {6},
     PAGES = {883--924},
     MRCLASS = {53C26 (53C80)},
  MRNUMBER = {2778451},
MRREVIEWER = {Derek G. Harland},
       DOI = {10.24033/asens.2135},
       URL = {https://doi.org/10.24033/asens.2135},
}

@Article{DePhilippis2017,
author={De Philippis, G.
and Franzina, G.
and Pratelli, A.},
title={Existence of Isoperimetric Sets with Densities ``Converging from Below'' on $\R^N$},
journal={The Journal of Geometric Analysis},
year={2017},
month={Apr},
day={01},
volume={27},
number={2},
pages={1086-1105},
doi={10.1007/s12220-016-9711-1},
url={https://doi.org/10.1007/s12220-016-9711-1}
}

@article{CatinoMazzieri,
title = "Gradient Einstein solitons",
journal = "Nonlinear Analysis",
volume = "132",
pages = "66 - 94",
year = "2016",
doi = "https://doi.org/10.1016/j.na.2015.10.021",
url = "http://www.sciencedirect.com/science/article/pii/S0362546X15003661",
author = "Catino, G. and Mazzieri, L."
}

@Article{Chodosh2016,
author={Chodosh, O.},
title={Large Isoperimetric Regions in Asymptotically Hyperbolic Manifolds},
journal={Communications in Mathematical Physics},
year={2016},
month={Apr},
day={01},
volume={343},
number={2},
pages={393-443},
doi={10.1007/s00220-015-2457-y},
}

@article {LeonardiRigot,
    AUTHOR = {Leonardi, G. P. and Rigot, S.},
     TITLE = {Isoperimetric sets on {C}arnot groups},
   JOURNAL = {Houston J. Math.},
  FJOURNAL = {Houston Journal of Mathematics},
    VOLUME = {29},
      YEAR = {2003},
    NUMBER = {3},
     PAGES = {609--637},
}

@Article{Pedrosa2004,
author={Pedrosa, R. H. L.},
title={The Isoperimetric Problem in Spherical Cylinders},
journal={Annals of Global Analysis and Geometry},
year={2004},
month={Nov},
day={01},
volume={26},
number={4},
pages={333-354},
doi={10.1023/B:AGAG.0000047528.20962.e2},
url={https://doi.org/10.1023/B:AGAG.0000047528.20962.e2}
}

@misc{Arezzo,
      title={Existence of CMC-foliations in asymptotically cuspidal manifolds}, 
      author={Arezzo, C. and Corrales, K.},
      year={2018},
      eprint={1811.12054},
      archivePrefix={arXiv},
}

@book{MorganBook,
   title =     {Geometric measure theory: a beginner's guide},
   author =    {Morgan, F.},
   publisher = {Academic Press},
   year =      {2000},
   series =    {},
   edition =   {3},
   volume =    {},
   url =       {http://gen.lib.rus.ec/book/index.php?md5=81ddadb6673c1d450b7af71d34cb0ac0}
}

@inproceedings {AmbrosioSurvey,
    AUTHOR = {Ambrosio, L.},
     TITLE = {Calculus, heat flow and curvature-dimension bounds in metric
              measure spaces},
 BOOKTITLE = {Proceedings of the {I}nternational {C}ongress of
              {M}athematicians---{R}io de {J}aneiro 2018. {V}ol. {I}.
              {P}lenary lectures},
     PAGES = {301--340},
}

@article {DePhilippisGigli18,
    AUTHOR = {De Philippis, G. and Gigli, N.},
     TITLE = {Non-collapsed spaces with {R}icci curvature bounded from
              below},
   JOURNAL = {J. \'{E}c. polytech. Math.},
  FJOURNAL = {Journal de l'\'{E}cole polytechnique. Math\'{e}matiques},
    VOLUME = {5},
      YEAR = {2018},
     PAGES = {613--650},
       DOI = {10.5802/jep.80},
}

@article {AmbrosioGigliMondinoRajala15,
    AUTHOR = {Ambrosio, L. and Gigli, N. and Mondino, A. and
              Rajala, T.},
     TITLE = {Riemannian {R}icci curvature lower bounds in metric measure
              spaces with {$\sigma$}-finite measure},
   JOURNAL = {Trans. Amer. Math. Soc.},
  FJOURNAL = {Transactions of the American Mathematical Society},
    VOLUME = {367},
      YEAR = {2015},
    NUMBER = {7},
     PAGES = {4661--4701},
       DOI = {10.1090/S0002-9947-2015-06111-X},
}

@article {Rajala12,
    AUTHOR = {Rajala, T.},
     TITLE = {Local {P}oincar\'{e} inequalities from stable curvature conditions on metric spaces},
   JOURNAL = {Calc. Var.},
  FJOURNAL = {Calc. Var. Partial Differential Equations},
    VOLUME = {44},
      YEAR = {2012},
    NUMBER = {},
     PAGES = {477--494},
       DOI = {},
}

@article {ChCo0,
    AUTHOR = {Cheeger, J. and Colding, T. H.},
     TITLE = {Lower bounds on {R}icci curvature and the almost rigidity of
              warped products},
   JOURNAL = {Ann. of Math. (2)},
  FJOURNAL = {Annals of Mathematics. Second Series},
    VOLUME = {144},
      YEAR = {1996},
    NUMBER = {1},
     PAGES = {189--237},
       DOI = {10.2307/2118589},
}

@article {ChCo1,
    AUTHOR = {Cheeger, J. and Colding, T. H.},
     TITLE = {On the structure of spaces with {R}icci curvature bounded
              below. {I}},
   JOURNAL = {J. Differential Geom.},
  FJOURNAL = {Journal of Differential Geometry},
    VOLUME = {46},
      YEAR = {1997},
    NUMBER = {3},
     PAGES = {406--480},
       URL = {http://projecteuclid.org/euclid.jdg/1214459974},
}

@article {ChCo2,
    AUTHOR = {Cheeger, J. and Colding, T. H.},
     TITLE = {On the structure of spaces with {R}icci curvature bounded
              below. {II}},
   JOURNAL = {J. Differential Geom.},
  FJOURNAL = {Journal of Differential Geometry},
    VOLUME = {54},
      YEAR = {2000},
    NUMBER = {1},
     PAGES = {13--35},
       URL = {http://projecteuclid.org/euclid.jdg/1214342145},
}

@article {Colding97,
    AUTHOR = {Colding, T. H.},
     TITLE = {Ricci curvature and volume convergence},
   JOURNAL = {Ann. of Math. (2)},
  FJOURNAL = {Annals of Mathematics. Second Series},
    VOLUME = {145},
      YEAR = {1997},
    NUMBER = {3},
     PAGES = {477--501},
       DOI = {10.2307/2951841},
}

@incollection {Ambrosio02,
    AUTHOR = {Ambrosio, L.},
     TITLE = {Fine properties of sets of finite perimeter in doubling metric
              measure spaces},
      NOTE = {Calculus of variations, nonsmooth analysis and related topics},
   JOURNAL = {Set-Valued Anal.},
  FJOURNAL = {Set-Valued Analysis. An International Journal Devoted to the
              Theory of Multifunctions and its Applications},
    VOLUME = {10},
      YEAR = {2002},
    NUMBER = {2-3},
     PAGES = {111--128},
       DOI = {10.1023/A:1016548402502},
}

@incollection {AmbrosioHonda17,
    AUTHOR = {Ambrosio, L. and Honda, S.},
     TITLE = {New stability results for sequences of metric measure spaces
              with uniform {R}icci bounds from below},
 BOOKTITLE = {Measure theory in non-smooth spaces},
    SERIES = {Partial Differ. Equ. Meas. Theory},
     PAGES = {1--51},
}

@article {AmbrosioDiMarino14,
    AUTHOR = {Ambrosio, L. and Di Marino, S.},
     TITLE = {Equivalent definitions of {$BV$} space and of total variation
              on metric measure spaces},
   JOURNAL = {J. Funct. Anal.},
  FJOURNAL = {Journal of Functional Analysis},
    VOLUME = {266},
      YEAR = {2014},
    NUMBER = {7},
     PAGES = {4150--4188},
       DOI = {10.1016/j.jfa.2014.02.002},
}

@article {Miranda03,
    AUTHOR = {Miranda, Jr., M.},
     TITLE = {Functions of bounded variation on ``good'' metric spaces},
   JOURNAL = {J. Math. Pures Appl. (9)},
  FJOURNAL = {Journal de Math\'{e}matiques Pures et Appliqu\'{e}es. Neuvi\`eme S\'{e}rie},
    VOLUME = {82},
      YEAR = {2003},
    NUMBER = {8},
     PAGES = {975--1004},
       DOI = {10.1016/S0021-7824(03)00036-9},
}

@article {ChCo3,
    AUTHOR = {Cheeger, J. and Colding, T. H.},
     TITLE = {On the structure of spaces with {R}icci curvature bounded
              below. {III}},
   JOURNAL = {J. Differential Geom.},
  FJOURNAL = {Journal of Differential Geometry},
    VOLUME = {54},
      YEAR = {2000},
    NUMBER = {1},
     PAGES = {37--74},
       URL = {http://projecteuclid.org/euclid.jdg/1214342146},
}

@article{CavallettiMilman16,
	title = {The globalization theorem for the {Curvature}-{Dimension} condition},
	volume = {226},
	doi = {10.1007/s00222-021-01040-6},
	number = {1},
	journal = {Inventiones mathematicae},
	author = {Cavalletti, F. and Milman, E.},
	year = {2021},
	pages = {1--137},
}

@article {AntonelliBrueFogagnoloPozzetta2021,
    AUTHOR = {Antonelli, G. and Bru\`e, E. and Fogagnolo, M. and
              Pozzetta, M.},
     TITLE = {On the existence of isoperimetric regions in manifolds with
              nonnegative {R}icci curvature and {E}uclidean volume growth},
   JOURNAL = {Calc. Var. Partial Differential Equations},
  FJOURNAL = {Calculus of Variations and Partial Differential Equations},
    VOLUME = {61},
      YEAR = {2022},
    NUMBER = {2},
     PAGES = {Paper No. 77, 40},
       DOI = {10.1007/s00526-022-02193-9},
       URL = {https://doi.org/10.1007/s00526-022-02193-9},
}

@article{AntonelliNardulliPozzetta,
	author = {Antonelli, G. and Nardulli, S. and Pozzetta, M.},
	title = {The isoperimetric problem via direct method in noncompact metric measure spaces with lower Ricci bounds},
	DOI= "10.1051/cocv/2022052",
	url= "https://doi.org/10.1051/cocv/2022052",
	journal = {ESAIM: COCV},
	year = 2022,
	volume = 28,
	pages = "57",
}

@misc{AntonelliPasqualettoPozzettaSemolaFIRSThalf,
      title={Sharp isoperimetric comparison on non collapsed spaces with lower {R}icci bounds}, 
      author={Antonelli, G. and Pasqualetto, E. and Pozzetta, M. and Semola, D.},
      year={2022},
      eprint={2201.04916},
      archivePrefix={arXiv},
}

@misc{Pozuelo,
      title={Existence of isoperimetric regions in sub-Finsler nilpotent groups}, 
      author={Pozuelo, J.},
      year={2021},
      eprint={2103.06630},
      archivePrefix={arXiv},
}

@misc {AntonelliPasqualettoPozzettaSemolaSECONDhalf,
    AUTHOR = {Antonelli, G. and Pasqualetto, E. and Pozzetta, M. and Semola, D.},
     TITLE = {Asymptotic isoperimetry on non collapsed spaces with lower {R}icci bounds},
      YEAR = {2022},
      eprint={2208.03739},
      archivePrefix={arXiv},
}

@article{BruePasqualettoSemola,
      title={Rectifiability of the reduced boundary for sets of finite perimeter over RCD$(K,N)$ spaces}, 
      author={Bru\`{e}, E. and Pasqualetto, E. and Semola, D.},
      year={2022},
      %eprint={1909.00381},
      %archivePrefix={arXiv},
      JOURNAL = {J. Eur. Math. Soc.},
      note = {published online first},
}

@misc{BruePasqualettoSemolaConstancy21,
      title={Constancy of the dimension in codimension one and locality of the unit normal on $\mathrm{RCD}(K,N)$ spaces}, 
      author={Bru\`{e}, E. and Pasqualetto, E. and Semola, D.},
      year={2021},
      eprint={2109.12585},
      archivePrefix={arXiv},
      note = {Accepted in: \emph{Ann. Sc. Norm. Super. Pisa Cl. Sci.}},
}

@article {Sturm1,
    AUTHOR = {Sturm, K.-T.},
     TITLE = {On the geometry of metric measure spaces. {I}},
   JOURNAL = {Acta Math.},
  FJOURNAL = {Acta Mathematica},
    VOLUME = {196},
      YEAR = {2006},
    NUMBER = {1},
     PAGES = {65--131},
       DOI = {10.1007/s11511-006-0002-8},
}

@article {GigliMondinoSavare15,
    AUTHOR = {Gigli, N. and Mondino, A. and Savar\'{e}, G.},
     TITLE = {Convergence of pointed non-compact metric measure spaces and
              stability of {R}icci curvature bounds and heat flows},
   JOURNAL = {Proc. Lond. Math. Soc. (3)},
  FJOURNAL = {Proceedings of the London Mathematical Society. Third Series},
    VOLUME = {111},
      YEAR = {2015},
    NUMBER = {5},
     PAGES = {1071--1129},
       DOI = {10.1112/plms/pdv047},
}

@book {BuragoBuragoIvanovBook,
    AUTHOR = {Burago, D. and Burago, Y. and Ivanov, S.},
     TITLE = {A course in metric geometry},
    SERIES = {Graduate Studies in Mathematics},
    VOLUME = {33},
 PUBLISHER = {American Mathematical Society, Providence, RI},
      YEAR = {2001},
     PAGES = {xiv+415},
       DOI = {10.1090/gsm/033},
       URL = {https://doi.org/10.1090/gsm/033},
}

@article {AgostinianiFogagnoloMazzieri,
    AUTHOR = {Agostiniani, V. and Fogagnolo, M. and Mazzieri,
              L.},
     TITLE = {Sharp geometric inequalities for closed hypersurfaces in
              manifolds with nonnegative {R}icci curvature},
   JOURNAL = {Invent. Math.},
  FJOURNAL = {Inventiones Mathematicae},
    VOLUME = {222},
      YEAR = {2020},
    NUMBER = {3},
     PAGES = {1033--1101},
       DOI = {10.1007/s00222-020-00985-4},
}

@article{MorganJohnson00,
    AUTHOR = {Morgan, F. and Johnson, D. L.},
     TITLE = {Some sharp isoperimetric theorems for {R}iemannian manifolds},
   JOURNAL = {Indiana Univ. Math. J.},
  FJOURNAL = {Indiana University Mathematics Journal},
    VOLUME = {49},
      YEAR = {2000},
    NUMBER = {3},
     PAGES = {1017--1041},
}

@article{MaheuxSaloffCoste95,
    AUTHOR = {Maheux, P. and Saloff-Coste, L.},
     TITLE = {Analyse sur les boules d'un op\'{e}rateur sous-elliptique},
   JOURNAL = {Math. Ann.},
  FJOURNAL = {Math. Ann.},
    VOLUME = {303},
      YEAR = {1995},
    NUMBER = {},
     PAGES = {713--740},
     DOI = {https://doi.org/10.1007/BF01461013},
}

@article {Lions84I,
    AUTHOR = {Lions, P.-L.},
     TITLE = {The concentration-compactness principle in the calculus of variations. {T}he locally compact case. {I}},
   JOURNAL = {Ann. Inst. H. Poincar\'{e} Anal. Non Lin\'{e}aire},
  FJOURNAL = {Annales de l'Institut Henri Poincar\'{e}. Analyse Non Lin\'{e}aire},
    VOLUME = {1},
      YEAR = {1984},
    NUMBER = {2},
     PAGES = {109--145},
}

@article{morgan2003regularity,
  title={Regularity of isoperimetric hypersurfaces in Riemannian manifolds},
  author={Morgan, F.},
  journal={Transactions of the American Mathematical Society},
  volume={355},
  number={12},
  pages={5041--5052},
  year={2003}
}

@book {chow-ricci,
    AUTHOR = {Chow, B. and Lu, P. and Ni, L.},
     TITLE = {Hamilton's {R}icci flow},
    SERIES = {Graduate Studies in Mathematics},
    VOLUME = {77},
 PUBLISHER = {American Mathematical Society},
   ADDRESS = {Providence, RI},
      YEAR = {2006},
     PAGES = {xxxvi+608},
   %MRCLASS = {53C44 (35K55 53C21 57M40 57M50)},
  %MRNUMBER = {2274812 (2008a:53068)},
%MRREVIEWER = {James McCoy},
}

@article {ChodoshEichmairVolkmann17,
    AUTHOR = {Chodosh, O. and Eichmair, M. and Volkmann, A.},
     TITLE = {Isoperimetric structure of asymptotically conical manifolds},
   JOURNAL = {J. Differential Geom.},
  FJOURNAL = {Journal of Differential Geometry},
    VOLUME = {105},
      YEAR = {2017},
    NUMBER = {1},
     PAGES = {1--19},
       URL = {http://projecteuclid.org/euclid.jdg/1483655857},
}

@book{docarmo,
  title={Riemannian Geometry},
  author={do Carmo, M.P.},
  series={Mathematics (Boston, Mass.)},
  %url={https://books.google.it/books?id=uXJQQgAACAAJ},
  year={1992},
  publisher={Birkh{\"a}user}
}

@article {FloresNardulli20,
    AUTHOR = {Mu\~{n}oz Flores, A. E. and Nardulli, S.},
     TITLE = {Local {H}\"{o}lder continuity of the isoperimetric profile in
              complete noncompact {R}iemannian manifolds with bounded
              geometry},
   JOURNAL = {Geom. Dedicata},
  FJOURNAL = {Geometriae Dedicata},
    VOLUME = {201},
      YEAR = {2019},
     PAGES = {1--12},
       DOI = {10.1007/s10711-018-0416-4},
}

@article {FloresNardulliCompactness,
    AUTHOR = {Mu\~{n}oz Flores, A. E. and Nardulli, S.},
     TITLE = {Generalized Compactness for Finite Perimeter Sets
and Applications to the Isoperimetric Problem},
   JOURNAL = {Journal of Dynamical and Control Systems},
  FJOURNAL = {Journal of Dynamical and Control Systems},
    VOLUME = {28},
      YEAR = {2022},
     PAGES = {59-69},
       URL = {https://doi.org/10.1007/s10883-020-09517-y},
}

@book{lafontaine,
  title={Riemannian Geometry},
  author={Gallot, S. and Hulin, D. and Lafontaine, J.},
  isbn={9783540204930},
  lccn={90022646},
  series={Universitext},
 % url={https://books.google.it/books?id=6F4Umpws\_gUC},
  year={2004},
  publisher={Springer Berlin Heidelberg}
}

@book {Heb00,
    AUTHOR = {Hebey, E.},
     TITLE = {Nonlinear analysis on manifolds: {S}obolev spaces and
              inequalities},
    SERIES = {Courant Lecture Notes in Mathematics},
    VOLUME = {5},
 PUBLISHER = {New York University, Courant Institute of Mathematical
              Sciences, New York; American Mathematical Society, Providence,
              RI},
      YEAR = {1999},
     PAGES = {x+309},
}

@article {GigliRCD,
    AUTHOR = {Gigli, N.},
     TITLE = {On the differential structure of metric measure spaces and
              applications},
   JOURNAL = {Mem. Amer. Math. Soc.},
  FJOURNAL = {Memoirs of the American Mathematical Society},
    VOLUME = {236},
      YEAR = {2015},
    NUMBER = {1113},
     PAGES = {vi+91},
       DOI = {10.1090/memo/1113},
}

@book {WeilIsoperimetricaSuperficiCH,
    AUTHOR = {Weil, A.},
     TITLE = {``Sur les surfaces \`{a} curbure n\'{e}gative'', in {O}euvres scientifiques/{C}ollected papers. {I}. 1926--1951},
    SERIES = {Springer Collected Works in Mathematics},
 PUBLISHER = {Springer, Heidelberg},
      YEAR = {1979},
     PAGES = {xviii+574},
}

@book {SchoenYauLectures,
    AUTHOR = {Schoen, R. and Yau, S.-T.},
     TITLE = {Lectures on differential geometry},
    SERIES = {Conference Proceedings and Lecture Notes in Geometry and
              Topology, I},
 PUBLISHER = {International Press, Cambridge, MA},
      YEAR = {1994},
     PAGES = {v+235},
}

@article {Kleiner92,
    AUTHOR = {Kleiner, B.},
     TITLE = {An isoperimetric comparison theorem},
   JOURNAL = {Invent. Math.},
  FJOURNAL = {Inventiones Mathematicae},
    VOLUME = {108},
      YEAR = {1992},
    NUMBER = {1},
     PAGES = {37--47},
       DOI = {10.1007/BF02100598},
}

@article {Croke84,
    AUTHOR = {Croke, C. B.},
     TITLE = {A sharp four-dimensional isoperimetric inequality},
   JOURNAL = {Comment. Math. Helv.},
  FJOURNAL = {Commentarii Mathematici Helvetici},
    VOLUME = {59},
      YEAR = {1984},
    NUMBER = {2},
     PAGES = {187--192},
       DOI = {10.1007/BF02566344},
}

@article {CrokeInequality80,
    AUTHOR = {Croke, C. B.},
     TITLE = {Some isoperimetric inequalities and eigenvalue estimates},
   JOURNAL = {Ann. Sci. \'{E}cole Norm. Sup. (4)},
  FJOURNAL = {Annales Scientifiques de l'\'{E}cole Normale Sup\'{e}rieure. Quatri\`eme
              S\'{e}rie},
    VOLUME = {13},
      YEAR = {1980},
    NUMBER = {4},
     PAGES = {419--435},
}

@article {GalliRitore,
    AUTHOR = {Galli, M. and Ritor\'{e}, M.},
     TITLE = {Existence of isoperimetric regions in contact sub-{R}iemannian
              manifolds},
   JOURNAL = {J. Math. Anal. Appl.},
  FJOURNAL = {Journal of Mathematical Analysis and Applications},
    VOLUME = {397},
      YEAR = {2013},
    NUMBER = {2},
     PAGES = {697--714},
}

@article {MirandaPallaraParonettoPreunkert07,
    AUTHOR = {Miranda Jr., M. and Pallara, D. and Paronetto, F. and  Preunkert, M.},
     TITLE = {Heat semigroup and functions of bounded variation on
              {R}iemannian manifolds},
   JOURNAL = {J. Reine Angew. Math.},
  FJOURNAL = {Journal f\"{u}r die Reine und Angewandte Mathematik. [Crelle's Journal]},
    VOLUME = {613},
      YEAR = {2007},
     PAGES = {99--119},
       DOI = {10.1515/CRELLE.2007.093},
}

@article {MondinoNardulli16,
    AUTHOR = {Mondino, A. and Nardulli, S.},
     TITLE = {Existence of isoperimetric regions in non-compact {R}iemannian
              manifolds under {R}icci or scalar curvature conditions},
   JOURNAL = {Comm. Anal. Geom.},
  FJOURNAL = {Communications in Analysis and Geometry},
    VOLUME = {24},
      YEAR = {2016},
    NUMBER = {1},
     PAGES = {115--138},
}

@article {MorganRitore02,
    AUTHOR = {Morgan, F. and Ritor\'{e}, M.},
     TITLE = {Isoperimetric regions in cones},
   JOURNAL = {Trans. Amer. Math. Soc.},
  FJOURNAL = {Transactions of the American Mathematical Society},
    VOLUME = {354},
      YEAR = {2002},
    NUMBER = {6},
     PAGES = {2327--2339},
       DOI = {10.1090/S0002-9947-02-02983-5},
}

@article {RitoreExistenceSurfaces01,
    AUTHOR = {Ritor\'{e}, M.},
     TITLE = {The isoperimetric problem in complete surfaces of nonnegative
              curvature},
   JOURNAL = {J. Geom. Anal.},
  FJOURNAL = {The Journal of Geometric Analysis},
    VOLUME = {11},
      YEAR = {2001},
    NUMBER = {3},
     PAGES = {509--517},
       DOI = {10.1007/BF02922017},
}

@incollection {PerelmanExampleCones,
    AUTHOR = {Perelman, G.},
     TITLE = {A complete {R}iemannian manifold of positive {R}icci curvature
              with {E}uclidean volume growth and nonunique asymptotic cone},
 BOOKTITLE = {Comparison geometry ({B}erkeley, {CA}, 1993--94)},
    SERIES = {Math. Sci. Res. Inst. Publ.},
    VOLUME = {30},
     PAGES = {165--166},
 PUBLISHER = {Cambridge Univ. Press, Cambridge},
      YEAR = {1997},
}

@article {Nar14,
    AUTHOR = {Nardulli, S.},
     TITLE = {Generalized existence of isoperimetric regions in non-compact
              {R}iemannian manifolds and applications to the isoperimetric
              profile},
   JOURNAL = {Asian J. Math.},
  FJOURNAL = {Asian Journal of Mathematics},
    VOLUME = {18},
      YEAR = {2014},
    NUMBER = {1},
     PAGES = {1--28},
}

@article{RitRosales04,
    AUTHOR = {Ritor\'{e}, M. and Rosales, C.},
     TITLE = {Existence and characterization of regions minimizing perimeter
              under a volume constraint inside {E}uclidean cones},
   JOURNAL = {Trans. Amer. Math. Soc.},
  FJOURNAL = {Transactions of the American Mathematical Society},
    VOLUME = {356},
      YEAR = {2004},
    NUMBER = {11},
     PAGES = {4601--4622},
}

@article {CheegerGromovTaylor,
    AUTHOR = {Cheeger, J. and Gromov, M. and Taylor, M.},
     TITLE = {Finite propagation speed, kernel estimates for functions of
              the {L}aplace operator, and the geometry of complete
              {R}iemannian manifolds},
   JOURNAL = {J. Differential Geometry},
  FJOURNAL = {Journal of Differential Geometry},
    VOLUME = {17},
      YEAR = {1982},
    NUMBER = {1},
     PAGES = {15--53},
}

@book{Petersen2016,
  title     = {Riemannian geometry. Third edition.},
  publisher = {Springer, Cham},
  year      = {2016},
  author    = {Petersen, P.},
  volume    = {171},
  series    = {Graduate Texts in Mathematics},
  address   = {},
  pages     = {xviii+499},
}

@book {MaggiBook,
    AUTHOR = {Maggi, F.},
     TITLE = {Sets of finite perimeter and geometric variational problems},
    SERIES = {Cambridge Studies in Advanced Mathematics},
    VOLUME = {135},
      NOTE = {An introduction to geometric measure theory},
 PUBLISHER = {Cambridge University Press, Cambridge},
      YEAR = {2012},
     PAGES = {xx+454},
       DOI = {10.1017/CBO9781139108133},
}

@book {PigolaRigoliSetti,
    AUTHOR = {Pigola, S. and Rigoli, M. and Setti, A. G.},
     TITLE = {Vanishing and finiteness results in geometric analysis},
    SERIES = {Progress in Mathematics},
    VOLUME = {266},
      NOTE = {A generalization of the Bochner technique},
 PUBLISHER = {Birkh\"{a}user Verlag, Basel},
      YEAR = {2008},
     PAGES = {xiv+282},
}

@article {Rit01NonExistence,
    AUTHOR = {Ritor\'{e}, M.},
     TITLE = {Constant geodesic curvature curves and isoperimetric domains
              in rotationally symmetric surfaces},
   JOURNAL = {Comm. Anal. Geom.},
  FJOURNAL = {Communications in Analysis and Geometry},
    VOLUME = {9},
      YEAR = {2001},
    NUMBER = {5},
     PAGES = {1093--1138},
       DOI = {10.4310/CAG.2001.v9.n5.a5},
}

@article {MorganPratelli,
    AUTHOR = {Morgan, F. and Pratelli, A.},
     TITLE = {Existence of isoperimetric regions in {$\Bbb{R}^n$} with density},
   JOURNAL = {Ann. Global Anal. Geom.},
  FJOURNAL = {Annals of Global Analysis and Geometry},
    VOLUME = {43},
      YEAR = {2013},
    NUMBER = {4},
     PAGES = {331--365},
       DOI = {10.1007/s10455-012-9348-7},
}

@article {AntonelliPasqualettoPozzetta21,
    AUTHOR = {Antonelli, G. and Pasqualetto, E. and Pozzetta,
              M.},
     TITLE = {Isoperimetric sets in spaces with lower bounds on the {R}icci
              curvature},
   JOURNAL = {Nonlinear Anal.},
  FJOURNAL = {Nonlinear Analysis. Theory, Methods \& Applications. An
              International Multidisciplinary Journal},
    VOLUME = {220},
      YEAR = {2022},
     PAGES = {Paper No. 112839, 59},
       DOI = {10.1016/j.na.2022.112839},
       URL = {https://doi.org/10.1016/j.na.2022.112839},
}

\typeout{get arXiv to do 4 passes: Label(s) may have changed. Rerun}
\end{document}